\documentclass{amsart}

\usepackage{amsmath}
\usepackage{amsfonts}
\usepackage{amssymb}
\usepackage{amsrefs}
\usepackage{latexsym}
\usepackage{amscd}

\usepackage{braket}
\usepackage{comment}
\usepackage{cleveref}

\usepackage[dvipdfmx]{graphicx} 
\usepackage[normalem]{ulem} 
\usepackage{color}

\newtheorem{thm}{Theorem}[section]
\newtheorem{lem}[thm]{Lemma}
\newtheorem{cor}[thm]{Corollary}
\newtheorem{prop}[thm]{Proposition}
\newtheorem{claim}[thm]{Claim}

\theoremstyle{definition}
\newtheorem{defn}[thm]{Definition}

\newtheorem{conj}[thm]{Conjecture}

\newtheorem{ques}[thm]{Question}

\theoremstyle{remark}
\newtheorem{rem}[thm]{Remark}

\newcommand{\A}{\mathcal A}
\newcommand{\B}{\mathcal B}
\newcommand{\G}{\mathcal G}
\newcommand{\su}{\mathfrak{su}(2)}

\def\ri{\rightarrow}
\newcommand{\pr}{\text{pr}}

\newcommand{\al}{\alpha}

\def\Om{\Omega}

\newcommand{\n}{\mathbb N}

\newcommand{\wt}{\widetilde}
\newcommand{\R}{\mathbb R}
\newcommand{\z}{\mathbb Z}
\def\q{{\mathbb Q}}
\newcommand{\F}{\mathcal{F}}
\newcommand{\co}{\mathbb C}

\newcommand{\calR}{\mathcal{R}}
\newcommand{\cs}{\mathit{cs}}

\newcommand{\Tr}{\operatorname{Tr}}

\renewcommand{\ker}{\operatorname{Ker}}
\newcommand{\im}{\operatorname{Im}}
\newcommand{\ad}{\operatorname{ad}}
\newcommand{\Ad}{\operatorname{Ad}}
\newcommand{\rank}{\operatorname{rank}}

\newcommand{\hol}{\operatorname{Hol}}
\newcommand{\aut}{\operatorname{Aut}}

\newcommand{\id}{\operatorname{Id}}
\newcommand{\mdeg}{\operatorname{mdeg}}

\def\inner<#1>{\langle #1 \rangle}

\def\tR{\widetilde{R}}

\newcommand{\ind}{\operatorname{ind}}

\newcommand{\grad}{\operatorname{grad}}
\newcommand{\End}{\mathrm{End}}
\def\map{\text{Map}}

\def\Af{\A^{\mathrm{flat}}}
\def\tmB{\widetilde{\mathcal{B}}} 

\newcommand{\SL}{\mathit{SL}}
\newcommand{\SU}{\mathit{SU}}
\newcommand{\Hom}{\operatorname{Hom}}
\newcommand{\irr}{\mathrm{irr}}
\newcommand{\Det}{\operatorname{Det}} 
\newcommand{\Span}{\operatorname{span}}

\makeatletter
\@namedef{subjclassname@2020}{\textup{2020} Mathematics Subject Classification}
\makeatother

\begin{document}

\title[Filtered instanton Floer homology]{Filtered instanton Floer homology and the homology cobordism group}

\author{Yuta Nozaki}
\address{Organization for the Strategic Coordination of Research and Intellectual Properties, Meiji University \\
4-21-1 Nakano, Nakano-ku, Tokyo, 164-8525 \\
Japan
\newline
Current address:
Graduate School of Advanced Science and Engineering, Hiroshima University \\
1-3-1 Kagamiyama, Higashi-Hiroshima City, Hiroshima, 739-8526 \\
Japan}
\email{nozakiy@hiroshima-u.ac.jp}

\author{Kouki Sato}
\address{Department of Mathematics,
Meijo University \\
1-501 Shiogamaguchi, Tempaku-ku, Nagoya, 468-8502 \\
Japan}
\email{satokou@meijo-u.ac.jp}

\author{Masaki Taniguchi}
\address{
RIKEN, the Institute of Physical and Chemical Research \\
2-1 Hirosawa, Wako, Saitama, 351-0198\\
Japan.}
\email{masaki.taniguchi@riken.jp}

\subjclass[2020]{Primary 57Q60, 81T13, Secondary 58J28, 57R58, 57K31}
\keywords{Homology cobordism group, instanton Floer homology, Chern-Simons functional, definite 4-manifold, knot concordance group}

\begin{abstract}
For any $s \in [-\infty, 0] $ and oriented homology 3-sphere $Y$,
we introduce a homology cobordism invariant $r_s(Y)\in (0,\infty]$. 
The values $\{r_s(Y)\}$ are included in the critical values of the $\SU(2)$-Chern-Simons functional of $Y$, and we show a negative definite cobordism inequality and a connected sum formula for $r_s$.
As applications, we obtain several new results on the homology cobordism group.
First, we give infinitely many homology 3-spheres which cannot bound any definite 4-manifold.
Next, we show that if the 1-surgery of $S^3$ along a knot has the Fr{\o}yshov invariant negative, then all positive $1/n$-surgeries along the knot are linearly independent in the homology cobordism group.
In another direction, we use $\{r_s\}$ to define a filtration on the homology cobordism group which is parametrized by $[0,\infty]$.
Moreover, we compute an approximate value of $r_s$ for the hyperbolic 3-manifold obtained by $1/2$-surgery along the mirror of the knot $5_2$.
\end{abstract}

\maketitle

\setcounter{tocdepth}{2}
\tableofcontents

\section{Introduction}
The study of the structure of the 3-dimensional homology cobordism group $\Theta^3_{\z}$ is one of the central topics in low-dimensional topology.
One of the motivations is a relation to the triangulation problem of topological manifolds.
In 1985, Galewski-Stern \cite{GS76} and Matumoto \cite{M78} proved that for any topological $n$-manifold $M$ with $n \geq 5$ admits a triangulation if and only if a certain cohomology class $\delta (M) \in H^5 ( M; \ker \mu)$ satisfies $\delta (M)=0$, where $\mu \colon \Theta^3_{\z} \to \z/2\z$ is the Rokhlin homomorphism.
Since there is no essential difference between PL and smooth categories for 3- and 4-manifolds, $\Theta^3_{\z}$ is isomorphic to its PL version.
On the other hand, the $n$-dimensional PL version of homology cobordism group is known to be trivial for $n\neq 3$ (\cite{K69}).
Also, Freedman's result \cite{Fr82} implies that the topological version of the 3-dimensional homology cobordism group is trivial.

Various gauge theories and Floer theories have been developed and used to improve the understanding of $\Theta^3_{\z}$.
In the 1980s, Donaldson~\cite{Do83} applied Yang-Mills gauge theory to 4-dimensional topology and proved the diagonalization theorem.
The diagonalization theorem and its extension due to Furuta \cite{Fu87} imply that the Poincar\'e sphere has infinite order in $\Theta^3_{\z}$.
Fintushel-Stern \cite{FS90} and Furuta \cite{Fu90} developed Yang-Mills gauge theory for orbifolds with cylindrical ends to prove that $\Theta^3_\z$ contains $\z^\infty$ as a subgroup.
On the other hand, Manolescu \cite{Man16} disproved the triangulation conjecture using Seiberg-Witten Floer theory.
Recently, Dai-Hom-Stoffregen-Truong \cite{DHST18} proved the existence of a $\z^\infty$-summand in $\Theta^3_\z$ using involutive Heegaard-Floer theory.
 
%

In this paper, we interpret the work in Yang-Mills gauge theory \cite{FS90, Fu90} in terms of instanton Floer homology, and introduce a family $\{r_s\}$ of real-valued homology cobordism invariants for any homology 3-sphere.

\subsection{The invariants $r_s$}
Let $Y$ be an oriented homology 3-sphere.
In \cite{Do02}, Donaldson defined an obstruction class $[\theta]$
(denoted by $D_1$ in \cite{Do02})
lying in the first instanton cohomology group of $Y$ such that $[\theta] \neq 0$ implies the non-existence of any negative definite 4-manifold with boundary $Y$.
On the other hand, Fintushel-Stern~\cite{FS92} defined filtered versions of the instanton cohomology group $\{I^\ast_{[r,r+1]}(Y)\}_{r \in \R}$ 
such that their filtrations are given by a perturbed Chern-Simons functional.
Here, one can see that Fintushel-Stern's method actually enables us
to define a cohomology group $I^\ast_{[s,r]}(Y)$ for an arbitrary interval $[s,r]$,
and the obstruction class $[\theta]$ is well-defined in
$I^1_{[s,r]}(Y)$ for any 
$r \in (0,\infty]$ ($:= \R_{>0} \cup \{\infty\}$) and 
$s \in [-\infty, 0]$ ($:= \R_{\leq 0} \cup \{- \infty\}$).
Therefore, it is natural to ask whether $[\theta]$ vanishes in $I^1_{[s,r]}(Y)$ for a given $Y$ and interval $[s,r]$.
In light of this observation, we define
\[
r_s(Y) := \sup \{ r \in (0,\infty] \mid [\theta]=0 \in I^1_{[s,r]}(Y) \}
\]
for any oriented homology 3-sphere $Y$ and $s \in [-\infty, 0]$.
A more precise definition of $r_s$ is stated in Definition~\ref{def:r_s}.
Such a quantitative construction in the Floer homology has appeared in several Floer theories including Hamiltonian Floer homology (\cite{FH94}, \cite{FHW94}) and embedded contact homology (\cite{Ha11}).

Our main theorem of this paper is stated as follows.

\begin{thm}\label{main theorem} 
The values $\{r_s(Y)\}_{s \in [-\infty, 0]}$ are homology cobordism invariants of $Y$.
Moreover, the invariants $\{r_s\}_{s \in [-\infty, 0]}$
satisfy the following properties: 
\begin{enumerate}
\item If there exists a negative definite cobordism $W$ with $\partial W= Y_1 \amalg -Y_2$, then the inequality
\[
r_s (Y_2) \leq r_s(Y_1) \]
holds for any $s \in [-\infty, 0]$.
Moreover, if $W$ is simply connected 
and $r_s(Y_1) < \infty$, then  the
strict inequality
\[
r_s (Y_2) <r_s(Y_1) \]
holds.
\item If $r_s (Y)< \infty $, then $r_s (Y)$ lies in  the critical values of the Chern-Simons functional of $Y$.
\item If $s_1 \leq s_2$,
then $r_{s_1}(Y) \geq r_{s_2}(Y)$.
\item The inequality 
\[
r_s(Y_1 \# Y_2 ) \geq \min \{ r_{s_1}(Y_1)+s_2 , r_{s_2}(Y_2) +s_1\} 
\]
 holds for any $s,s_1,s_2 \in (-\infty, 0] $ with $s=s_1+s_2$.
\end{enumerate}
\end{thm}

Recently, Daemi \cite{D18} introduced a family
$\{\Gamma_Y(i)\}_{i \in \z}$ of
real-valued homology cobordism invariants.
Since the $\Gamma_Y(i)$ are also defined by using
instanton Floer theory and satisfy the properties (1) and (2) 
in \Cref{main theorem}, it is natural to ask
whether the $\Gamma_Y(i)$ are related to our $r_s(Y)$.
Roughly speaking, our invariants $\{r_s(Y)\}_{s \in [-\infty, 0]}$
can be seen as a one-parameter family including $\Gamma_{-Y}(1)$.
Precisely, we prove the following equality.

\begin{thm} 
\label{compare}
For any oriented homology $3$-sphere $Y$, the equality
\[
r_{-\infty} (Y) = \Gamma_{-Y}(1)
\]
holds.
\end{thm}

As consequences of \Cref{compare} and results in \cite{D18},
we can understand a relationship between $r_s$ and the Fr{\o}yshov invariant $h \colon \Theta^3_{\z} \to \z$ (\cite{Fr02}),
and obtain infinitely many examples with non-trivial $r_s$.
Remark that $r_s(S^3)= \infty$ for any $s \in [-\infty, 0]$, and so we say that $r_s(Y)$ is non-trivial if $r_s(Y)< \infty$.

\begin{cor}
\label{Froyshov}
The inequality $r_{-\infty}(Y) < \infty$ holds if and only if $h(Y)<0$.
In particular, if $h(Y)<0$, then $r_s(Y)$ has a finite value for any 
$s \in [-\infty, 0]$.
\end{cor}

Let $\Sigma(a_1, a_2, \ldots, a_n)$ denote the Seifert homology 3-sphere corresponding to a tuple of pairwise coprime integers $(a_1, a_2, \ldots, a_n)$, and let $R(a_1, a_2 ,\ldots, a_n)$ be an odd integer introduced by Fintushel-Stern \cite{FS85}.

\begin{cor}
\label{general Seifert}
If  $R(a_1, a_2 ,\ldots, a_n)>0$, then for any $s \in [-\infty, 0]$,
the equalities
$$
r_s(-\Sigma(a_1,a_2, \ldots, a_n))= \frac{1}{4a_1a_2\cdots a_n}
$$
and
$$
r_s(\Sigma(a_1,a_2, \ldots, a_n))= \infty
$$
hold.
\end{cor}

For instance, it is known that $R(p,q,pqk-1)=1$ for any coprime $p,q>1$
and $k \in \z_{>0}$.
Here, one might ask whether $r_s$ is constant for any $Y$.
We show that the answer is negative.
Indeed, the connected sum formula for $r_0$ in \Cref{main theorem} and the above corollaries imply that any $Y_k := 2\Sigma(2,3,5)\#-\Sigma(2,3,6k+5)$ ($k \in \z_{>0}$) satisfies $r_0(Y_k)= \frac{1}{24(6k+5)} < \infty$, while $r_{-\infty}(Y_k)= \infty$ because of $h(Y_k)=1$.

\subsection{Topological applications}
\label{section 1.2}
Next, we show topological applications of $r_s$, 
which include several new results on the homology cobordism group $\Theta^3_\z$
and the knot concordance group $\mathcal{C}$.

\subsubsection{Homology $3$-spheres with no definite bounding}
We call a 4-manifold \emph{definite} if it is positive definite or negative definite.
It is well-known that the Fr{\o}yshov invariant (\cite{Fr10, Fr02}) and the Heegaard-Floer correction term (\cite{OS04}) are obstructions for the existence of one of positive and negative definite bounding.
However, there was no invariant which is an obstruction for the existence of both positive and negative definite boundings.
Our invariant $r_s(Y)$ is the first example of such an obstruction.
We have the following theorem.

\begin{thm}
\label{definite bounding} 
There exist infinitely many homology $3$-spheres $\{Y_k\}_{k=1}^{\infty}$ which cannot bound any definite $4$-manifold.
Moreover, we can take such $Y_k$ so that 
the $Y_k$ are linearly independent in $\Theta^3_\z$.
\end{thm}

Indeed, we can take 
$\{Y_k\}_{k=1}^{\infty} := \{2\Sigma(2,3,5) \# (-\Sigma(2,3,6k+5))\}_{k=1}^{\infty}$
as a concrete example for \Cref{definite bounding}.
We will show that $r_0(Y_k) < \infty$ and $r_0(-Y_k) < \infty$.
Here, we note that if a homology 3-sphere $Y$ is Seifert or obtained by a knot surgery, then $Y$ bounds a definite 4-manifold.
In addition, the existence of a definite bounding is invariant under homology cobordism.
Therefore, we have the following corollaries.

\begin{cor}
For any $k \in \z_{>0}$, the homology cobordism class
$[2\Sigma(2,3,5) \#( -\Sigma(2,3,6k+5))]$
does not contain any Seifert homology $3$-sphere.
\end{cor}

The existence of such a homology $3$-sphere was first proved by Stoffregen~\cite{S15}.
The method of the proof is $Pin(2)$-monopole Floer homology.
On the other hand, our proof is based on Yang-Mills instanton theory.

\begin{cor}
For any $k \in \z_{>0}$, 
no representative of $[2\Sigma(2,3,5) \# (-\Sigma(2,3,6k+5))]$
is obtained by a knot surgery.
\end{cor}

\subsubsection{Linear independence of $1/n$-surgeries}
\label{section 1.2.2}

In \cite{FS90} and \cite{Fu90}, Fintushel-Stern and Furuta proved that for any coprime integers $p,q>1$,
the Seifert homology 3-spheres $\{\Sigma(p,q,pqn-1)\}_{n =1}^{\infty}$ are linearly independent in $\Theta^3_{\z}$.
We note that $\Sigma(p,q,pqn-1) = -S^3_{1/n}(T_{p,q})$, where $T_{p,q}$ is the $(p,q)$-torus knot and $S^3_{1/n}(K)$ denotes the $1/n$-surgery along a knot $K$ in $S^3$.
From this viewpoint, we generalize the above results as follows.

\begin{thm} \label{knot surgery}
For any knot $K$ in $S^3$, if $h(S^3_1(K))<0$, then $\{S^3_{1/n}(K)\}_{n=1}^{\infty}$
are linearly independent in $\Theta^3_{\z}$.
\end{thm}

\Cref{knot surgery} gives a huge number of linearly independent families in 
$\Theta^3_{\z}$.
In fact, there exist infinitely many hyperbolic knots and satellite knots with
$h(S^3_1(K)) < 0$.
As hyperbolic examples, we can take the mirrors $K_k^*$ of the 2-bridge knots $K_k$ ($k \in \z_{>0}$) corresponding to the rational numbers $\frac{2}{4k-1}$. 
(These $K_k$ are often called \emph{twist knots}.
See Figure~\ref{K_k} in Section~\ref{section 5.2}.)

\begin{cor}
\label{K_k indep}
For any $k \in \z_{>0}$, the homology $3$-spheres $\{S^3_{1/n}(K_k^*)\}_{n=1}^{\infty}$ are linearly independent in $\Theta^3_{\z}$.
\end{cor}

As satellite examples, we can take the $(2,q)$-cable of any knot $K$ 
(denoted by $K_{2,q}$) with odd $q \geq 3$.

\begin{cor}
\label{cable indep}
For any knot $K$ in $S^3$ and odd integer $q \geq 3$,
the homology $3$-spheres
$\{S^3_{1/n}(K_{2,q})\}_{n=1}^{\infty}$ are linearly independent in $\Theta^3_{\z}$.
\end{cor}


\subsubsection{Linear independence of Whitehead doubles}
In this paper, we consider the subgroup $\mathcal{T}$ in the knot concordance group $\mathcal{C}$ generated by topologically slice knots.
The group $\mathcal{T}$ has been studied well via several gauge theories, Floer theories \cite{E95, OS03, MO07, OS08, OS11, HK12, KM13, Hom14I, HW16, OSS17, HM17,PJ17, DHML19, KM19, DS19, ASA20} and Khovanov homology theory \cite{Ra10} as in the case of $\Theta^3_\z$.
However, the structure of $\mathcal{T}$ is still mysterious.

Here we focus on the positively-clasped Whitehead double $D(K)$ of a knot $K$.
Since $D(K)$ has trivial Alexander polynomial, $D(K)$ lies in $\mathcal{T}$, namely $D(K)$ is topologically slice (\cite{Fr82}).
There is a famous conjecture about the Whitehead doubles stated as follows.

\begin{conj}[\text{\cite[Problem 1.38]{Ki97}}]
As elements of the knot concordance group $\mathcal{C}$,
the equality $[D(K)]=0$ holds 
if and only if $[K]=0$. 
\end{conj}

Motivated by this conjecture, Hedden-Kirk~\cite{HK12}
conjectured that the map 
$$
D \colon \mathcal{C} \to \mathcal{C}, \ [K] \mapsto [D(K)]
$$
preserves the linear independence, and they proved that the conjecture holds for
the family $\{T_{2,2^n-1}\}_{n =2}^{\infty}$, that is, the Whitehead doubles
$\{D(T_{2,2^n-1})\}_{n=2}^{\infty}$ are linearly independent in $\mathcal{C}$. 

We refine their result as follows.

\begin{thm}\label{np+q}
For any coprime integers $p,q>1$,
the Whitehead doubles $\{ D(T_{p,np+q})\}^{\infty}_{n=0}$ are linearly independent
in $\mathcal{C}$.
\end{thm}

\begin{cor}
The Whitehead doubles $\{D(T_{2,2n-1})\}^{\infty}_{n=2}$ are linearly independent
in $\mathcal{C}$.
\end{cor}

Note that Hedden-Kirk's results were extended to more general satellite knots in \cite{PJ17}, and our technique enables us to extend a result in \cite{PJ17}.
Moreover, our approach can be used to see the linear independence of $D(K)$ for a certain family of twisted knots $K$.

\subsection{Additional structures on $\Theta^3_\z$ and $\ker h$}

Using involutive Heegaard-Floer theory, Hendricks, Hom and Lidman~\cite{HHL21} introduced a poset filtration on $\Theta^3_\z$ and reproved the existence of a $\z^\infty $-subgroup of $\Theta^3_\z$.
Moreover, for the knot concordance group, such filtrations coming from Heegaard-Floer theory are also given in \cite{HHN13} and \cite{Sa18full}.
Inspired by these work, we give a $[0,\infty]$-filtration of $\Theta^3_\z$ using our invariant $r_s$,
which can be used to reprove that Fintushel-Stern's and Furuta's sequence $\{\Sigma(p,q, pqk -1 )\}_{k =1}^{\infty}$ is linearly independent in $\Theta^3_\z$ for any pair $(p,q)$ of coprime integers.
Since $r_s(Y)$ coincides with a critical value of the $SU(2)$-Chern-Simons functional of $Y$,
our filtration has a flavor of geometry.

More precisely, 
for any $r \in [0,\infty]$,
we consider the set
\[
\Theta^3_{\z, r} := \left\{ [Y] \in \Theta^3_\z \bigm| \min \{ r_0 (Y), r_0 (-Y )\} \geq r \right\}.
\]
Then it follows from the connected sum formula for $r_0$ that $\Theta^3_{\z, r}$ is a subgroup of $\Theta^3_{\z}$. 
Moreover, by definition, it is obvious that if $r \geq r'$, then 
$\Theta^3_{\z, r} \subset \Theta^3_{\z, r'}$.
In particular, $\Theta^3_{\z, 0} = \Theta^3_{\z}$.
For this filtration, we prove that any quotient group is infinitely generated.

\begin{thm}
\label{Theta_zr}
For any $r \in (0,\infty]$, 
the quotient group $\Theta^3_\z/\Theta^3_{\z, r}$ contains $\z^\infty$
as a subgroup.
\end{thm}

\begin{figure}[htbp]
\begin{center}
\includegraphics[scale= 1]{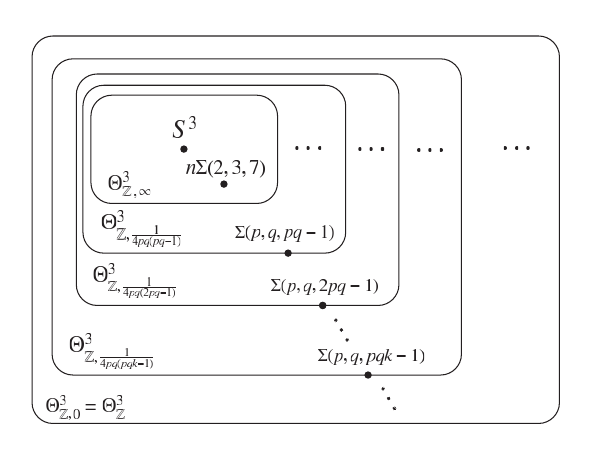}
\caption{A schematic picture of the filtration $\{\Theta^3_{\z,r}\}$.
\label{Fig filt}}
\end{center}
\end{figure}

Figure~\ref{Fig filt} gives a schematic picture of the filtration $\{\Theta^3_{\z, r}\}$.
Since $\Theta^3_{\z,r}$ is a subgroup for any $r \in [0,\infty]$, it is easy to see from the figure that $\{\Sigma(p,q,pqk-1)\}_{k=1}^\infty$ are linear independent in $\Theta^3_\z$. Here, we note that the smallest subgroup $\Theta^3_{\z,\infty}$ is infinitely generated.
In fact, it is proved by Hendricks-Hom-Lidman \cite{HHL21} that 
$\{S^3_{-1}(T_{2,4n+1}) \}_{n =1}^{\infty}$ are linearly independent in $\Theta^3_{\z}$. 
Moreover, it is not hard to see that $S^3_{-1}(T_{2,4n+1})$ bounds both positive definite 4-manifold and negative definite 4-manifold, and hence \Cref{main theorem}(1) gives $r_0(S^3_{-1}(T_{2,4n+1})) = r_0(-S^3_{-1}(T_{2,4n+1})) = \infty$.
Note that any positive knot bounds a null-homologous disk in $B^4 \# k \mathbb{C}P^2$ for sufficiently large $k$ (\cite{CL86}).
Therefore,
we have 
\[
\Theta^3_{\z, \infty}  \supset \Span_\z \{ [S^3_{-1}(T_{2,4n+1})] \}_{n =1}^{\infty}
\cong \z^{\infty},
\]
where $\Span_\z \{ [S^3_{-1}(T_{2,4n+1})] \}_{n =1}^{\infty}$
is the $\z$-linear span of $\{ [S^3_{-1}(T_{2,4n+1})] \}_{n =1}^{\infty}$ in $\Theta^3_\z$.
We post the following fundamental questions about 
the filtration
$\{\Theta^3_{\z,r}\}_{r \in [0,\infty]}$. 


\begin{ques}
Which subquotient $\Theta^3_{\z, r'}/\Theta^3_{\z, r}$ is infinitely generated?
\end{ques}

As another approach to studying $\Theta^3_{\z}$, we use the value 
\[
s_\infty(Y):= \sup \{ s \in [-\infty,0] \mid r_s(Y)= \infty \}
\]
to introduce a pseudometric on $\ker(h\colon \Theta^3_\z \to \z)$.
The pseudometric induces a metric on the quotient group $\ker h/\Theta^3_{\z, \infty}$.
For more details, see \Cref{metric}.

\subsection{Computations for a hyperbolic 3-manifold}\label{subsec:IntroHyperbolic}
Finally, we discuss the relation between our invariants $\{r_s\}$ and  geometric structures on homology 3-spheres.
While it is proved by Myers \cite{My83} that any homology cobordism class contains a hyperbolic representative, we also know that there exist infinitely many homology cobordism classes containing no Seifert representative, as discussed in \Cref{section 1.2}.
As the next step, it is natural to ask whether $\Theta^3_{\z}$ is generated by Seifert homology 3-spheres or not. Recently, Hendricks, Hom, Stoffregen, and Zemke~\cite{HHSZ20} proved that the homology cobordism class $[S^3_1(T^*_{6,7}\#T^*_{6,7}\#T_{6,13}\#T^*_{2,3:2,5})]$ is not contained in $\Theta^3_S$, where $T_{2,3:2,5}$ is the $(2,5)$-cable of $T_{2,3}$ and $\Theta^3_S$ denotes the subgroup of $\Theta^3_{\z}$ generated by Seifert homology 3-spheres.
Namely, they proved $\Theta^3_\z \supsetneq \Theta^3_S$.

Here, we mention that $S^3_1(T^*_{6,7}\#T^*_{6,7}\#T_{6,13}\#T^*_{2,3:2,5})$ is not Seifert but a graph manifold.
(The proof will be given in \Cref{sec:HHSZ20}.)
Hence, the following question still remains open.


\begin{ques}
\label{ques:graph}
Is the group $\Theta^3_\z$ generated by graph homology 3-spheres? 
\end{ques} 

Let $\Theta^3_G$ denote the subgroup of $\Theta^3_\z$ generated by all graph homology 3-spheres, and then \Cref{ques:graph} is equivalent to whether the equality 
$\Theta^3_{\z} = \Theta^3_G$ holds or not.
Here we note that critical values of the $\SU(2)$-Chern-Simons functional of graph 3-manifolds are rational (\cite{A94}), and hence
the image $r_s(\Theta^3_G)$ is included in $\q_{>0} \cup \{\infty\}$ for any $s \in [-\infty, 0]$.
Therefore, we have the following proposition.

\begin{prop}
\label{prop:graph}
If there exists a homology $3$-sphere $Y$ and $s \in [-\infty,0]$ such that $r_s(Y)$ is finite and irrational, then $[Y] \notin \Theta^3_G$.
\end{prop}

From the viewpoint of \Cref{prop:graph},
we try to calculate $\{r_s\}$ for the $1/2$-surgery along the knot $5_2^\ast$,
where $5_2^\ast$ is the mirror of the knot $5_2$ in Rolfsen's knot table. 
Note that $5_2$ is $K_2$ as a twist knot, 
$S^3_1(5_2^*) \cong -\Sigma(2,3,11)$ and that $S^3_{1/2}(5_2^\ast)$ is a hyperbolic 3-manifold (see \cite{MY01}).
These facts imply that the value $r_s(S^3_{1/2}(5_2^*))$ is finite and might be irrational. 
Moreover, using the computer, we have the following result.

\begin{thm}\label{approx}
The numerical approximation
 \[
 r_s(S^3_{1/2}(5_2^\ast)) \approx 
 0.0017648904\ 7864885113\ 0739625897\ 0947779330\ 4925308209 
 \]
holds for any $s \in [-\infty, 0]$,
where its error is at most $10^{-50}$.
\end{thm}

It is an open problem whether there exists a 3-manifold whose
$\SU(2)$-Chern-Simons functional has an irrational critical value.
Note that the decimal in \Cref{approx} has
no repetition. Therefore, we have the following conjecture.

\begin{conj}
\label{irrational}
The value $r_s(S^3_{1/2}(5_2^\ast))$ is an irrational number. 
\end{conj}

If \Cref{irrational} is true, then it follows from \Cref{prop:graph} that 
$[S^3_{1/2}(5_2^*)]$ is not contained in $\Theta^3_{G}$.



\subsection*{Organization}
The paper is organized as follows.
In Section~\ref{Review of filtered instanton Floer homology}, we give a review of filtered instanton homology.
In Section~\ref{The invariant $r_s$}, we introduce the invariants $r_s$ using notions of Section~\ref{Review of filtered instanton Floer homology}, and establish several basic properties of $r_s$.
In particular, \Cref{main theorem} will be proved in this section.
Section~\ref{Comparison with Daemi's invariants} is devoted to discussing the relation between $r_s$ and Daemi's $\Gamma_Y(k)$.
In Section~\ref{Applications}, we prove all assertions stated in Section~\ref{section 1.2}.
In Section~\ref{sec:Additional_structures}, we discuss additional structures on $\Theta^3_\z$ and $\ker h$ by using $r_s$.
In Section~\ref{Calculations}, we explain how to compute an approximate value of $r_s(S^3_{1/2} (5_2^\ast))$.

\subsection*{Acknowledgments}
The authors would like to express their deep gratitude to Aliakbar Daemi for discussing the invariants $\Gamma_Y(k)$.
They also would like to thank Jennifer Hom, Min Hoon Kim, JungHwan Park, and Yoshihiro Fukumoto for their useful comments.
The first author was supported by Iwanami Fujukai Foundation and the third author was supported by the Program for Leading Graduate Schools, MEXT, Japan and RIKEN iTHEMS
Program.
Also, this work was supported by JSPS KAKENHI Grant Numbers JP20K14317, JP18J00808, 17J04364.
Finally, the authors wish to express their thanks to the referee for many helpful suggestions improving the previous version.


\section{Review of filtered instanton Floer homology} \label{Review of filtered instanton Floer homology}
Throughout this paper, all manifolds are assumed to be smooth, compact, orientable and oriented, and diffeomorphisms are orientation-preserving unless otherwise stated.
In this section, we review the definition of the filtered instanton Floer homology. For the instanton Floer homology, see \cite{Fl88} and \cite{Do02}. For the filtered version of instanton Floer homology, see \cite{FS92}.
\subsection{Preliminaries}
\subsubsection{Chern-Simons functional}
For a homology 3-sphere $Y$, we denote the product $\SU(2)$ bundle by $P_Y$,
and the product connection on $P_Y$ by $\theta$.   
In addition, we denote
\begin{itemize}
\item the set of $\SU(2)$-connections on $P_Y$ by $\A(Y)$, 
\item the set of $\SU(2)$-flat connections on $P_Y$ by $\Af(Y)$,
\item $\tmB(Y):= \A(Y) /\map_0(Y, \SU(2))$, 
\item $\tR(Y):=\Af(Y) /\map_0(Y, \SU(2))$, and
\item $R(Y):= \Af(Y)/\map(Y, \SU(2))$,
\end{itemize}
where $\map(Y, \SU(2))$ (resp.\ $\map_0(Y,\SU(2))$) is the set of smooth functions
(resp.\ smooth functions of mapping degree $0$),
and the right action of $\map(Y,\SU(2))$ on $\A(Y)$ is given by 
$a\cdot g := g^{-1} dg + g^{-1} ag$. Note that the action preserves the flatness of $a$ for any $g$.
Also, we write $\tmB^*(Y)$, $\tR^*(Y)$ and $R^*(Y)$ respectively for the subsets of $\tmB(Y)$, $\tR(Y)$ and $R(Y)$ whose stabilizers are constants in $\{\pm I_2\}$.
The elements in $\tmB^*(Y)$ and $\tR^*(Y)$ are called \emph{irreducible connections}.
When the stabilizer of an $\SU(2)$-connection is larger than $\{\pm I_2\}$, it is said to be \emph{reducible}.
The \emph{Chern-Simons functional on $\A(Y)$} is the map $\cs_Y \colon \A(Y) \to \R$ defined by 
\[
\cs_Y(a):=\frac{1}{8\pi^2}\int_Y \Tr(a\wedge da +\frac{2}{3}a\wedge a\wedge a).
\]
It is known that $\cs_Y(a \cdot g) -\cs_Y(a)= \deg(g)$ holds for $g \in \map(Y,\SU(2))$, where $\deg(g)$ is the mapping degree of $g$.
Therefore, $\cs_Y(a \cdot g)= \cs_Y(a)$ for any $g \in \map_0(Y, \SU(2))$, 
and hence $\cs_Y$ descends to a map $\tmB(Y)\ri \R$.
We denote it by the same notation $\cs_Y$.  
Moreover, we use the notations $\Lambda_Y$ and $\Lambda^*_Y$ as $\cs_Y({\tR(Y)})$ and $\cs_Y({\tR^*(Y)})$, respectively.
Note that the set $\Lambda_Y$ is locally finite, that is, $[m,m+1]\cap \Lambda_Y$ is a finite set for any $m \in \R$.
For example, one can see $\Lambda_{S^3} = \z$ and $\Lambda^*_{S^3}= \emptyset$.
Set
\[
\R_Y := \R \setminus \Lambda_Y
\]
for any oriented homology $3$-sphere $Y$.

\subsubsection{Perturbations of $\cs_Y$}

Roughly speaking, the instanton Floer homology of $Y$ is
the Morse homology associated to $\cs_Y \colon \tmB^*(Y) \to \R$, where the set of critical points is $\tR^*(Y)$. However, $\tR^*(Y)$ does not satisfy non-degeneracy in general, and hence 
we need to perturb $\cs_Y$
so that $\tR^*(Y)$ becomes non-degenerate.
In this paper, we use several classes of perturbations of $\cs_Y$ introduced in \cite[Section~(1b)]{Fl88} and \cite[Section~5.5.1]{BD95}.

For any $d \in \z_{>0}$ and fixed $l \gg 2$, define the set of orientation-preserving embeddings of $d$ solid tori into $Y$; 
\[
\F_d:=  \left\{ (f_i \colon S^1\times D^2 \hookrightarrow Y )_{1\leq i \leq d} \right\},
\]
and denote by $C^{l}(\SU(2)^d,\R)_{\ad}$ 
the set of adjoint invariant $C^l$ functions on $\SU(2)^d$.
Then, the \emph{set of perturbations} is defined by
\[
\mathcal{P}(Y):= \bigcup_{d \in \n}\F_d\times C^{l}(\SU(2)^d,\R)_{\ad}.
\]

Fix a 2-form $d\mathcal{S}$ on $D^2$ supported in the interior of $D^2$ with $\int_{D^2}d\mathcal{S}=1$.
Then, for any $\pi = (f,h)\in \mathcal{P}(Y)$, 
we can define the \emph{$\pi$-perturbed Chern-Simons functional} $\cs_{Y,\pi} \colon\widetilde{\B}^*(Y) \ri \R$ by
\begin{align}\label{def of per}
\cs_{Y,\pi}(a)= \cs_Y(a)
+ \int_{x \in D^2} h(\hol_{f_1(-,x)}(a), \dots, \hol_{f_d(-,x)}(a)) d\mathcal{S},
\end{align}
where $\hol_{f_i(-,x)}(a)$ is the holonomy around the loop $t \mapsto f_i(t,x)$ for each $i \in \{1,\dots,d\}$. 
We denote $\|h\|_{C^l}$ by $\|\pi \|$ and the second term of the right-hand side in \eqref{def of per} by $ h_f$.

\subsubsection{Gradient of $\cs_{Y, \pi}$}\label{grad1}

We next consider the gradient of $\cs_{Y, \pi}$. 
Fix a Riemannian metric $g_Y$ on $Y$.
For $i \in \{1, \ldots, d\}$, let $\iota_i \colon \SU(2) \to \SU(2)^d$ denote the $i$-th inclusion, and set $h_i := h \circ \iota_i \colon \SU(2) \to \R $.
Then, identifying $\su$ with its dual by the Killing form, 
we can regard the derivative $h'_i$ as a map $h'_i \colon \SU(2) \to \su$. 

Using the value of the holonomy around each loop $\{f_i(s,x) \mid s\in S^1\}$, we obtain a section $\hol_{f_i(s,x)}(a)$ of the bundle $\aut P_Y$ over $\im f_i$. Sending the section $\hol_{f_i(s,x)}(a)$ by the bundle map induced by $h_i' \colon \aut P_Y \to \ad P_Y$, we obtain a section $h_i'(\hol_{f_i(s,x)}(a))$ of $\ad P_Y$ over $\im f_i$.

We now describe the gradient-line equation of $\cs_{Y,\pi}$ with respect to $L^2$-metric:
\begin{align}\label{grad}
 \frac{\partial}{\partial t} a_t=\grad_a \cs_{Y,\pi} = *_{g_Y}\left( F(a_t)+\sum_{1\leq i \leq d} h'_i(\hol(a_t)_{f_i(s,x)})(f_i)_*{\pr}_2^*d\mathcal{S} \right),
\end{align}
where $\pr_2$ is the second projection $\pr_2 \colon S^1\times D^2 \ri D^2$, $*_{g_Y}$ is the Hodge star operator and $F(a)$ denotes the curvature of a connection $a$.
We denote $\pr_2^*d\mathcal{S}$ by $\eta$.
We set
\[
\widetilde{R}(Y)_\pi:= \left\{a \in \widetilde{\B}(Y) \Biggm |F(a)+\sum_{1\leq i \leq d} h'_i(\hol(a)_{f_i(s,x)})(f_i)_*\eta=0 \right\}, 
 \]
 and 
 \[
 \widetilde{R}^*(Y)_\pi:= \widetilde{R}(Y)_\pi \cap \widetilde{\B}^*(Y).
 \]
The solutions of \eqref{grad} correspond to connections $A$ over $Y\times \R$ which satisfy an equation:
\begin{align}\label{pASD}
F^+(A)+ \pi(A)^+=0,
\end{align}
where
\begin{itemize}
\item the 2-form $\pi(A)$ is given by 
\[
\sum_{1\leq i \leq d} h'_i(\hol(A)_{\tilde{f}_i(t,x,s)}){(\tilde{f}_i)}_* (\pr_1^* \eta),
\]
\item the map $\pr_1$ is a projection map from $(S^1\times D^2) \times \R$ to $S^1\times D^2$,
\item the notation $+$ is $\frac{1}{2}(1+*)$ where $*$ is the Hodge star operator with respect to the product metric on $Y\times \R$ and 
\item the map $\tilde{f}_i\colon  S^1\times D^2\times \R \ri Y\times \R$ is $f_i\times \id$. 
\end{itemize}

We introduce the spaces of trajectories $M^Y(a,b)_\pi$ for given $a,b \in \widetilde{R}^*(Y)_\pi$. Fix a positive integer $q\geq3$. Let $A_{a,b}$ be an $\SU(2)$-connection on $Y \times \R$ satisfying $A_{a,b}|_{Y\times (-\infty,1]}=p^*a$ and $A_{a,b}|_{Y\times [1,\infty)}=p^*b$ where $p$ is the projection $Y\times \R \ri Y$.
We define $M^Y(a,b)_\pi$ by
\begin{align}\label{*}
M^Y(a,b)_\pi:=\left\{A_{a,b}+c  \Bigm| c \in \Omega^1(Y\times \R)\otimes \su_{L^2_q}\text{ with } \eqref{pASD} \right\}/ \G(a,b),
\end{align}
where the gauge group $\G(a,b)$ is given by
\begin{align*}
\G(a,b):=\left\{ g \in \aut(P_{Y\times \R})\subset {\End(\mathbb{C}^2)}_{L^2_{q+1,\text{loc}}} \Bigm| g^* A_{a,b} - A_{a,b} \in L^2_q \right\}. 
\end{align*}
Here the space $L^2_{q+1,\text{loc}}$ consists of the sections which are $L^2_{q+1}$ on each compact set in $Y\times \R$ and $g^* A_{a,b}$ denotes the pull-back of the connection $A_{a,b}$ by $g$.  
The group $\G(a,b)$ acts on $\left\{A_{a,b}+c \Bigm| c \in \Omega^1(Y\times \R)\otimes \su_{L^2_q}\text{ with }\eqref{pASD} \right\}$ via the pull-backs of connections. 
Since 
\[
 \|g^* A_{a,b} - A_{a,b}   \|_{L^2_q(Y \times [n,n+1]) } \to 0 \text{ as $n \to \pm \infty$}, 
 \]
 $g$ lies in the stabilizer $\{\pm1\}$ of $a$ and $b$ asymptotically on both ends respectively. When we define $M^Y(a,\theta)_{\pi,\delta}$, we use the $L^2_{q,\delta}$-norm in instead of $L^2_q$-norm. The definition of $L^2_{q,\delta}$-norm is given later in \eqref{weighted}. 
The space $\R$ has an action on $M^Y(a,b)_\pi$ by the translation.
\subsubsection{Classes of perturbations} 
\label{classes of perturbations}
We also use several classes of the perturbations.
If the cohomology groups defined by the complex given in \cite[(12)]{SaWe08} satisfies $H^i_{\pi,a}=0$ for all $[a] \in \widetilde{R}(Y)_\pi \setminus \{ [g^*\theta] \mid g \in \map (Y,\SU(2)) \}$ for a given $\pi$, we call $\pi$ a \emph{non-degenerate perturbation}. 
If $\pi$ satisfies the following conditions, we call $\pi$ a \emph{regular perturbation} for a fixed small number $\delta>0$ and $g_Y$. 
\begin{itemize}
\item The linearization of the left-hand side of \eqref{pASD}
\[
d^+_A+d\pi^+_A \colon \Om^1(Y\times \R)\otimes \su_{L^2_q}\ri \Om^+(Y\times \R)\otimes \su_{L^2_{q-1}}
\]
 is surjective for $a,b\in \wt{R}^*(Y)_\pi$ and $[A] \in M^Y(a,b)_\pi$. 
 \item The linearization of  the left-hand side of \eqref{pASD}
 \[
d^+_A+d\pi^+_A \colon  \Om^1(Y\times \R)\otimes \su_{L^2_{q,\delta}}\ri \Om^+(Y\times \R)\otimes \su_{L^2_{q-1,\delta}}
\]
is surjective for $a\in \wt{R}^*(Y)_\pi$ and $[A] \in M^Y(a,\theta)_\pi$. 
 \end{itemize}
  The norms are given by 
\[
 \|f\|^2_{L^2_q}:=\sum_{0\leq j \leq q} \int_{Y\times \R} |\nabla^j_{A_{a,b}}f|^2
\]
and
\begin{align}\label{weighted}
 \|f\|^2_{L^2_{q,\delta}}:=\sum_{0\leq j \leq q} \int_{Y\times \R}e^{\delta \tilde{\tau}} |\nabla^j_{A_{a}}f|^2
\end{align}
for $f \in \Om^i(Y\times \R) \otimes \su$ with compact support, where
 \begin{itemize}
 \item $A_{a,b}$ and $A_{a,\theta}$ are fixed connections as above, 
 \item $|-|$ is the product metric on $Y\times \R$, 
 \item the positive integer $q$ is grater than $2$ and 
 \item $\tilde{\tau}\colon Y\times \R \ri \R$ is a smooth function satisfying $\tilde{\tau}(y,t)=t$ if $t>1$ and $\tilde{\tau}(y,t)=0$ if $t<-1$.
 \end{itemize}
 Here the spaces  $M^Y(a,b)_\pi$ and $M^Y(a,\theta)_{\pi,\delta}$ are given in \eqref{*} in Section~\ref{grad1}. 
 
Next, we will introduce a class of small perturbations which we actually use. In order to explain this, we follow a method introduced in \cite{FS92}. 
\begin{defn}\label{epsilon perturbation} Let $Y$ be a homology 3-sphere and $g$ be a Riemannian metric on $Y$.
For $\epsilon>0$, we define a class of perturbations $\mathcal{P}^\epsilon(Y,g)$ as the subset of $\mathcal{P} (Y)$ consisting of elements $ \pi=(f,h)$ which satisfy 
\begin{enumerate} 
 \item \label{eps}$| h_f(a) | < \epsilon \text{ for all } a \in \widetilde{\B}(Y)$ and 
 \item \label{eps2}$\| \grad_g h_f(a) \|_{L^4}< \frac{\epsilon}{2}, \| \grad_g h_f(a) \|_{L^2}< \frac{\epsilon}{2}$  for all $a \in \widetilde{\B}(Y)$.
\end{enumerate}
\end{defn}
If necessary, for a non-degenerate regular perturbation $\pi=(f, h)$, we can assume $h$ is smooth (see \cite[Section~8]{SaWe08}).

For $r,s \in \R_{Y} \cup \{-\infty\}$ and a fixed Riemannian metric $g$, we define a class of perturbations $\mathcal{P}(Y,r,s,g)$ in the following way. 
Let $\{R_\alpha\}$ be the connected components of ${R}^*(Y)$. Let $U_\al$ be a neighborhood of $R_\al$ in $\B(Y)$ with respect to the $C^\infty$-topology such that $U_\al \cap U_\beta = \emptyset$ if $\al \neq \beta$ and $\{U_\al\}$ is a covering of $R^*(Y)$.  We take all lifts of $U_\al$ with respect to $\text{pr}:\widetilde{\B}_Y \to \B_Y$. Since $\map (Y,\SU(2)) / \map_0 (Y,\SU(2))$ is isomorphic to $\z$, we denote all lifts by $\{U_\al^i\}_{i \in \z}$.
 In addition, we impose the following conditions on $U_\al^i$.
\begin{itemize}
\item If $a \in U^i_\al$,  $|\cs(a)-\cs(R_\al)| <  \min\left\{ \frac{ d(r, \Lambda_Y)}{8}, \frac{ d(s, \Lambda_Y)}{8}\right\}$, where $d(r, \Lambda_Y)$ is given by
\[
d(r, \Lambda_Y) := \min\{| r - a | \in \R_{>0} \mid a \in \Lambda_Y\}.
\]
\item $U^i_\al$ has no reducible connections.
\end{itemize}
Note that, for any element $\rho \in \wt{R}(Y)$, we have unique $\al$ and $i\in \z$ such that $\rho \in U_\al^i$. 
 
By the Uhlenbeck compactness theorem, we can take a sufficiently small real number $\epsilon_1(Y,g, \{U_\al\})>0$ satisfying the following condition:
\begin{align}\label{nbd}
\text{ If }a \in \B^*(Y) \text{ and }\|F(a)\|_{L^2} \leq \epsilon_1(Y,g, \{U_\al\}) \text{, then } a \in U_\al \text{ for some $\al$}.
\end{align}

\begin{defn}Now we take the supremum value
\[
\epsilon_1(Y,g) := \frac{1}{2}\sup_{ \{U_\al\}}  \epsilon_1(Y,g, \{U_\al\}) ,
\]
where $\{U_\al\}$ runs over all coverings of $\{R_\al \}$ given as above method.
\end{defn}

We also use the notation  $\lambda_Y:= \min \{ |a-b| \mid a,b \in \Lambda_Y \text{ with }a\neq b  \}$. 
Then we define a class of perturbations which we will use later. 

\begin{defn}
For a given $r\in \R_{Y}$, $s \in [-\infty, \infty)$ and metric $g$,  
we define
\[
\epsilon(Y,r,s,g):= \begin{cases}  
\min \{ \epsilon_1(Y,g), \frac{ d(s, \Lambda_Y)}{8}, \frac{ d(r, \Lambda_Y)}{8},\frac{ \lambda_Y }{32} \} & \text{if $s \in \R_Y$,} \\ 
\min \{ \epsilon_1(Y,g), \frac{ d(r, \Lambda_Y)}{8},\frac{ \lambda_Y }{32} \} & \text{if $s \in \Lambda_Y$} \\
 \end{cases}
\]
and
\[
\mathcal{P}(Y,r,s,g):= 
\mathcal{P}^{\epsilon(Y,r,s,g)} (Y,g) \subset \mathcal{P}(Y). 
\]
\end{defn}

By the use of $\mathcal{P}(Y,r,s,g)$, we have the following fundamental properties of the values of the perturbed Chern-Simons functional.

\begin{lem}\label{keylemma}
Given $r \in \R_Y$, $s \in \R$, and $\pi \in \mathcal{P}(Y,r,s,g)$, for any $a \in \wt{R}_\pi(Y)$ one has
\begin{enumerate}
 \item $|\cs_\pi(a)-r| > \frac{3}{4}d(r,\Lambda_Y)$,
 \item $|\cs_\pi(a)-s| > \frac{3}{4}d(s,\Lambda_Y)$ if $s \in \R_Y$,
 \item $|\cs_\pi(a)-s+\frac{1}{2}\lambda_Y| > \frac{3}{4}d(s-\lambda_Y,\Lambda_Y)$ if $s \in \Lambda_Y$.
\end{enumerate}
\end{lem}

\begin{proof}
We only show (1).
Due to the choice of perturbations, for each $a \in \wt{R}_\pi(Y)$ one can find $\rho \in \wt{R}(Y)$ satisfying $a \in U_\rho$, which leads to
\begin{align*}
 |\cs_\pi(a)-r| &\geq |\cs(\rho)-r|-|\cs_\pi(a)-\cs(a)|-|\cs(a)-\cs(\rho)| \\
 &>  d(r,\Lambda_Y)-\frac{1}{8}d(r,\Lambda_Y)-\frac{1}{8}d(r,\Lambda_Y)
 = \frac{3}{4}d(r,\Lambda_Y).
\end{align*}
The proofs of (2) and (3) are essentially the same as that of (1).
\end{proof}


\subsection{Instanton Floer homology}
In this subsection, we give the definition of the filtration of the instanton Floer (co)homology by using the technique in \cite{FS92}. First, we give the definition of $\z$-graded instanton Floer homology. 
Let $Y$ be a homology $S^3$ and fix a Riemannian metric $g_Y$ on $Y$. Fix a non-degenerate regular perturbation $\pi \in \mathcal{P}(Y)$.
Roughly speaking, the instanton Floer homology is infinite dimensional Morse homology with respect to 
\begin{align}\label{R-valued}
\cs_{Y,\pi} \colon \widetilde{\B}^*(Y) \ri \R.
\end{align}

Floer defined $\ind\colon  \widetilde{R}^*(Y)_\pi \ri \z$, called the Floer index. 
The (co)chains of the instanton Floer homology are defined by
\[
CI_i(Y):= \z \left\{ a \in \widetilde{R}^*(Y)_\pi \Bigm| \ind(a)=i \right\} \ (CI^i(Y):= \Hom(CI_i(Y),\z)).
\]
The (co)boundary maps $\partial \colon CI_i(Y) \ri CI_{i-1}(Y)$ $(\delta \colon CI^i(Y)\ri CI^{i+1}(Y))$ are given by
\[
\partial (a) := \sum_{b \in \widetilde{R}^*(Y)_\pi,\, \ind(b)=i-1}\# (M^Y(a,b)_\pi/\R) b\ (\delta:=\partial^*),
\] 
where $M^Y(a,b)_\pi$ is the space of trajectories of $\cs_{Y,\pi} $ from $a$ to $b$.  
\begin{rem}
Originally, the instanton Floer homology is modeled on infinite dimensional Morse homology with respect to the functional 
 \begin{align}\label{S-valued}
\cs_{Y,\pi} \colon \B^*(Y) : =\A^*(Y) / \map (Y, \SU(2) )  \ri S^1. 
\end{align}
If we use $ \B(Y)$, the Floer indices take values in $\z/8\z$. For our purpose, we will use the $\R$-valued Chern-Simons functional, that is, we consider \eqref{R-valued} instead of \eqref{S-valued}. On $\widetilde{\B}^*(Y)$, the Floer indices take values in $\z$. So, we obtain a $\z$-graded chain complex. The original instanton chain group was given by 
\[
C_i(Y) :=  \z \left\{ a \in \widetilde{R}^*(Y)_\pi / \z  \Bigm| \ind(a)=i   \right\} 
\]
for each $i \in \z/8\z$.
The following isomorphism gives a relation between the original instanton Floer homology and $\z$-graded instanton Floer homology of $Y$: 
\[
H_\ast(CI_j (Y), \partial )  \xrightarrow{\cong} H_\ast(C_i(Y), \partial )   
\]
for $i\in \z/ 8\z$,  $j \in \z$ with $j \equiv i \mod 8$. 
\end{rem}

Here, we need to give orientations of $M^Y(a,b)_\pi/\R$.  We review how to give orientations of $M^Y(a,b)_\pi/\R$. For a given homology $3$-sphere $Y$, a non-degenerate perturbation $\pi$, $a \in \widetilde{R}^*_\pi (Y)$ and any oriented compact 4-manifold $X$ bounded by $-Y$, we define a configuration space
\begin{align}\label{configuration for X}
\B (a, X):= \{ A_{a} + c\ |\ c \in \Om^1(X^*) \otimes \su _{L^2_q}  \} /\G (a,X^*), 
\end{align}
where $X^\ast$ denotes $X \cup Y \times [0, \infty)$ with a product Riemann metric on the end and $\G (a,X^*)$ is the gauge group defined analogously to the cylindrical case as above.
Our convention of the orientations is the same as that in \cite[Section~5.4]{Do02}.
If we choose a connection $a$ as the product connection $\theta$, we need to use the weighted Sobolev norm \eqref{weighted} to define the space $\B (\theta, X)$.  
When we give orientations of the spaces $M^Y(a,b)_\pi/\R$, we use real line bundles 
\begin{align}\label{detline}
\mathbb{L}_a \to \B (a, X)
\end{align}
for $a \in \widetilde{R}^*(Y)_\pi $ which is called the \emph{determinant line bundle}.
This bundle is defined as a determinant line bundle of a family of operators $d_A^*+d_A^+$ parametrized by $A \in \B (a, X)$. For the details, see \cite[Section~5.4]{Do02}. 
In \cite[Section~5.4]{Do02}, it is shown that the bundle $\mathbb{L}_a \to \B (a, X)$ is trivial.
Define
\[
\mathbb{L}_X := \bigwedge^{\text{max}} (H^0(X; \R ) \oplus  H^1(X; \R ) \oplus H^+(X; \R )).
\]
If we fix an orientation of the orientation bundle 
\begin{align}\label{lambda ax} 
\lambda_{a, X}:=\mathbb{L}_a \otimes \mathbb{L}_X
\end{align} associated with $a$, one can define an orientation of $M^Y(a,b)_\pi$ by the following way.
By the gluing argument of the connections and the operators, one can obtain a continuous map
\[
\mathfrak{gl} \colon M^Y(a,b)_\pi \times \B(b,X) \to \B(a,X).
\]
Furthermore, we obtain a bundle isomorphism whose restriction to the fiber over $(A,B) \in M^Y(a,b)_\pi \times \B(b,X)$ is given by
\[
\widetilde{\mathfrak{gl} } \colon (\Det( TM^Y(a,b)_\pi) \otimes \mathbb{L}_b)|_{(A,B)}  \to  \mathfrak{gl}^* \mathbb{L}_a|_{(A,B)},
\]
where $TM^Y(a,b)_\pi$ is the tangent bundle of $M^Y(a,b)_\pi$.
We have two orientations of $M^Y(a,b)_\pi$: an orientation of $\Det( TM^Y(a,b)_\pi)$ so that $\widetilde{\mathfrak{gl}}$ is an orientation-preserving map and one coming from the $\R$-translation.
The consistency of these orientations gives a sign of the differential. 
This definition does not depend on the choice of $(A, B)$ and bump functions which are used to construct the map $\mathfrak{gl}$. Moreover, one can see that $\partial^2=0$ holds as in the case of Morse homology for finite dimensional manifolds.

The instanton Floer (co)homology $I_*(Y)$ (resp.\ $I^*(Y)$) is defined by 
\[
I_*(Y):= \ker \partial  / \im \partial \ (\text{resp.\ $I^*(Y):= \ker \delta/\im \delta$}).
\]
If we take another data of perturbations, Riemannian metric and orientations of $\lambda_{a, X}$, then the corresponding  chain complexes are chain homotopy equivalent to each other. 
Therefore the isomorphism class of the groups $CI_*(Y)$ and $CI^*(Y)$ are well-defined. 

\subsection{Filtered instanton Floer homology}\label{filter}
In this section, we introduce filtered instanton Floer homology which refines Fintushel-Stern's Floer homology introduced in \cite{FS92}.

We recall $\Lambda_Y = \cs_Y(\widetilde{R}(Y))$, $\Lambda^*_Y = \cs_Y(\widetilde{R}^*(Y))$ and $\R_Y = \R \setminus \Lambda_Y$.
For $r \in \R_Y$, we define the filtered instanton homology $I^{[s,r]}_*(Y) (I^*_{[s,r]}(Y))$ using $\epsilon$-perturbations.

\begin{defn}\label{defofIr*}
We fix $s \in [-\infty, 0]$.
For a given $r \in \R_Y$, metric $g$ on $Y$, a non-degenerate regular perturbation $\pi \in \mathcal{P}(Y,r,s,g)$ and orientations on line bundles $\lambda_{a, X}$,
the (co)chains of the filtered instanton Floer (co)homologies are defined by
\[ 
CI^{[s,r]}_i(Y, \pi):=\begin{cases} 
 \z \left\{ [a] \in \widetilde{R}^*(Y)_\pi \Bigm| \ind(a)=i,\  s<\cs_{Y,\pi}(a)<r \right\} & \text{if $s \in \R_{Y}$,} \vspace{0.5ex}\\
 \z \left\{ [a] \in \widetilde{R}^*(Y)_\pi \Bigm| \ind(a)=i,\  s- \frac{\lambda_Y}{2} <\cs_{Y,\pi}(a)<r \right\} & \text{if $s \in \Lambda_{Y}$} \\
\end{cases} 
\]
and
\[
CI^i_{[s,r]}(Y, \pi):= \Hom(CI_i^{[s,r]}(Y, \pi),\z), 
\]
 where $\lambda_Y:= \min \{ |a-b| \mid a\neq b , a,b \in \Lambda_Y\}$. 
The (co)boundary maps 
\[
\partial^{[s,r]} \colon CI_i^{[s,r]}(Y, \pi) \ri CI_{i-1}^{[s,r]}(Y, \pi) \ (\text{resp.\ $\delta^r \colon CI^i_{[s,r]}(Y)\ri CI^{i+1}_{[s,r]} (Y)$})
\]
are given by the restriction of $\partial$ to $CI_i^{[s,r]}(Y)$ (resp.\ $\delta^{[s,r]}:=(\partial^{[s,r]})^*$). 
\end{defn}

Then, one can see $(\partial^{[s,r]})^2=0$.

\begin{defn}
The \emph{filtered instanton Floer \textup{(}co\textup{)}homology} $I^{[s,r]}_*(Y)$ (resp.\ $I_{[s,r]}^*(Y)$) is defined by 
\begin{align*}
I^{[s,r]}_*(Y):= \ker \partial^{[s,r]}  / \im \partial^{[s,r]} \ (\text{resp.\ $I^\ast_{[s,r]}(Y):= \ker \delta^{[s,r]} / \im \delta^{[s,r]}$}).
\end{align*}
\end{defn}

Although the isomorphism classes of $CI^{[s,r]}_i(Y, \pi)$ depends on the choice of $\pi$, the chain homotopy type is an invariant of $Y$.
Thus, we omit the notion $\pi$ for the Floer (co)homology groups.
The following lemma provides well-definedness of our invariants $I^{[s,r]}_*(Y)$ and $I^*_{[s,r]}(Y)$.

\begin{lem}
Fix $s \in [-\infty,0]$, $r \in \R_Y$ with $ s\leq 0 \leq r$, two Riemannian metric $g$ and $g'$ on $Y$, non-degenerate regular perturbations $\pi$, $\pi'$ in $\mathcal{P}(Y,r,s,g)$ and orientations of orientation bundles for $\widetilde{R}^*(Y)_\pi$ and $\widetilde{R}^*(Y)_{\pi'}$ respectively.
If we choose two elements $\pi$ and $\pi'$ in $\mathcal{P}(Y,r,s,g)$ and $\mathcal{P}(Y,r,s,g')$, then there exists a chain homotopy equivalence between $CI^{[s,r]}_i(Y,\pi)$ and $CI^{[s,r]}_i(Y,\pi')$, where $CI^{[s,r]}_i(Y,\pi)$ \textup{(}resp.\ $CI^{[s,r]}_i(Y,\pi')$\textup{)} is the instanton chain complex with respect to $\pi$ \textup{(}resp.\ $\pi'$\textup{)}.
\end{lem}

\begin{proof}
Fix the following data: 
\begin{itemize}
\item Fix a Riemannian metric $g_\#$ on $Y \times \R$ which coincides with $g+ dt^2$ on $Y \times (-\infty, -1]$ and $g'+ dt^2$ on $Y \times [1,\infty)$.
\item Fix a regular perturbation $\pi_\#$ on $Y \times \R$ which coincides with $\pi$ on $Y \times (-\infty, -1]$ and $\pi'$ on $Y \times [1,\infty)$ such that 
\[
\| \pi_\# (A) \|_{L^2 (Y\times [-1,1])} < \min \{\epsilon(Y, r,s,g), \epsilon (Y, r,s,g')\}.
\]
(In \Cref{classes of perturbations}, we gave the definition of regular perturbations for the product perturbations.
In general case, we also define regular perturbations using a certain surjectivity condition.
For more details, see \cite{Do02}.)
\end{itemize}
Then, by considering the moduli space $M^{Y} (a,b)_{\pi_\#}$ under the assumption $\ind (a)-\ind(b)=0$, we obtain  compact 0-dimensional manifolds.
The orientation comes from in the same as the case of cobordism maps which we will introduce in \Cref{Cobordism maps}. 
Thus, we have a map 
\[
\mu_{\pi,\pi'} \colon CI^{[s,r]}_i(Y,\pi) \to CI^{[s,r]}_i(Y,\pi')
\]
defined by 
\[
\displaystyle a \mapsto \sum_{b:\,\ind (a)-\ind(b)=0} \#M^Y(a,b)_{\pi_\#}.
\]
Similarly, we have a map $\mu_{\pi',\pi} \colon CI^{[s,r]}_i(Y,\pi') \to CI^{[s,r]}_i(Y,\pi)$. One can check that the maps $\mu_{\pi,\pi'}$  and $\mu_{\pi',\pi}$ are chain maps.
Moreover, $\mu_{\pi',\pi} \mu_{\pi,\pi'}$ and $\mu_{\pi,\pi'} \mu_{\pi',\pi}$  are chain homotopic to the identity.
This is the same argument as in \cite{FS92}.
This completes the proof. 
\end{proof}

\begin{lem}\label{filtration} 
For $r,r' \in \R_Y$, $s, s' \in \R$ with $s\leq s' \leq 0 \leq r\leq r'$ , then there exists a chain map 
\[
i_{[s,r]}^{[s',r']} \colon  CI^{[s,r]}_i(Y) \to CI^{[s',r']}_i(Y).
\]
The map $i_{[s,r]}^{[s',r']} $ satisfies the following conditions: 
\begin{enumerate} 
\item The chain homotopy class of the map $i_{[s,r]}^{[s',r']} $ does not depend on the choice of additional data.
\item If we take three positive numbers $(r, r', r'')$, $(s, s',s'')$ with $s\leq s' \leq s'' \leq 0 \leq r \leq r'\leq r''$, then 
\[
i^{[s,r]}_{[s',r']}\circ i^{[s',r']}_{[s'',r'']}= i^{[s,r]}_{[s'',r'']}
\]
 holds as induced maps on cohomologies, where $i^{[s,r]}_{[s',r']}$, $i^{[s',r']}_{[s'',r'']}$ and $i^{[s,r]}_{[s'',r'']}$ are duals of $i_{[s,r]}^{[s',r']}$, $i_{[s',r']}^{[s'',r'']}$ and $i_{[s,r]}^{[s'',r'']}$.
\item If $[r, r'], [s,s']  \subset  \R \setminus \Lambda^*_Y$ , then the map  $i_{[s,r]}^{[s',r']} $ gives a chain homotopy equivalence. 
\end{enumerate}
\end{lem} 

\begin{proof}
First, we see the construction of $i_{[s,r]}^{[s',r']} $.
One can take a non-degenerate regular perturbation $\pi$ satisfying $\pi \in \mathcal{P} (Y,r,s,g) \cap \mathcal{P} (Y,r',s',g)$.
This give a natural map $i_{[s,r]}^{[s',r']} \colon CI^{[s,r]}_i(Y,\pi) \to CI^{[s',r']}_i(Y,\pi)$ by considering 
\[
i_{[s,r]}^{[s',r']}(a) :=  
\begin{cases} 
 a & \text{if $a \in CI^{[s',r']}_i(Y,\pi)$,}\\ 
 0 & \text{otherwise.} \\
\end{cases} 
\]
This gives a chain map.
The proof of (1) is similar to the proof of independence of choices of $\pi$.
The proof of (2) is obvious, because we can take a non-degenerate regular perturbation 
\[
 \pi \in  \mathcal{P} (Y,r,s,g) \cap  \mathcal{P} (Y,r',s' ,g) \cap  \mathcal{P} (Y,r'',s'',g).
\]
Here we give the proof of (3).
Suppose that $[r, r'], [s,s']  \subset  \R \setminus \Lambda^*_Y$.
We take a sequence of non-degenerate regular perturbations $\{\pi_n\} \subset \mathcal{P} (Y,r,s,g) \cap \mathcal{P} (Y,r',s',g)$ such that $\|\pi_n\| \to 0$.
We show that the following maps are bijective for sufficiently large $n$:
\[
i_n \colon \left\{ a\in \widetilde{R}(Y)_{\pi_n} \Bigm|  s<\cs_{\pi_n}(a)<r \right\} \to \left\{ a\in \widetilde{R}(Y)_{\pi_n} \Bigm|  s'<\cs_{\pi_n}(a)<r' \right\}.
\]

Suppose there is a sequence $\{ n_k\}$ of positive integers such that $n_k \to \infty$ as $k \to \infty$ and $i_{n_k}$ is not bijective for any $k \in \z_{>0}$. Then we can take a sequence $\{b_k\}$ with $s <\cs_{\pi_k}(b_k)<s'$ and  $r<\cs_{\pi_k}(b_k)<r'$. 
Using Uhlenbeck's compactness theorem, one obtain the bound of norm $\|g_k^*b_k\|_{L^2_k(Y)} \leq C_k$ for some gauge transformations.  By taking a subsequence, we can take a limit connection $b_\infty$.
Since the reducible connection is isolated for a small perturbation $\pi_{n_k}$, we can assume that $b_\infty$ is irreducible.
Since $\|\pi_{n_k}\| \to 0$, $b_\infty$ satisfies $F(b_\infty)=0$ and  $\cs(b_\infty) \in [s,s']\cup [r,r']$.
 This gives the contradiction. 
Therefore, $i_k$ is bijection for sufficiently large $k$. This completes the proof.
\end{proof}

\subsection{Cobordism maps}\label{Cobordism maps}
First, let us fix the convention about oriented cobordisms.
We use the \emph{outward normal first} convention.
For example,
\[
\partial ( Y \times [0,1] ) \cong -  \partial ( [0,1]\times Y ) = \{ 0 \} \times Y \amalg ( - \{ 1 \} \times Y  )
\]
for an oriented 3-manifold $Y$, where the symbol $\cong$ denotes the orientation-preserving diffeomorphism.
In this section, we review the cobordism maps for filtered instanton chain complexes. These maps are already considered in \cite{FS92}.
We fix $s_j \in [-\infty,0]$ for $ 1 \leq j \leq m$ and put $s= \sum_{1 \leq j \leq m} s_j$. Let  $Y^-$ be the finite disjoint union of oriented homology 3-spheres  $Y^-_{j}$ for $ 1 \leq j \leq m$, $Y^+$ an oriented homology sphere and $W$ a negative definite connected cobordism with $\partial W=  Y^+ \amalg (-Y^-)$ and $b_1(W)=0$.  
We assume $s_j \in \R_{Y^-_j}$ for $ 1 \leq j \leq m$.
First, we fix the following data related to $\partial W$: 
\begin{itemize}
\item a Riemannian metric $g$ on $\partial W = Y^+ \amalg -Y^-$ and 
\[
r \in \R_{Y^+ } \cap  (\bigcap_{ \substack{ 1 \leq j \leq m  }}  \R_{Y^- _{j}}).
\]
\item non-degenerate regular perturbations $\pi^+ \in \mathcal{P}(Y^+,r, s, g_{Y^+})$ and $\pi^-_{j} \in \mathcal{P}(Y^-_{j},r, r-s+s_j , g_{Y^-_{j}})$ for $ 1 \leq j \leq m$.
\item For any $a \in R^*(Y^+)_{\pi^+}$ and $b_{j} \in R^*(Y^-_{j})_{\pi^-_{j}}$,  orientations of $\mathbb{L}_{a}$ and $\mathbb{L}_{b_j}$ and $ 1 \leq j \leq m$.
\end{itemize}
Using the above data, one can define filtered Floer chain complexes $(C^{[s,r]}_*(Y^+), \partial^r) $ and $(C^{[s,r]}_*(Y^-_j), \partial^r) $.
Let us denote by $W^*$ the end-cylindrical $4$-manifold given by 
\[
Y^+ \times \R_{\leq 0} \cup W \cup Y^- \times \R_{\geq 0}.
\]
We fix an orientation of $W^*$ which agrees the orientations on $Y^+ \times \R_{\leq 0}$  and $Y^- \times \R_{\geq 0}$ and a Riemannian metric $g_{W^*}$ on $W^*$ which coincides with the product metric of $g$ and the standard metric of $\R$ on $Y^+ \times \R_{\leq 0} \amalg Y^- \times \R_{\geq 0}$.  
  For $\frak{a} \in  R(Y^+)_{\pi^+}$ and $\frak{b}= (b_j ) \in \prod_{  1\leq j \leq m}  R(Y^-_{j})_{\pi^-_{j}}$,  we can define the ASD moduli space ${M} (\frak{a},W^* ,\frak{b})$. The moduli space ${M} (\frak{a},W^* ,\frak{b})$ is given by 
\[
{M} (\frak{a},W^* ,\frak{b}):= \left\{ A_{\frak{a},\frak{b}} + c \Bigm| c \in \Om^1(W^*) \otimes \su _{L^2_q} \ ,  *  \right\} \Bigm/ \G (\frak{a},W^*, \frak{b}),
\]
where the condition $*$ is given by 
\begin{align*}\label{cob}
F^+( A_{\frak{a},\frak{b}} + c)+ \pi_W^+ ( A_{\frak{a},\frak{b}} + c)=0, 
\end{align*}
$A_{\frak{a},\frak{b}}$ is a $\SU(2)$-connection on $W^*$ whose restriction on the ends $Y^+ \times \R_{\leq -1} \cup Y^- \times \R_{\geq 1}$ coincide with the pull-backs of $\frak{a}$ and $\frak{b}$ and the group $\G (\frak{a},W^*, \frak{b})$ is given by similar way as in the product case. 
If we take a limit connection as $\theta$, we use the weighted norm with a small positive weight as in the case of $Y\times \R$. 
The part $\pi_W$  is a perturbation on $W^*$ satisfying the following conditions $(\ast \ast)$: 
\begin{itemize} 
 \item The perturbation $\pi_W$ coincides with $\pi^-_j$ on $Y^-_j\times \R_{\geq 0}$  for any $j$ and $\pi^+$ on $Y^+ \times \R_{\leq 0}$.
 \item For $a \in \Om^1(W)_{L^2_q}$, 
\[
\|\pi_W^+(a)\|_{L^2} 
< \frac{1}{8} \min_{j} 
\left\{ 
\begin{gathered}
d(r,
\Lambda_{Y^-_j}), 
\lambda_{Y^-_j},d(r-s+ s_j ,
\Lambda_{Y^-_j}) 
),\\
d(r,
\Lambda_{Y^+} 
), \lambda_{Y^+}, d(s ,
\Lambda_{Y^+} 
), d(s_j, \Lambda_{Y^-_j}) 
\end{gathered}
\right\}.
\]
 \item For any irreducible element $A \in {M} (\frak{a},W^* ,\frak{b})$, 
\[
d_A^+ + d(\pi_W^+)_A \colon \Omega^1 (W^*)\otimes \su_{L^2_q} \to  \Omega^+ (W^*)\otimes \su_{L^2_{q-1}}
\]
is surjective, where $d(\pi_W^+)_A$ is the linearization of $\pi_W^+ $. If $\frak{a}$ or $\frak{b}$ contains the reducible connection $\theta$, we need to consider the weighted norm as in the case of $Y\times \R$. 
\end{itemize} 

Now we explain how to give an orientation of ${M} (\frak{a},W^* ,\frak{b})$.  Let $X^-_j$ and $X^+$ be compact oriented 4-manifolds with $\partial (X^-_{j} )= Y^-_j$ with $\partial (X^+ )= Y^+$.
Then we obtain a continuous map 
\[
\frak{gl} \colon {M} (\frak{a},W^* ,\frak{b}) \times   \prod_{j=1}^m  \B (b_j,  Y^-_j ) \to \B\Bigl( a, W \cup   \bigcup_{j=1}^m  X^-_j \Bigr)
\]
by gluing connections using cut-off functions. 
This gives a bundle map 
\[
\widetilde{\frak{gl}} \colon \Det T {M} (\frak{a},W^* ,\frak{b})  \otimes  \bigotimes_{j=1}^m \mathbb{L}_{b_j } \to \frak{gl}^*   \mathbb{L}_{a}.
\]
We fix the orientation of $\Det T {M} (\frak{a},W^* ,\frak{b})$ so that $\widetilde{\frak{gl}}$ is orientation-preserving with respect to induced orientations from orientations of orientation bundles.
Then by counting the $0$-dimensional part of ${M} (\frak{a},W^* , \frak{b})$, we get a map 
\begin{align*}
CW^{[s,r]}_i \colon  CI^{[s,r]}_{i}(Y^+) \ri  \bigoplus_{ \substack{ \sum_j l_{j} =i, \\ 0\leq l_j \leq  i }} \bigotimes_{j=1}^m CI^{[s,r]}_{l_{j}}(Y^-_{j})
\end{align*}
for $i \in \z$.
In this paper, we use only the cases of $m=1$ and $m=2$. In particular, for the case of $m=2$, we also use the following map
\[
\widetilde{CW}^{[s,r]}_i \colon  CI^{[s,r]}_{i}(Y^+) \ri CI^{[s,r]}_{i}(Y^-_{1}) \oplus CI^{[s,r]}_{i}(Y^-_{2}) 
\]
defined via the $0$-dimensional moduli spaces ${M} (\frak{a},W^* , (b, \theta) )$ and ${M} (\frak{a},W^* , (\theta, b) )$.
(We use the weighted norm here.)

The following is the key lemma of this paper.
Roughly speaking, this lemma implies that the cobordism maps are filtered.

\begin{lem}\label{useful}
The following facts hold. 
\begin{enumerate}
\item Suppose $m=1$. Let $r \in \R_{Y^+ } \cap   \R_{Y^- _{1}}$, $s \in [-\infty, r)$, $\pi^+  \in \mathcal{P} (Y^+, r,s, g)$ and $\pi^-_1 \in \mathcal{P} (Y^-_1, r, s, g^-_{1}) $.
Let $\frak{a} = a  \in  R(Y^+)_{\pi^+}$ and $\frak{b}= b \in  R(Y^-_1)_{\pi^-_{1}}$. Suppose that 
$M(\frak{a},W^* ,\frak{b}) \neq \emptyset$ for some perturbation $\pi_W$ with the condition $(\ast \ast)$ and $a \in C^{[s,r ]}_* (Y^+)$.
Then, $
\cs_{\pi^-_1} (b) < r$
holds. 
\item Suppose $m=2$. Let $r \in \R_{Y^+ } \cap  \R_{Y_1^-} \cap \R_{Y_2^-}$ and $s, s_1 , s_2 \in [-\infty, r)$ with $s=s_1+s_2$, $r-s_1 \in  \R_{Y^- _{2}}$ and $r-s_2 \in  \R_{Y^- _{1}}$ and choose perturbations $\pi^+  \in \mathcal{P} (Y^+, r,r', g)$ , $\pi^-_1 \in \mathcal{P} (Y^-_1, r-s_2 , s_1, g^-_{1}) $ and $\pi^-_1 \in \mathcal{P} (Y^-_2, r-s_1 , s_2, g^-_{2}) $. 
Let  $a  \in  R(Y^+)_{\pi^+}$ and $\frak{b}= (b_1, b_2 ) \in R(Y^-_1)_{\pi^-_1} \times R(Y^-_2)_{\pi^-_2}$. If
$M(\frak{a},W^* ,\frak{b}) \neq \emptyset$ for some perturbation $\pi_W$ with the condition $(\ast \ast)$, $ s< \cs_{\pi^+} (a)< r $ and $s_1<\cs_{\pi_1^-} (b_1) < r-s_2$, then $
\cs_{\pi^-_2} (b_2) < r-s_1$ holds. 

\item Under the same assumption as in \Cref{keylemma}\textup{(2)}, the following fact holds.
Let $a  \in  R(Y^+)_{\pi^+}$ and $b_2 \in  R(Y^-_2)_{\pi^-_2}$. Suppose that 
$M({a},W^* ,(\theta, b_2)) \neq \emptyset$ for some perturbation $\pi_W$ with the condition $(\ast \ast)$ and $ s< \cs_{\pi^+} (a)< r $.
Then, $
\cs_{\pi^-_2} (b_2) < r-s_1$
holds. 
 \item Under the same assumption as in \Cref{keylemma}\textup{(2)}, we additionally assume $s_1+s_2 \in \R_{Y^+}$.
Let $a  \in  R(Y^+)_{\pi^+}$ and $\frak{b}= (b_1, b_2 ) \in R(Y^-_1)_{\pi^-_1}
 \times R(Y^-_2)_{\pi^-_2}$. If
$M(\frak{a},W^* ,\frak{b}) \neq \emptyset$ for some perturbation $\pi_W$ with the condition $(\ast \ast)$, $b_1 \in C^{[s_1,r-s_2]}(Y^-_1)$ and $b_2 \in C^{[s_2,r-s_1]}(Y^-_2)$, then $
\cs_{\pi^+ } (a) > s_1+s_2$
holds. 
 \item Under the same assumption as in \Cref{keylemma}\textup{(2)}, we additionally assume $s_1+s_2 \in \R_{Y^+}$.
Let $a  \in  R(Y^+)_{\pi^+}$ and $b_2 \in R(Y^-_2)_{\pi^-_2}$.
If $M({a},W^* ,(\theta, b_2)) \neq \emptyset$ for some perturbation $\pi_W$ with the condition $(\ast \ast)$ and $b_2 \in C^{[s_2,r-s_1]}(Y^-_2)$, then $
\cs_{\pi^+ } (a) > s_1+s_2$
holds. 
\end{enumerate}
\end{lem}

\begin{proof}
First, let us show (1).
By \Cref{keylemma}, we have \begin{align*}
|r- \cs_{ \pi^-_1 }(b) | > \frac{3}{4} d(r, \Lambda_{Y^-} ) .
\end{align*}
If $r- \cs_{ \pi^-_1 }(b) > \frac{3}{4} d(r, \Lambda_{Y^-} )$, then this is the conclusion. We assume 
\begin{align}\label{csv3}
r- \cs_{ \pi^-_1 }(b) < -\frac{3}{4} d(r, \Lambda_{Y^-} ).
\end{align}
Let $A$ be an element in $M(a,W^*, b)$. 
We set $A_+= A|_{\partial (Y^+\times \R_{\leq 0})}$ and $A_- =  A|_{\partial (Y^-\times \R_{\geq 0})}$.
Since $A$ is a flow of grad $\cs_{\pi^+}$ on $Y^+\times \R_{\leq 0}$, 
\[
\cs_{\pi^+} (a)\geq \cs_{\pi^+} (A_+).
\]
We also have $\cs_{\pi^-_1} (A_-) \geq  \cs_{\pi^-_1}(b)$ by the same argument. 
Moreover, we see
\begin{align*}
& \cs_{\pi_1^-}(A_-)-\cs_{\pi^+}(A_+)  \\
& = (\cs_{\pi^-_1} (A_-) - \cs(A_-))  - (\cs_{\pi^+}(A_+)- \cs(A_+)) + \cs(A_-)-\cs(A_+) \\
& \leq 
 \max \left\{ \frac{1}{4}d(r,\Lambda_{Y^+}) ,\frac{1}{4}d(r,\Lambda_{Y^-})\right\}
- \frac{1}{8\pi^2} \int_{W} \Tr F(A) \wedge F(A) \\
\end{align*}
Here, the second term is bounded by  $\frac{1}{8}\min \{d(r,\Lambda_{Y^+}),d(r, \Lambda_{Y^-_1})\} $ because
\begin{align*}
 - \frac{1}{8\pi^2} \int_{W} \Tr F(A) \wedge F(A)
& = -\frac{1}{8\pi^2} \int_W \Tr (F^+(A)+F^-(A)) \wedge (F^+(A) +F^-(A)) \\
& = - \frac{1}{8\pi^2} \int_W \left(
\Tr  \pi^+_W(A)\wedge \pi^+_W(A)  \right)  +   \frac{1}{8\pi^2}  \int_W \Tr F^-(A) \wedge * F^-(A) \\ 
& \leq \frac{1}{8}\min \{d(r,\Lambda_{Y^+}), d(r, \Lambda_{Y^-})\} -\frac{1}{8\pi^2} \| F^-(A) \|^2_{L^2(W)} 
\end{align*}
by the choice of $\pi_W$. 
 Therefore, we have  
 \begin{align*}
 \cs_{\pi_1^-}(A_-)-\cs_{\pi^+}(A_+) \leq  \frac{3}{8}\max \{d(r,\Lambda_{Y^+}), d(r, \Lambda_{Y^-_1})\}.
\end{align*}
On the other hand, we have 
\begin{align*}
\cs_{\pi^+}(a)  <  -\frac{3}{4} d(r,\Lambda_{Y^+}) +r 
\end{align*}
by \Cref{keylemma}. 
Now we have 
\begin{align}\label{csv2}
\cs_{\pi^-_1}(b) < r -  \frac{3}{4} d(r, \Lambda_{Y^+} ) + \frac{3}{8} \max \{d(r,\Lambda_{Y^+}), d(r, \Lambda_{Y^-_1})\}.
\end{align}
Combining \eqref{csv3} and \eqref{csv2}, we get
\[
0 < -\frac{3}{4} d(r, \Lambda_{Y^+} ) -\frac{3}{4} d(r, \Lambda_{Y^-_1} )+ \frac{3}{8}\max\{  d(r, \Lambda_{Y^+} ) , d(r, \Lambda_{Y^-_1} )\}
\]
This gives a contradiction. 

Next, we show (2).
Using \Cref{keylemma}, we have
\begin{align*}
|r-s_1- \cs_{ \pi^-_2 }(b_2) | > \frac{3}{4} d(r-s_1, \Lambda_{Y^-_2} ) .
\end{align*}
As the above discussion, we assume 
\begin{align*}
r-s_1- \cs_{ \pi^-_2 }(b_2)  <- \frac{3}{4} d(r-s_1, \Lambda_{Y^-_2} ) .
\end{align*}
Suppose $A \in M(\frak{a},W^* ,\frak{b})$,  $A_+ = A|_{ Y^+ \times \R_{\leq 0} }$, $A_-^1 = A|_{ Y^-_1 \times \R_{\geq 0} }$ and $A_-^2 = A|_{ Y^-_2 \times \R_{\geq 0} }$.
Then we have 
\begin{align*} 
& -r + \frac{3}{4} d(r, \Lambda_{Y^+}) +s_1 +  \frac{3}{4} d(s_1, \Lambda_{Y^-_1}) + \cs_{ \pi^-_2} (b_2) \\
& \leq -\cs_{\pi^+} (A_+) + \cs_{\pi^-_1}(A_-^1) + \cs_{\pi_2^-} (A_-^2) \\
& \leq \frac{1}{2} \max \{ d (r,\Lambda_{Y^+}) , d (s_1, \Lambda_{Y^-_1}) , d (r-s_1 ,  \Lambda_{Y^-_2}) \} 
\end{align*}
by the same argument.
These give a contradiction.

 Let us show (3). 
 By using \Cref{keylemma}, we have \begin{align*}
|r-s_1- \cs_{ \pi^-_2 }(b_2) | > \frac{3}{4} d(r-s_1, \Lambda_{Y^-_2} ) .
\end{align*}
As the above discussion, we assume 
\begin{align*}
r-s_1- \cs_{ \pi^-_2 }(b_2)  <- \frac{3}{4} d(r-s_1, \Lambda_{Y^-_2} ) .
\end{align*}
By a similar discussion, since $\cs_{\pi^-_1} (\theta) =0$, we have
\begin{align*} 
& -r + \frac{3}{4} d(r, \Lambda_{Y^+}) +s_1 +  \frac{3}{4} d(s_1, \Lambda_{Y^-_1}) + \cs_{ \pi^-_2} (b_2) \\
& \leq \frac{1}{2} \max \{ d (r,\Lambda_{Y^+}) , d (s_1, \Lambda_{Y^-_1}) , d (r-s_1, \Lambda_{Y^-_2}) \} .
\end{align*}
This gives a contradiction.

 Next, we show (4). 
 By using \Cref{keylemma}, we have 
\begin{align*} 
| \cs_{\pi^+} (a)-s | > \frac{3}{4} d(s,\Lambda_{Y^+} ) .
\end{align*}
As the above discussion, we assume 
\begin{align*} 
 cs_{\pi^+} (a)-s <- \frac{3}{4} d(s,\Lambda_{Y^+} ) .
\end{align*} For $A \in M(\frak{a},W^* ,\frak{b})$,  $A_+ = A|_{ Y^+ \times \R_{\leq 0} }$, $A_-^1 = A|_{ Y^-_1 \times \R_{\geq 0} }$ and $A_-^2 = A|_{ Y^-_2 \times \R_{\geq 0} }$, one has 
\begin{align*} 
 s_1 + \frac{3}{4} d(s_1, \Lambda_{Y^-_1}) +s_2 +  \frac{3}{4} d(s_2, \Lambda_{Y^-_2}) + \cs_{ \pi^+} (a) 
 \leq \frac{1}{2} \max \{ d (s, \Lambda_{Y^+}) , d (s_1, \Lambda_{Y^-_1}) , d (s_2,  \Lambda_{Y^-_2}) \} .
\end{align*}
These give a contradiction. 
The proof of (5) is similar to (4). 
\end{proof}

Fintushel-Stern showed the following lemma.
This is a corollary of \Cref{useful} under the assumption $m=1$. 
\begin{lem} 
For $r \in \R_{Y^+} \cap \R_{Y^- }$, 
\[
 CW^{[s,r]}\partial_{[s,r]}^{Y^+}=\partial_{[s,r]}^{Y^-} CW^{[s,r]} .
\]
\end{lem}

We denote the induced map of $CW^{[s,r]}$ (resp.\ $CW_ {[s,r]}$) on instanton Floer (co)homology by $IW^{[s,r]}$ (resp.\ $IW_ {[s,r]}$).

\subsection{Obstruction class}
In this section, we give a refinement of $D_1$ appeared in Donaldson's book \cite{Do02}, which counts gradient flows of Chern-Simons functional between irreducible critical points and the product connection. We define a filtered version of $ D_1$.

Let $Y$ be an oriented homology sphere. 
For $r \in \R_Y \cap [0, \infty]$ and $s \in [-\infty, 0]$,
we now define the invariant in $I^1_{[s,r]}(Y)$. A version of this invariant is defined in the third author's paper \cite{Ma18}. 

\begin{defn}\label{defiofthetar}
We set a homomorphism 
$
\theta^{[s,r]}_Y \colon CI^{[s,r]} _1(Y)\ri \z$
 by
\begin{align}
\theta^{[s,r]}_Y([a]):= \# (M^Y(a,\theta)_{\pi,\delta}/\R).
\end{align}
As in \cite[Section~3.3.1]{Do02} and \cite[Section~2.1]{Fr02}, we use the weighted norm $L^2_{q,\delta}$ in \eqref{weighted} for $M^Y(a,\theta)_{\pi,\delta}$ to use Fredholm theory. (In \cite[Section~3.3.1]{Do02} and \cite[Section~2.1]{Fr02}, $\theta$ is written by $D_1$ or $\delta$.) By the same discussion for the proof of $(\delta^{[s,r]})^2=0$, we can show $\delta^{[s,r]} (\theta^{[s,r]}_Y)=0$. Therefore it defines the class $[\theta^{[s,r]}_Y] \in I^1_{[s,r]}(Y)$. 
\end{defn}

The class $[\theta^{[s,r]}_Y]$ does not depend on the small perturbation and the metric. The proof is similar to the proof for the original one $[\theta]$.

\begin{lem} \label{func}
Let $Y_1$ and $Y_2$ be oriented homology spheres. Suppose that there is an oriented negative definite cobordism $W$  with $H^1(W;\R)=0$ and $\partial W= Y_1 \amalg (-Y_2)$. 
For $r \in \R_{Y_1} \cap \R_{Y_2}$, the equality
\[
IW_{[s,r]} [\theta^{[s,r]}_{Y_2}] = c(W) [\theta^{[s,r]}_{Y_1}]
\]
 holds, where $c(W):= \# H_1(W;\z) $.
\end{lem}

\begin{proof}
We first consider the case of $H^1(W;\z)=0$.
First we fix Riemannian metrics $g_i$ on $Y_i$, non-degenerate regular perturbations $\pi_i \in \mathcal{P} (Y_i, r, g_i)$ for  $i=1,2$ and orientations of $\lambda_{a, X}$ for every critical point $a$. To consider the cobordism map induced by $W$, we fix the perturbation $\pi_W$ satisfying the condition $(\ast \ast)$.
For fixed $a \in \wt{R}(Y_1)_{\pi_1}$ satisfying $\cs_{\pi_1}(a) < r$ and $\ind (a)=1$, we consider the end-cylindrical manifold $W^\ast$ and the moduli space $M(a,W^\ast,\theta)$ as above.
We can choose a perturbation $\pi_W$ such that $M(a,W^\ast,\theta)$ has a structure of a 1-manifold.
There is a natural orientation on $M(a,W^\ast,\theta)$ induced by orientations of $\lambda_{a, X}$.  
We now describe the end of a compactification of $M(a,W^\ast,\theta)$.

By dimension counting and instanton gluing, we have two maps $\mathfrak{gl}_1$  from
\[
\left(\bigcup_{b \in \wt{R}^\ast(Y_1)_{\pi_1},\, \ind(b)=0 } M^{Y_1}(a,b)_{\pi_1}/\R\times M(b,W^\ast,\theta)  \cup M^{Y_1}(a,\theta)/\R\times M(\theta,W^\ast,\theta)  \right) \times (0, \infty) \]
to $M(a,W^\ast,\theta)$ and
\[
\mathfrak{gl}_2 \colon \left( \bigcup_{c \in \wt{R}^\ast(Y_2)_{\pi_2},\, \ind(c)=1 } M(a,W^\ast, c) \times M^{Y_2}(c,\theta)_{\pi_2}/\R  \right)   \times (-\infty, 0)\to  M(a,W^\ast,\theta).
\]
These are diffeomorphisms onto their images.
Also, the complement of the union of their images is compact.

\begin{claim}\label{orientations of moduli}
By the definition of the orientations of $M^{Y_1}(a,b)_{\pi_1}/\R $, $M^{Y_1}(a,\theta)_{\pi_1}/\R$, and $M^{Y_2}(c,\theta)_{\pi_2}/\R$, the maps $\mathfrak{gl}_1$ and $\mathfrak{gl}_2$ are orientation-preserving.
\end{claim}

More details about orientations of the moduli spaces are explained in the proof of Claim~\ref{orientations of moduli} written after the proof of \Cref{func}.
Using the maps $\mathfrak{gl}_1$ and $\mathfrak{gl}_2$, one can compactify $M(a,W^\ast,\theta)$.
The end of the compactified moduli space is the disjoint union of three types of oriented points: 
\begin{itemize}
 \item $\displaystyle
 \bigcup_{b \in \wt{R}^\ast(Y_1)_{\pi_1},\, \ind(b)=0 } M^{Y_1}(a,b)_{\pi_1}/\R\times M(b,W^\ast,\theta)$,
 \item $\displaystyle
 M^{Y_1}(a,\theta)/\R\times M(\theta,W^\ast,\theta)$,
 \item $\displaystyle
 -\bigcup_{c \in \wt{R}^\ast(Y_2)_{\pi_2},\, \ind(c)=1 } M(a,W^\ast, c) \times M^{Y_2}(c,\theta)_{\pi_2}/\R$.
\end{itemize}
Here we follow the convention of orientations in \cite[Section~5.4]{Do02}. 
Since the first homology of $W^\ast$ is equal to zero and the formal dimension of $M(\theta, W^\ast, \theta)$ is $-3$, there is no reducible connection except for $\theta$ in $M(\theta, W^\ast, \theta)$.  So the space $M(\theta, W^\ast, \theta)$ has just one point. When the space $M^Y(a,b)_\pi/\R$ is non-empty, the inequality $\cs_{\pi_1} (a)>\cs_{\pi_2} (b)$ holds. Therefore, the first case can be written by 
\[
\bigcup_{b \in \wt{R}^\ast(Y_1)_{\pi_1}, \ind(b)=0,\, \cs_{\pi_1}(b)<r } M^{Y_1}(a,b)_{\pi_1}/ \R\times M(b, W^\ast, \theta).
\]
If we can write the third case as 
\[
- \bigcup_{c \in \wt{R}^\ast(Y_2)_{\pi_2}, \ind(c)=1,\, \cs_{\pi_2} (c) <r} M(a,W^\ast, c) \times M(c,\theta)_{\pi_2}/\R ,
\]
these countings imply that 
\begin{align} \label{important}
 \delta^r ( n^r)(a)  + \theta^{[s,r]}_{Y_1}(a)= CW^r \theta^{[s,r]}_{Y_2}(a)
\end{align}
for any $a \in CI^{[s,r]} _1(Y_1)$.
We will show if $M(a,W^\ast, c)$ is non-empty and $\cs_{\pi_1}(a)<r$ then $\cs_{\pi_2} (c)<r$.
By the use of \Cref{useful}, we get 
\[
\cs_{\pi_2}(c)<r.
\]
Next, we see the case of $H^1(W;\R)=0$.
We need to consider the transversality of the moduli space $M(\theta ,W^\ast, \theta)$.
As explained in \cite[(2.16)]{D18} and \cite{Do87}, 
there are two types of reducible flat connections and we can take a perturbation $\pi_W$ so that $M(\theta ,W^\ast, \theta)$ is finite set.
Moreover, Donaldson \cite{Do87} showed that the orientations of the points are the same.
One can count the points and see that $\# M(\theta ,W^\ast, \theta) = \#H_1(W;\z)$.
Using the above perturbation, we have
\[
 \delta^r ( n^r)(a)  + c(W) \theta^{[s,r]}_{Y_1}(a)= CW^r \theta^{[s,r]}_{Y_2}(a),
\]
where $c(W)= \#H_1(W;\z)$.
\end{proof}

\begin{proof}[Proof of \Cref{orientations of moduli}]
We discuss the orientations of the moduli spaces.
Let us consider the restriction maps
\begin{align*}
& M^{Y_1}(a,b)_{\pi_1}/\R\times M(b,W^\ast,\theta)  \times  (0, \infty) \to M(a,W^\ast,\theta), \\
& M(a,W^\ast, c) \times M^{Y_2}(c,\theta)_{\pi_2}/\R    \times (-\infty, 0) \to M(a,W^\ast,\theta)
\end{align*}
of $\mathfrak{gl}_1$ and $\mathfrak{gl}_2$, respectively, where $b \in \wt{R}^\ast(Y_1)_{\pi_1}$ with $\ind(b)=0$ and $c \in \wt{R}^\ast(Y_2)_{\pi_2}$ with $\ind(c)=1$.
We just give a sketch of the proof for these maps being orientation-preserving.
Note that the case $b= \theta$ is shown similarly.

\begin{enumerate}
\renewcommand{\labelenumi}{(\roman{enumi})}
 \item
We first introduce the configuration spaces $\B^{Y_1} (a, b)$, $\B(b, W^*, \theta)$, $\B(a, W^*, c )$ and $\B^{Y_2} (c, \theta)$ containing the moduli spaces $M^{Y_1} (a, b)$, $M(b, W^*, \theta)$, $M(a, W^*, c )$ and $M^{Y_2} (c, \theta)$, respectively, which are defined similarly to \eqref{configuration for X}. 
 \item
As introduced in \eqref{detline}, using the sliced and linearized ASD map $d^*_A + d_A^+ + d\pi^+_A$, we have the determinant line bundles $\mathbb{L} (a, b)$, $\mathbb{L} (b, W^*, \theta)$, $\mathbb{L} (a, W^*, c )$ and $\mathbb{L}(c, \theta)$ over $\B^{Y_1} (a, b)$, $\B(b, W^*, \theta)$, $\B(a, W^*, c )$ and $\B^{Y_2} (c, \theta)$, respectively.
Let $\Lambda^{Y_1} (a, b)$, $\Lambda (b, W^*, \theta)$, $\Lambda (a, W^*, c )$ and $\Lambda^{Y_2}  (c, \theta)$ denote the sets of orientations of $\mathbb{L} (a, b)$, $\mathbb{L} (b, W^*, \theta)$, $\mathbb{L} (a, W^*, c )$ and $\mathbb{L}  (c, \theta)$, respectively.
 \item
The (homotopy classes of) pregluing maps
\begin{align*}
& \B^{Y_1} (a, b) \times \B(b, W^*, \theta) \to \B(a, W^*, \theta), \\ 
& \B(a, W^*, c ) \times \B^{Y_2} (c, \theta) \to \B(a, W^*, \theta) 
\end{align*} 
give identifications 
\begin{align*} 
& i_1\colon \Lambda^{Y_1} (a,b) \times_{\z_2} \Lambda (b, W^*, \theta)  \to \Lambda (a, W^*, \theta), \\
& i_2\colon \Lambda (a, W^*, c )  \times_{\z_2}  \Lambda^{Y_2} (c, \theta) \to \Lambda (a, W^*, \theta) , 
\end{align*} 
where the $\z_2$-actions on the set of orientations are non-trivial.
See \cite[Proposition~5.11]{Do02} for details of the construction.
A similar argument can be found in \cite[Section~20.3]{KrMr07}.

 \item
We now fix orientations of the bundle $\lambda_{a, X_1}$ defined in \eqref{lambda ax} for critical points $a$ of the perturbed Chern-Simons functional of $Y_1$, where $X_1$ is a compact $4$-manifold bounded by $-Y_1$. 
Note that we have the canonical homology orientation of $W$ since $b_1(W)=0$ and $b^+(W)=0$.
Using these data, we associate an element of $\Lambda (a, W^*, \theta)$ with each critical point $a$ by the following discussion:
An excision argument about determinant line bundles similar to \cite[Proposition~5.11, p.~132]{Do02} gives an identification
\begin{align}\label{ori map pres}
& \{\text{The set of orientations of $\lambda_{a, X_1}$}\} \times_{\z_2} \Lambda (a, W^*, \theta) \\
&\hspace{5em} \to
\{\text{The set of orientations of $\lambda_{\theta, X_1\cup_{Y_2} W}$}\}. \notag
\end{align} 
Here we used the canonical homology orientation of $W$.
By the definition of  $\lambda_{a, X}$, there is a canonical orientation on $\lambda_{\theta, X_1\cup_{Y_2}  W}$.
For the orientation of $\lambda_{a, X_1}$ in (iv), one has the corresponding orientation in $\Lambda (a, W^*, \theta)$ via \eqref{ori map pres} and the canonical orientation on $\lambda_{\theta, X_1\cup_{Y_2}  W}$.
We fix these orientations in $\Lambda (a, W^*, \theta)$ for all $a$.
We also take an element in $\Lambda^{Y_1} (a,b)$ which is compatible with fixed elements in $\Lambda (b, W^*, \theta)$ and $\Lambda (a, W^*, \theta)$ under $i_1$.
Here, for a sufficiently small $T_1$, one can consider a gluing map 
\begin{align*}
 \mathfrak{g}'_1\colon  \{ [A_1] \in M^{Y_1}(a,b)_{\pi_1} \mid c(A_1)< T_1\}  \times M(b,W^\ast,\theta) \to M(a,W^\ast,\theta),
\end{align*}
where $c(A)$ denotes the center of the density function $\| F(A ) + \pi (A) \|_{Y\times \{r\}}$ ($r \in \R $) for a finite (perturbed) energy $\SU(2)$-connection $A$ on $Y\times \R$.

Let us explain the construction of $\mathfrak{g}_1'$, which implies that $\mathfrak{g}'_1$ is orientation-preserving.
For $([A_1], [A_2]) \in M^{Y_1}(a,b)_{\pi_1} \times M(b,W^\ast,\theta)$ with $c(A_1)\ll 0$, we first fix representatives $(A_1, A_2)$ of $([A_1], [A_2])$ so that the exponential decay estimates \cite[Proposition~4.3]{Do02} are satisfied.
In particular, under the assumption $c(A_1) \ll 0$, there exist positive constants $C_k$ and $\delta$ such that
\[
\|A_1|_{Y_1 \times [s,s+1]}  -p^*b \|_{L^2_k (Y_1 \times [s,s+1])} \leq C_k e^{\delta  ( c(A_1)-s)}
\]
holds for $c(A_1) \leq s \leq 0$, where $p$ denotes the projection $Y_1\times \R \to Y_1$.
Similarly, there exist $C'_k>0$ and $\delta'>0$ such that
\[
\| A_2|_{Y_1 \times [s,s+1] } - p^*b  \|_{L^2_k (Y_1 \times [s,s+1])} \leq C'_k e^{\delta' s}
\]
holds for $s \leq -1$.
We now define an $\SU(2)$-connection $\psi (A_1, A_2)$ on $W^*$ called a pregluing by 
\[
\psi (A_1, A_2) =
\begin{cases}
 A_1 & \text{on $Y_1 \times \left(-\infty, \frac{3}{4} c(A_1)\right)$,} \\
A_1\psi_{c(A_1)}(\textendash) + (1-\psi_{c(A_1)}(\textendash))p^*b & \text{on $Y_1 \times \left(\frac{3}{4}c(A_1)-1, \frac{1}{2}c(A_1) +\frac{1}{2}\right)$,} \\ 
(1- \psi_{c(A_1)}(\textendash) ) p^*b + A_2 \psi_{c(A_1)}(\textendash) & \text{on $Y_1 \times \left(\frac{1}{2}c(A_1)-\frac{1}{2}, \frac{1}{4}c(A_1) +1\right)$,}     \\
A_2 & \text{on $Y_1 \times \left(\frac{1}{4}c(A_1) , 0\right] \cup W\cup Y_2 \times [0,\infty )$.}
\end{cases} 
\]
Here $\psi_{s}\colon \left(\frac{3}{4}s -1, \frac{1}{4}s +1\right) \to \R$ is a smooth cut-off function satisfying
\[\psi_{s}(t) =  
\begin{cases}
 1 & \text{if $t \in \left(\frac{3}{4}s -1 , \frac{3}{4}s\right) \cup \left(\frac{1}{4}s, \frac{1}{4}s+1\right)$,} \\
 0 & \text{if $t \in \left(\frac{1}{2}s-\frac{1}{2}, \frac{1}{2}s+\frac{1}{2} \right)$}
\end{cases} 
\]
and $|d \psi _s|  \leq  c/|s|$ for $s \ll 0 $, where $c>0$ is a constant independent of $s$.
It follows from the exponential decay estimates that there exist $\delta''>0$ and $C_k'' >0$ such that
\[
\|F^+(\psi (A_1, A_2)) + \pi^+( \psi (A_1, A_2))  \|_{L^2_{k-1}(W^*)} \leq C''_k e^{\delta'' c(A_1)}.
\]
We perturb the connection $\psi (A_1, A_2)$ to obtain a solution by the following argument based on \cite[Proof of Theorem~9.1, p.~854]{SaWe08}.
First, for the above $(A_1, A_2)$ with $c(A_1)\ll 0$, we can prove that the operator
\[
\mathcal{D}_{\psi (A_1, A_2)} :=  d^* _{\psi (A_1, A_2)}+ d^+_{\psi (A_1, A_2)} + d\pi^+_{\psi (A_1, A_2)} 
\]
has a right inverse $Q_{\psi (A_1, A_2)}$ with uniform bounds stated in \cite[(112) and (113)]{SaWe08}.
Then, we apply the implicit function theorem for the perturbed ASD equation $F^+ (\psi (A_1, A_2)+a ) + \pi^+ (\psi (A_1, A_2)+a )=0$ with a slice $d^*_{\psi (A_1, A_2)} (a)=0$ and obtain a solution $a (A_1, A_2)$ to these equations.
For a sufficiently small $T_1$, we define the gluing map 
\begin{align*}
 \mathfrak{g}'_1 \colon  \{ [A_1] \in M^{Y_1}(a,b)_{\pi_1} \mid c(A_1)< T_1\}  \times M(b,W^\ast,\theta) \to M(a,W^\ast,\theta)
\end{align*}
by
\[
 \mathfrak{g}_1' ([A_1], [A_2]) = [\psi (A_1, A_2) + a (A_1, A_2)].
\]
This map is orientation-preserving by its construction.
Similarly, under suitable orientations, the map 
\[
 \mathfrak{g}'_2 \colon M(a,W^\ast, c) \times \{ [A_2] \in M^{Y_2}(c,\theta)_{\pi_2} \mid c(A_2)> T_2\} \to M(a,W^\ast,\theta) 
\]
is defined and orientation-preserving for a sufficiently large $T_2$.

 \item
Now, we orient the moduli spaces $M^{Y_1}(a,b)_{\pi_1}/\R$ and $M^{Y_2}(c,\theta)_{\pi_2}/\R$ so that 
\begin{align}\label{how to give orientations}
(M^{Y_1}(a,b)_{\pi_1}/\R ) \times \R = M^{Y_1}(a,b)_{\pi_1}
\text{ and }
(M^{Y_2}(c,\theta)_{\pi_2}/\R )  \times \R = M^{Y_2}(c,\theta)_{\pi_2}
\end{align}
hold as oriented ($\R$-equivariant) manifolds, where the orientations of $M^{Y_1}(a,b)_{\pi_1}$ and $M^{Y_2}(c,\theta)_{\pi_2}$ are those in (iv).
We now fix the convention of the $\R$-action as $c( s \cdot A ) = c(A) -s$ for $s \in \R$.
Here, we identify $\{ [A_1] \in M^{Y_1}(a,b)_{\pi_1} \mid c(A_1)< T_1\}$ with $(M^{Y_1}(a,b)_{\pi_1}/\R)\times(-T_1,\infty)$ by sending $[A_1] \mapsto ([A_1]/\R, -c(A_1))$, which is orientation-preserving due to the fact $c(r \cdot A) = c(A) -r$ for all $r \in \R$.
Also, one has a similar identification between $\{ [A_2] \in M^{Y_2}(a,b)_{\pi_2} \mid c(A_2)> T_2\}$ and $(M^{Y_2}(a,b)_{\pi_2}/\R)\times(-\infty, -T_2)$.
By these identifications, $\mathfrak{g}'_1$ and $\mathfrak{g}'_2$ induce orientation-preserving maps
\begin{align*}
& \mathfrak{g}_1\colon M^{Y_1}(a,b)_{\pi_1}/\R\times (-T_1, \infty) \times M(b,W^\ast,\theta)   \to M(a,W^\ast,\theta),\\
& \mathfrak{g}_2\colon M(a,W^\ast, c) \times M^{Y_2}(c,\theta)_{\pi_2}/\R    \times (-\infty, -T_2) \to M(a,W^\ast,\theta) 
\end{align*}
with respect to the orientations fixed in (iv) and (v).
Since $\dim M(b,W^\ast,\theta)=0$, we also have an identification
\[
M^{Y_1}(a,b)_{\pi_1}/\R \times M(b,W^\ast,\theta) \times (-T_1, \infty)  =M^{Y_1}(a,b)_{\pi_1}/\R\times (-T_1, \infty) \times M(b,W^\ast,\theta)
\]
as oriented manifolds.
Thus, $ \mathfrak{g}_1$ induces an orientation-preserving map
\[
 \mathfrak{g}_1\colon M^{Y_1}(a,b)_{\pi_1}/\R\times (-T_1, \infty) \times M(b,W^\ast,\theta)   \to M(a,W^\ast,\theta) . 
\]
\end{enumerate}
This completes the sketch of the proof.
\end{proof}

The following property of the class $\theta^{[s,r]}_Y$ is useful to study the invariants $\{r_s\}$. 
\begin{lem}\label{inclusion}
For $s, s' \in \R_Y$ and $r, r' \in \R$ with $s \leq s' \leq 0 \leq r \leq r'$, the inequality
\[
i^{[s,r]}_{[s',r']}[ \theta^{[s',r']}_Y] = [\theta^{[s,r]}_Y]
\]
 holds. 
\end{lem}
\begin{proof}
This property follows from the construction of  $i^{[s,r]}_{[s',r']}$ in \Cref{filtration}.
\end{proof}

\section{The invariant $r_s$}\label{The invariant $r_s$}
\subsection{Definition and invariance}
We now introduce a family of invariants of an oriented homology 3-sphere $Y$. The definition of our invariants uses the birth-death property of our obstruction class $[\theta^{[s,r]}_Y]$ given in the previous section.

Before introducing our invariant $r_s(Y)$, we need to show the following lemma.

\begin{lem}\label{fundamental lemma}
Let $R$ be a commutative ring with $1$.
For any homology $3$-sphere $Y$, 
\[
\Set { r\in (0, \infty] | 0=[\theta^{[s,r]}_Y \otimes \id_R]  \in I^1_{[s,r]} (Y;R)  } \neq \emptyset
\]
holds for any $s \in [-\infty, 0]$.
\end{lem}

\begin{proof}
Suppose that 
\[
\Set { r\in (0, \infty] | 0=[\theta^{[s,r]}_Y \otimes \id_R]  \in I^1_{[s,r]} (Y;R)  } = \emptyset.
\]
Then there exists a sequence $r_n \subset \R_Y$ with $0 \neq [ \theta^{[s,r_n]}_{Y}] \in I^1_{[s,r_n]}(Y) $ and $0<r_n \to 0$. We take a sequence of non-degenerate regular perturbations $\pi_n$ in $\mathcal{P}(Y,g,s,r_n) $ satisfying the following conditions: 
\begin{itemize} 
\item $\|\pi_n\|\to 0$. 
\item There exists a small neighborhood $U$ of $[\theta] \in \widetilde{\B}^*(Y) $ such that $(h_n)_f |_{U} =0$, where $\pi_n=(f, h_n)$. 
\end{itemize}
Note that we can assume the second condition because of
\[
\ker (* d_\theta \colon \ker d^*_{\theta}  \subset \Om^1_Y \otimes \mathfrak{su}(2)  \to \ker d^*_{\theta} )= H^1 (Y; \R) \otimes \mathfrak{su}(2) = \{0\}.
\]
Since $0 \neq [ \theta^{[s,r_n]}_{Y}]$, one can take a sequence $a_n \in \widetilde{R}(Y)_{\pi_n}$ such that $M^Y (a_n, \theta)_{\pi_n}$ is non-empty for all $n$ and $\cs_{\pi_n} (a_n) \to 0$.
Because of the choice of perturbations $\pi_n$, we have $a_n \notin U$ for each $n$.  We take a sequence $A_n$ in $M^Y (a_n, \theta)_{\pi_n}$.
Moreover there is no bubble because the dimension of moduli spaces $M^Y (a_n, \theta)_{\pi_n}$ is $1$.
Since $\{A_n\}$ has bounded energy and $\ind (a_n)=1$, there exists a sequence of real numbers $\{s_j\}$, subsequence $\{A_{n_j}\}$ of $\{A_n\}$ and gauge transformations $\{g_j\}$ on $Y \times \R$ such that $g_j^*T_{s_j}^* A_{n_j}$ converges to $A_\infty$ on $Y \times \R$, where $T_{s_j}$ is translation map on $Y\times \R$.
We denote the limit connection of $A_\infty$ by $a_\infty$.
One can see that $[a_\infty] \neq [\theta]$ because of the condition $a_n \notin U$.

On the other hand, we have $\cs(a_\infty) = 0$.
This implies that $A_\infty$ becomes a flat connection on $Y\times \R$.
However, $\lim_{  t \to \infty}  A_\infty|_{Y \times \{t\}} \cong \theta$ holds.
Since the connection $\theta$ is isolated in $\widetilde{R}(Y)$, we have $[a_\infty] = [\theta]$.
This gives a contradiction.
\end{proof}

\begin{defn}\label{def:r_s}
For $s \in [-\infty, 0]$ and a commutative ring $R$ with $1$ and an oriented homology $3$-sphere $Y$, we define
\[
r^{R}_s(Y) :=  \sup\left\{ r\in (0, \infty] \Bigm| 0=[\theta^{[s,r]}_Y \otimes \id_R] \in I^1_{[s,r]} (Y;R) \right\}, 
\]
where $\id_R$ is the identity map on $R$.
\end{defn}
We often abbreviate $\theta^{[s,r]}_Y \otimes \id_R$ to $\theta^{[s,r]}_Y$.
By definition, it follows that $r^R_s(Y)$ is invariant under orientation-preserving diffeomorphisms of $Y$.
In addition, the inequality $r^{\z}_s(Y) \leq r^{\q}_s(Y)$ obviously holds.
We focus on $r^\q_s (Y)$ in the most part of this paper, and hence we denote
$r^\q_s (Y)$ simply by $r_s(Y)$.

Non-triviality of $r_s$ implies the following:

\begin{thm} Suppose that $r_s(Y) < \infty$, then for any metric $g$ on $Y$, there exists a solution $A$ to the ASD equation on $Y \times \R$ with
\[
\frac{1}{8\pi^2} \| F(A) \| ^2_{L^2} = r_s(Y).
\]
\end{thm}

\begin{proof}Suppose that $r_s(Y) <\infty $ for some $s$.  We put $r=r_s(Y)$ and take a sequence $\epsilon_n$ with $0<\epsilon_n \to 0$ and a sequence of regular non-degenerate perturbations $\pi_n \in \mathcal{P}(Y,g,r+ \epsilon_n,s )$ with $\|\pi_n \|\to 0$.
Since $0\neq[\theta^{\epsilon_n +r}_s]$, we have a sequence $a_n \in \widetilde{R}(Y)_{\pi_n}$ such that $M^Y(a_n, \theta)_{\pi_n}$ is non-empty for all $n$ and $\cs_{\pi_n} (a_n) \to r$.
We take elements $A_n$ in $M^Y (a_n, \theta)_{\pi_n}$ for each $n$.
There is no bubble because the dimension of moduli spaces $M^Y(a_n, \theta)_{\pi_n}$ is $1$.
Since $\ind (a_n)=1$, by the gluing argument, we can conclude that there exists a sequence of real numbers $s_j$, subsequence $\{A_{n_j}\}$ of $\{A_n\}$ and gauge transformations $\{g_j\}$ on $Y \times \R$ such that $\{g_j^*T_{s_j}^* A_{n_j}\}$ converges $A_\infty$ on $Y \times \R$, where $T_{s_j}$ is translation map on $Y\times \R$.
We can see 
\[
\frac{1}{8\pi^2} \| F(A_\infty) \| ^2= \lim_{n\to \infty} \cs_{\pi_n}(a_n) = r = r_s(Y).
\]
Moreover, $A_\infty$ satisfies $F^+(A_\infty)=0$. This completes the proof.
\end{proof}

In the following, we state fundamental properties of $r^{R}_s$.
\begin{lem}[\Cref{main theorem}(2)]
\label{value}
For any $s \in [-\infty, 0]$ and a homology $3$-sphere $Y$, 
we have
$r^R_s(Y) \in \Lambda^*_Y \cup\{\infty\}$.
\end{lem}
\begin{proof}
By using Lemmas~\ref{inclusion} and \ref{filtration}, we obtain the conclusion.
\end{proof}
In the case of $S^3$, note that $\Lambda^*_{S^3}=\emptyset$.
Therefore, by \Cref{value},
we have $r^R_s(S^3) = \infty$ for any $s$.

\begin{lem}[\Cref{main theorem}(3)]
\label{different s}
Let $s\leq s'$ be non-positive numbers.
Then, for any homology $3$-sphere $Y$, we have $ r^R_{s'}(Y) \leq r^R_s(Y)$ holds.
\end{lem}

\begin{proof}
This is also a corollary of Lemmas~\ref{inclusion} and \ref{filtration}.
\end{proof}

Using \Cref{inclusion}, 
we have the following lemma. 

\begin{lem}
\label{r_0 and theta^r}
For any $s \in [-\infty, 0]$ and $r \in \R^{>0}_Y \cup \{\infty\} $,  if $r < r^{R}_s(Y)$, then $[\theta^{[s,r]}_Y]=0$, where $\R^{>0}_Y := \R_Y \cap (0,\infty)$.
\end{lem}

\begin{proof}
By the definition of $r_s^R(Y)$, we can take $r' \in \R^{>0}_Y$ such that $[\theta^{[s,r']}_Y]=0$ and $r<r'\leq r^R_s(Y)$.
Therefore, it follows from \Cref{inclusion} that
\[
[\theta^{[s,r]}_Y]= i_{ [s,r']}^{[s,r]}([\theta^{[s,r']}_Y])=0.
\]
\end{proof}

Now we show an important property of $r_s$.

\begin{thm}
\label{general cobneq}
Fix a commutative ring $R$ with $1$.
Let $Y_1$ and $Y_2$ be oriented homology $3$-spheres. 
Suppose that there is an oriented negative definite cobordism $W$ with $H^1(W;\R)=0$ and $\partial W = Y_1 \amalg -Y_2$. 
If $c(W)= \# H_1(W;\z)$ is invertible in $ R$, then the inequality
\[
 r^R_s(Y_2) \leq r^R_s(Y_1) 
\]
holds for any $s \in [-\infty,0]$.
Moreover, if $r^R_s(Y_2) = r^R_s(Y_1)< \infty$, then there exist irreducible $\SU(2)$-representations $\rho_1$ and $\rho_2$ of $\pi_1(Y_1)$ and $\pi_1(Y_2)$ respectively which extend to one of $\pi_1(W)$.
\end{thm}

\begin{proof}
Suppose that $r^R_s(Y_1) < \infty$.
For $\epsilon >0$ satisfying $\epsilon +r_s^R(Y_1) \notin \Lambda_{Y_2}$, 
by Lemma~\ref{func}, we get 
\[
IW^{\epsilon +r_s^R(Y_1) }_s [\theta^{[s, \epsilon +r_s^R(Y_1)] }_{Y_2}] = c(W) \theta^{[s, \epsilon +r_s^R(Y_1)] }_{Y_1}.
\]
Since $[\theta_s^{\epsilon +r_s^R(Y_1) }({Y_1})  ]\neq 0$ for any $\epsilon >0$ and $c(W)$ is invertible,  we have $r^R_s(Y_2) \leq r^R_s(Y_1)+ \epsilon$. This implies the conclusion. 

Suppose that $r:= r^R_s(Y_2) = r^R_s(Y_1)$ for some $s$.  Fix Riemannian metrics $g_1$ and $g_2$ on $Y_1$ and $Y_2$. We take a sequence $r<r_n \to r$ with $r_n \in \R_{Y_1} \cap \R_{Y_2}$, the classes $[\theta^{[s,r_n]}_{Y_1}] \neq 0$ and  $[\theta^{[s,r_n]}_{Y_2} ]\neq 0$.  Then we have sequences of regular perturbations on $Y_1$ and $Y_2$ denoted by $\{\pi^1_n\} \subset \mathcal{P}(Y_1, g_1,  r_n)$ and $\{\pi^2_n\} \subset \mathcal{P}(Y_2, g_2,  r_n) $ satisfying 
\[
0 \neq [\theta^{[s,r_n]}_{Y_1}] \in  CI^1_{[s,r_n]} (Y_1) \text{ and }  0 \neq  [\theta^{[s,r_n]}_{Y_2}] \in  CI^1_{[s,r_n]} (Y_2)  .
\]
Moreover, one can take critical points $a_n$ and $b_n$ of $\{\pi^1_n\}$ and $\{\pi^2_n\}$ and regular perturbations $\pi^n_W$ on  $W^*$  satisfying the following conditions 
\begin{itemize}
\item $\cs_{\pi^1_n} (a_n) \to r \text{, } \cs_{\pi^2_n} (b_n) \to r $, 
\item $\| \pi^i _n \|\to 0$ for $i=1$ and $2$, 
\item 
$\| \pi_W^n \|_{C^1} \to 0 $ and
\item $\theta_Y ( a_n) \neq 0$ and $\theta_Y ( b_n) \neq 0$ .
\end{itemize}
For all such data, by using \eqref{important}, we take $a_n$ and $b_n$ satisfying
\[
M(a_n ,W^* , b_n) \neq \emptyset . 
\]
Now we choose an element $A_n$ in $M(a_n, W^*,b_n) $ for each $n$. Since we take regular perturbations, the dimension of $M(a_n, W^*,b_n) $ is $0$. Since $\{A_n\}$ has bounded energy and there is no sliding end sequence by gluing argument. One can take a subsequence $\{A_{n_j}\}$ and gauge transformations $\{g_j\}$ such that $\{g_j^*A_{n_j}\}$ converges on $W^*$. 
We write the limit by $A_\infty$.
By the second condition, we can see that the limit points $a_\infty$ and $b_\infty$ are flat connections. 
Moreover, $\cs(a_\infty)= \cs(b_\infty)=r$ and $\|F(A_\infty)\|^2_{L^2(Y\times \R)}= 0$ hold.
Since we can take perturbations so that the reducible flat connections of $Y_1$ and $Y_2$ are isolated, we see that 
$a_\infty$ and $b_\infty$ are irreducible flat connections.
Therefore, $A_\infty$ determines some irreducible flat connection on $W$.
This gives a homomorphism $\rho(A_\infty) \colon \pi_1(W) \to \SU(2)$.
\end{proof}

This result gives the following conclusion. 

\begin{cor}
\label{invariance}
The invariants $r^R_s$ are homology cobordism invariants. 
\end{cor}

In addition, we also have the following corollary.

\begin{cor}
\label{pi1=1}
If there exists a negative definite simply connected cobordism
with boundary $Y_1 \amalg -Y_2$ and $r^R_s(Y_1)< \infty$,  
then the strict inequality 
\[
r_s^R(Y_2) < r_s^R(Y_1) 
\] 
holds. 
\end{cor}

Also by \Cref{general cobneq}, for the case of $r_s=r^{\q}_s$, we have the following.

\begin{thm}[\Cref{main theorem}(1)]
\label{cobneq}
Let $Y_1$ and $Y_2$ be oriented homology $3$-spheres. 
Suppose that there is an oriented negative definite cobordism $W$ with $\partial W = Y_1 \amalg -Y_2$. 
Then the inequality
\[
 r_s(Y_2) \leq r_s(Y_1) \]
holds for any $s \in [-\infty,0]$.
Moreover, if $H^1(W; \R)=0$ and $r_s(Y_2) = r_s(Y_1)< \infty$, then there exist irreducible $\SU(2)$-representations $\rho_1$ and $\rho_2$ of $\pi_1(Y)$ and $\pi_2(Y)$ respectively which extend the same representation of $\pi_1(W)$.  
\end{thm}

\begin{proof}
By surgering out loops representing the free part of $H_1(W;\z)$,
without loss of generality, we may assume that $H_1(W; \R)=0$.
Then, $c(W)= \# H_1(W;\z)$ is invertible in $\q$, and hence
\Cref{general cobneq} gives
\[
 r_s(Y_2) \leq r_s(Y_1). \]
The last-half assertion of \Cref{cobneq} directly follows from \Cref{general cobneq}.
\end{proof}

\begin{cor}
\label{obstruct bounding}
If a homology $3$-sphere $Y$ bounds a negative definite 4-manifold,
then for any $s \in [-\infty, 0]$, we have 
$$
r_s(Y)= \infty.
$$
\end{cor}

\begin{proof}
Suppose that $Y$ bounds a negative definite 4-manifold $X$, and
let $W$ denote $X$ with open 4-ball deleted.
Then $W$ is a negative definite 4-manifold with 
$\partial W = Y \amalg -S^3$. 
Therefore, by \Cref{cobneq}, we have
$$
r_s(Y) \geq r_s(S^3) = \infty
$$
for any $s \in [-\infty, 0]$.
\end{proof}

\subsection{Connected sum formula}

The aim of this subsection is to
prove the following connected sum formula for $r_s$.

\begin{thm}[\Cref{main theorem}(4)]
\label{conn}
Let $s,s_1,s_2 \in (-\infty, 0] $ with $s=s_1+s_2$.
For any homology $3$-spheres $Y_1$ and $Y_2$,
we have the inequality 
\[
r_s(Y_1 \# Y_2 ) \geq \min \{ r_{s_1}(Y_1)+s_2 , r_{s_2}(Y_2) +s_1\}. 
\]
\end{thm}

Before starting the proof, let us fix several additional data to define filtered instanton Floer homology. 
Fix $s, s_1, s_2 \in (-\infty, 0]$ with $s=s_1+s_2$
and homology 3-spheres $Y_1$ and $Y_2$.
Take $r \in \R^{>0}_{Y_1 \# Y_2}$
such that $r-s_2 \in \R^{>0}_{Y_1}$ and $r-s_1 \in \R^{>0}_{Y_2}$.
Fix Riemannian metrics $g_i$ on $Y_i$ (resp.\ $g_{\#}$ on $Y_1 \# Y_2$), non-degenerate regular perturbations $\pi_i \in \mathcal{P}(Y_i,r-s_j, s_i,g_i)$ for $\{i,j\}=\{1,2\}$ (resp.\ a non-degenerate regular perturbation $\pi_{\#} \in \mathcal{P}(Y_1 \# Y_2,r,s,g_{\#})$) and orientations on line bundles $\lambda_{a, X}$ with respect to $\pi_1$ , $\pi_2$ and $\pi_\#$. 
Here, we first suppose that $s_1 \in \R_{Y_1}$, $s_2 \in \R_{Y_2}$ and $s \in \R_{Y_1\# Y_2}$. 
Next, let us consider a cobordism $W$ with 
$\partial W = (Y_1 \# Y_2) \amalg -(Y_1 \amalg Y_2)$,
which consists of only a single 1-handle. 
Define $\q$-vector spaces $C^{[s,r]}_i$ $(i=0, 1)$ as
$$
C^{[s,r]}_0:=
\begin{array}{c}
CI^{[s_1,r-s_2]}_{0}(Y_1) \otimes_{\q} CI^{[s_2,r-s_1]}_0(Y_2)\\
\oplus\\
CI^{[s_1,r-s_2]}_{0}(Y_1)\\
\oplus\\
CI^{[s_2,r-s_1]}_{0}(Y_2)\\
\end{array}
$$
and
$$
C^{[s,r]}_1:=
\begin{array}{c}
(CI^{[s_1,r-s_2]}_{1}(Y_1) \otimes_{\q} CI^{[s_2,r-s_1]}_0(Y_2)) 
\oplus (CI^{[s_1,r-s_2]}_{0}(Y_1) \otimes_{\q} CI^{[s_2,r-s_1]}_1(Y_2))\\
\oplus\\
CI^{[s_1,r-s_2]}_{1}(Y_1)\\
\oplus\\
CI^{[s_2,r-s_1]}_{1}(Y_2)\\
\end{array}.
$$
By the discussion of Section~\ref{Cobordism maps}, the above initial data give the maps
\[
CW^{[s,r]}_i \oplus \widetilde{CW}^{[s,r]}_i  \colon CI^{[s,r]}_i(Y_1 \# Y_2) \to C^{[s,r]}_i
\]
 for $i= 0$ and $i=1$.
We denote $\text{pr}_j \circ  CW^{[s,r]}_i \oplus \widetilde{CW}^{[s,r]}_i$ by $p_jCW_ i^{[s,r]}$, where $\text{pr}_j$ is the projection to the $j$-th component of $C^{[s,r]}_i$ for $j \in \{1,2,3 \}$.
The following lemma is a key to prove the connected sum inequality.

\begin{lem}
\label{conn lem}
Suppose that $s_1 \in \R_{Y_1}$, $s_2 \in \R_{Y_2}$ and $s \in \R_{Y_1\# Y_2}$.
The homomorphisms $CW^{[s,r]}_0$ and $CW^{[s,r]}_1$ satisfy the following equalities: 
\begin{enumerate}
\item
$p_1  CW^{[s,r]}_0 \circ \partial^{[s,r]}_{Y_1 \# Y_2}-
\left( \partial^{[s_1,r-s_2]}_{Y_1} \otimes 1,
1 \otimes \partial^{[s_2,r-s_1]}_{Y_2} \right) \circ p_1CW^{[s,r]}_1 =0$,
\item
$p_2  CW^{[s,r]}_0 \circ \partial^{[s,r]}_{Y_1 \# Y_2}- \partial^{[s_1,r-s_2]}_{Y_1}
\circ p_2CW^{[s,r]}_1 - \left(0, 1 \otimes  \theta^{[s_2,r-s_1]}_{Y_2} \right) 
\circ p_1 CW^{[s,r]}_1 =0$, and
\item
$p_3  CW^{[s,r]}_0 \circ \partial^{[s,r]}_{Y_1 \# Y_2}- \partial^{[s_2,r-s_1]}_{Y_2} 
\circ p_3CW^{[s,r]}_1 -
\left( \theta^{[s_1,r-s_2]}_{Y_1} \otimes 1, 0\right) \circ p_1 CW^{[s,r]}_1 =0$.
\end{enumerate}
\end{lem}

\begin{proof}
First, let us prove (1).
For generators $[a] \in CI^{[s,r]}_1(Y_1 \# Y_2)$
and $[b_1] \otimes [b_2] \in CI^{[s_1,r-s_2]}_0(Y_1) \otimes_{\q} CI^{[s_2,r-s_1]}_0(Y_2)$,
we see that the moduli space $M(a, W^* , b_1 \amalg b_2)$
has a structure of oriented manifold whose dimension $1$ whose orientation is induced by the orientations of line bundles $\lambda_{a, X}$.
Moreover, by the gluing argument, we obtain the gluing map $\mathfrak{gl}$ from the union of
\begin{align} \label{11}
\left( \bigcup_{\substack{ [c] \in \tR^*(Y_1\# Y_2)_{\pi_{\#}},\, \ind (c)=0, \\ \cs_{\pi_\#} ([c]) < \cs_{\pi_\#} ([a]) }}
 M^{Y_1 \# Y_2} (a,c)_{\pi_\# }  /\R\times M(c,W^*,b_1 \amalg b_2) \right) \times (0, \infty), 
\end{align}
\begin{align} \label{22}
\left( \bigcup_{\substack{ [d] \in \tR^*(Y_1)_{\pi_1},\, \ind (d)=1,\\ \cs_{\pi_1} ([d]) >\cs_{\pi_1} ([b_1]) }}
 M(a ,W^*,d \amalg b_2)  \times M^{Y_1}(d,b_1)_{\pi_1 } /\R   \right)  \times  (-\infty, 0 ),
\end{align}
and
\begin{align}\label{33}
\left(  \bigcup_{ \substack{ [e] \in \tR^*(Y_2)_{\pi_2},\, \ind (e)=1,\\ \cs_{\pi_2} ([e]) >\cs_{\pi_2} ([b_2]) }}
 M(a ,W^*,b_1 \amalg e)  \times M^{Y_2}(e,b_2)_{\pi_2 } /\R  \right)   \times (-\infty, 0 )
\end{align}
to $M(a, W^* , b_1 \amalg b_2)$.
On the first two components \eqref{11} and \eqref{22}, we can check that $\mathfrak{gl}$ is orientation-preserving as in the case of $Y_2=\emptyset$.
For the third component \eqref{33}, in general, $\mathfrak{gl} $ changes the orientation by $(-1)^{\ind (b_1) }$.
It follows from a standard calculation of index bundles via gluing argument.
In our situation, since $\ind (b_1) =0$, $\mathfrak{gl}$ is orientation-preserving.
So, the oriented boundaries of the compactification of $M(a, W^* , b_1 \amalg b_2)$ are given as follows.
\begin{itemize} 
\item $\displaystyle
\bigcup_{
\substack{ [c] \in \tR^*(Y_1\# Y_2)_{\pi_{\#}},\, \ind (c)=0, \\ \cs_{\pi_\#} ([c]) < \cs_{\pi_\#} ([a])  } } 
 M^{Y_1 \# Y_2} (a,c)_{\pi_\# }  /\R\times M(c,W^*,b_1 \amalg 
b_2),
$
\item $-
\displaystyle
\bigcup_{\substack{ [d] \in \tR^*(Y_1)_{\pi_1},\, \ind (d)=1,\\ \cs_{\pi_1} ([d]) >\cs_{\pi_1} ([b_1]) }} M(a ,W^*,d \amalg b_2)  \times M^{Y_1}(d,b_1)_{\pi_1 } /\R,
$
\item 
$- 
\displaystyle
\bigcup_{ \substack{ [e] \in \tR^*(Y_2)_{\pi_2},\, \ind (e)=1,\\ \cs_{\pi_2} ([e]) >\cs_{\pi_2} ([b_2]) }} M(a ,W^*,b_1 \amalg e)  \times M^{Y_2}(e,b_2)_{\pi_2 } /\R. 
$
\end{itemize}

\begin{claim}
The following inequalities hold. 
\begin{itemize} 
\item $\cs_{\pi_\#} ([c]) > s$,
\item $\cs_{\pi_1} ([d]) <r-s_2 $, and 
\item $\cs_{\pi_2} ([e]) <r-s_1$.
\end{itemize}
\end{claim}

\begin{proof}
This is just a corollary of \Cref{useful}.
\end{proof}

By using this lemma, we can regard the above $[c]$, $[d]$ and $[e]$
as $[c] \in CI^{[s,r]}_0(Y_1 \# Y_2)$,   
$[d] \in CI^{[s_1,r-s_2]}_1(Y_1)$ and $[e] \in CI^{[s_2,r-s_1]}_1(Y_2)$
respectively.
Thus, we have
\begin{align*}
& p_1  CW^{[s,r]}_0 \circ \partial^{[s,r]}_{Y_1 \# Y_2}([a])
-
\left( \partial^{[s_1,r-s_2]}_{Y_1} \otimes 1, 0\right)
\circ p_1CW^{[s,r]}_1([a]) \\
& \hspace{15em} - \left(0,  1 \otimes \partial^{[s_2,r-s_1]}_{Y_2} \right)
\circ p_1CW^{[s,r]}_1([a])
=0.
\end{align*}
Next, we prove (2).
For generators $[a]\in CI^{[s,r]}_1(Y_1 \# Y_2)$ and $[b_1] \in CI^{[s_1,r-s_2]}_0(Y_1)$, consider $M(a, W^* , b_1 \amalg \theta_{Y_2})$ as an oriented $1$-manifold.
Then its ends are the following:
\begin{itemize} 
\item $\displaystyle
\bigcup_{[c] \in \tR^*(Y_1\# Y_2)_{\pi_\#},\, \ind (c)=0 } M^{Y_1\#Y_2}(a,c)_{\pi_\# } /\R\times M(c,W^*,b_1 \amalg \theta_{Y_2}),
$
\item $\displaystyle
-\bigcup_{ [d] \in \tR^*(Y_1)_{\pi_1},\, \ind (d)=1 } M(a ,W^*,d \amalg \theta_{Y_2})  \times M^{Y_1}(d,b_1)_{\pi_1 } /\R, 
$
\item $\displaystyle
-\bigcup_{ [e] \in \tR^*(Y_2)_{\pi_2},\, \ind (e)=1 } M(a ,W^*,b_1 \amalg e)  \times M^{Y_2}(e,\theta_{Y_2})_{\pi_2 } /\R. 
$
\end{itemize}
We need to show the following claim.

\begin{claim}
The following inequalities hold. 
\begin{itemize} 
\item $\cs_{\pi_\#} ([c]) > s$,
\item $\cs_{\pi_1} ([d]) <r-s_2 $ and 
\item $\cs_{\pi_2} ([e]) <r-s_1$.
\end{itemize}
\end{claim}

\begin{proof}
This is also a corollary of \Cref{useful}. 
\end{proof}

Hence, we have
\begin{align*}
& p_2  CW^{[s,r]}_0 \circ \partial^{[s,r]}_{Y_1 \# Y_2}([a]) - \partial^{[s_1,r-s_2]}_{Y_1}
\circ p_2CW_1([a]) \\
&\hspace{8em} - \left(0, 1 \otimes  \theta^{[s_2,r-s_1]}_{Y_2}\right) \circ p_1 CW^{[s,r]}_1(a) =0.
\end{align*}
By the same argument, the third assertion follows from
considering the 1-dimensional moduli space
$M(a, W^*, \theta_{Y_1} \amalg b_2)$.

Since the proof of (3) is essentially the same as that of (2), we omit it.
\end{proof}


Next, we define a homomorphism 
$\partial_C \colon C^{[s,r]}_1 \to C^{[s,r]}_0$ by
\[
\partial_C=
\begin{bmatrix}
\left( \partial^{[s_1,r-s_2]}_{Y_1} \otimes 1,  1 \otimes \partial^{[s_2,r-s_1]}_{Y_2} \right)
& 0 & 0 \vspace{1mm}\\
\left( 0, 1 \otimes  \theta^{[s_2,r-s_1]}_{Y_2} \right) & \partial^{[s_1,r-s_2]}_{Y_1} 
& 0 \vspace{1mm}\\
\left(  \theta^{[s_1,r-s_2]}_{Y_1} \otimes 1, 0 \right) & 0 & \partial^{[s_2,r-s_1]}_{Y_2}\\
\end{bmatrix}.
\]

\begin{lem}For $s_1 \in \R_{Y_1}$, $s_2 \in \R_{Y_2}$ and $s \in \R_{Y_1\# Y_2}$, we have the equality
\label{conn diff}
\[
\partial_{C} \circ CW^{[s,r]}_1 = CW^{[s,r]}_0 \circ \partial^{[s,r]}_{Y_1\# Y_2}.
\]
\end{lem}

\begin{proof}
By Lemma~\ref{conn lem}, we have
\begin{align*}
\partial_{C} \circ CW^{[s,r]}_1
&=
\begin{bmatrix}
\left(\partial^{[s_1,r-s_2]}_{Y_1} \otimes 1, 1 \otimes \partial^{[s_2,r-s_1]}_{Y_2} \right)
\circ p_1CW^{[s,r]}_1
\vspace{1mm}\\
\left( 0, 1 \otimes  \theta^{[s_2,r-s_1]}_{Y_2} \right)\circ p_1CW^{[s,r]}_1 +\partial^{[s_1,r-s_2]}_{Y_1} \circ p_2CW^{[s,r]}_1
 \vspace{1mm}\\
\left( \theta^{[s_1,r-s_2]}_{Y_1} \otimes 1, 0 \right) \circ p_1CW^{[s,r]}_1
\partial^{[s_2,r-s_1]}_{Y_2} \circ p_3CW^{[s,r]}_1\\
\end{bmatrix} \\
&=
\begin{bmatrix}
p_1CW^{[s,r]}_0 \circ \partial^{[s,r]}_{Y_1 \# Y_2} \\
p_2CW^{[s,r]}_0 \circ \partial^{[s,r]}_{Y_1 \# Y_2} \\
p_3CW^{[s,r]}_0 \circ \partial^{[s,r]}_{Y_1 \# Y_2} 
\end{bmatrix}
= CW^{[s,r]}_0 \circ \partial^{[s,r]}_{Y_1 \# Y_2}.
\end{align*}
\end{proof}


\begin{lem}
\label{conn theta}For $s_1 \in \R_{Y_1}$, $s_2 \in \R_{Y_2}$ and $s \in \R_{Y_1\# Y_2}$, there exists a cochain $f \in CI_{[s,r]}^{0}(Y_1 \# Y_2)$
such that
$$
\theta^{[s,r]}_{Y_1 \# Y_2} +  f \circ \partial^{[s,r]}_{Y_1 \# Y_2}
- 
\left( 0,  \theta^{[s_1,r-s_2]}_{Y_1},   \theta^{[s_2,r-s_1]}_{Y_2} \right) \circ CW^{[s,r]}_1 = 0.
$$
\end{lem}

\begin{proof}
For a generator $[a] \in CI^{[s,r]}_1(Y_1 \# Y_2)$, 
consider $M(a, W^*, \theta_{Y_1} \amalg \theta_{Y_2})$ as an oriented 1-manifold, and then its ends are the following:
\begin{itemize}
 \item $\displaystyle
 M^{Y_1 \# Y_2}(a,\theta_{Y_1 \# Y_2})_{\pi_\# } /\R\times M(\theta_{Y_1 \# Y_2}, W^*,\theta_{Y_1} \amalg \theta_{Y_2})$,
 \item $\displaystyle
 \bigcup_{[b] \in \tR^*(Y_1\# Y_2)_{\pi_\#},\, \ind (b)=0 } M^{Y_1 \# Y_2}(a,b)_{\pi_\# } /\R\times M(b,W^*,\theta_{Y_1} \amalg \theta_{Y_2})$,
 \item $\displaystyle
 -\bigcup_{ [c] \in \tR^*(Y_1)_{\pi_1},\, \ind (c)=1} M(a ,W^*,c \amalg \theta_{Y_2}) \times M^{Y_1}(c,\theta_{Y_1})_{\pi_1 } /\R$,
 \item $\displaystyle
 -\bigcup_{[d] \in \tR^*(Y_2)_{\pi_2},\, \ind (d)=1} M(a ,W^*,\theta_{Y_1} \amalg d)  \times M^{Y_2}(d,\theta_{Y_2})_{\pi_2 } /\R$.
\end{itemize}
Since $Y_1$ and $Y_2$ are homology spheres, we see that 
$M(\theta_{Y_1 \# Y_2},W^*,\theta_{Y_1} \amalg \theta_{Y_2})$ has just one point.
Thus, defining a homomorphism $f\colon CI^{[s,r]}_0(Y_1 \# Y_2) \to \q$
by
$$[b] \mapsto \#(M(b, W^*, \theta_{Y_1} \amalg \theta_{Y_2})),$$
we have
$$ 
\theta^{[s,r]}_{Y_1 \# Y_2} ([a]) + f \circ \partial^{[s,r]}_{Y_1 \# Y_2}([a])
-  \theta^{[s_1,r-s_2]}_{Y_1} \circ p_2 CW_1([a]) 
-   \theta^{[s_2,r-s_1]}_{Y_2} \circ p_3 CW_1([a])=0.
$$
This completes the proof. 
\end{proof}

\begin{thm}
\label{conn coboundary}Let $s_1 \in \R_{Y_1}$, $s_2 \in \R_{Y_2}$ and $s \in \R_{Y_1\# Y_2}$. 
If $ \theta^{[s_1,r-s_2]}_{Y_1}$ and $ \theta^{[s_2,r-s_1]}_{Y_2}$
are coboundaries, then $\theta^{[s,r]}_Y(Y_1 \# Y_2)$ is also a coboundary.
\end{thm}

\begin{proof}
Suppose that a $0$-cochain $f_i \in CI_{[s_i,r-s+s_i]}^0(Y_i)$ satisfies 
$f_i \circ \partial^{[s_i, r-s+s_i]}_{Y_i}=  \theta^{[s_i,r-s+s_i]}_{Y_i}$ for each $i \in \{1,2\}$.
Then, we have a homomorphism 
\[
f' : = \left( -f_1 \otimes f_2, 
f_1, 
f_2 \right) \colon C_0^{[s,r]} \to \q
\]
and the equalities
\begin{align*}
f' \circ \partial_C
&= \left(- f_1 \otimes f_2, 
 f_1, 
 f_2 \right) \circ \partial_C \\
&= 
\left( \ast, 
 f_1 \circ \partial^{[s_1,r-s_2]}_{Y_1}, 
 f_2 \circ \partial^{[s_2,r-s_1]}_{Y_2} \right) \\
&=
\left( \ast, 
  \theta^{[s_1,r-s_2]}_{Y_1}, 
  \theta^{[s_2,r-s_1]}_{Y_2} \right),
\end{align*}
where
\[
\ast = 
-\left( 
 (f_1 \circ \partial^{[s_1,r-s_2]}_{Y_1}) \otimes f_2,
 f_1 \otimes (f_2 \circ \partial^{[s_2,r-s_1]}_{Y_2})
\right)
+\left(0, f_1 \otimes  \theta^{[s_2,r-s_1]}_{Y_2} \right)
+ \left( \theta^{[s_1,r-s_2]}_{Y_1} \otimes f_2, 0 \right)
=0.
\]
Therefore, combining it with Lemmas~\ref{conn diff} and \ref{conn theta}, we have
\begin{align*}
\theta^{[s,r]}_Y(Y_1 \# Y_2) 
&=
-  f \circ \partial^{[s,r]}_{Y_1 \# Y_2}
+ \left(0,  \theta^{[s_1,r-s_2]}_{Y_1},   \theta^{[s_2,r-s_1]}_{Y_2} \right)
\circ CW^{[s,r]}_1 \\
&=
-  f \circ \partial^{[s,r]}_{Y_1 \# Y_2}
+ f' \circ \partial_{C}
\circ CW^{[s,r]}_1 \\
&=
(- f + f' \circ CW^{[s,r]}_0) \circ \partial^{[s,r]}_{Y_1 \# Y_2}.
\end{align*}

\end{proof}

\def\proofname{Proof of \Cref{conn}}

\begin{proof}
First, we suppose that $s_1 \in \R_{Y_1}$, $s_2 \in \R_{Y_2}$ and $s \in \R_{Y_1\# Y_2}$.
Without loss of generality, we may require that $0< r_{s_1}(Y_1)+s_2 \leq r_{s_2}(Y_2)+s_1$.
Assume $r_s(Y_1 \# Y_2)< r_{s_1}(Y_1)+s_2$.
Then, there exists $r \in \R^{>0}_{Y_1 \# Y_2}$
such that $r_s(Y_1 \# Y_2)< r< r_{s_1}(Y_1)+s_2$, $r-s_2 \in \R_{Y_1}$ and $r-s_1\in \R_{Y_2}$.
Lemma~\ref{r_0 and theta^r} implies that 
\[
[ \theta^{[s_1,r-s_2]}_{Y_1}]=0 \text{ and } [ \theta^{[s_2,r-s_1]}_{Y_2}]=0.
\]
Therefore, by Theorem~\ref{conn coboundary},
we have $[\theta^{[s,r]}_{Y_1 \# Y_2}]=0$.
This contradicts to the assertion $r_s(Y_1 \# Y_2)< r$,
and hence we have $r_s(Y_1 \# Y_2) \geq r_{s_1}(Y_1)+s_2$.

Here, we give the case of $s_1 \in \Lambda_{Y_1}$, $s_2 \in \Lambda_{Y_2}$ or $s \in \Lambda_{Y_1\# Y_2}$.
Let $\{s^n_i\}_{n \in \z_{>0}}$ be sequences for $i=1$ and $2$ such that $s^n_1 \to s_1$, $s^n_2 \to s_2$, $s^n_1 \leq s_1$, $s^n_2 \leq s_2$, $s_1^n +s_2^n$ is a regular value of $Y_1\# Y_2$, and $s_i^n$ is a regular value of $Y_i$ for $i=1$ and $2$. Since the choices of $\{s^n_i\}_{n \in \z_{>0}}$ for $i=1$ and $2$, we have 
\[
r_{s_1^n +s_2^n} (Y_1 \#Y_2) \geq \min \{ r_{s_1^n} (Y_1) +s_2^n, r_{s_2^n}(Y_2) +s_1^n\} .
\]
For a sufficiently large $n$, we have $r_{s_1^n +s_2^n} (Y_1 \#Y_2)= r_s(Y_1 \#Y_2)$, $r_{s_1^n} (Y_1)= r_{s_1}(Y_1)$ and $r_{s_2^n} (Y_2)= r_{s_2}(Y_2)$.
This completes the proof. 
\end{proof}

\def\proofname{Proof}

\begin{rem}
As described in \cite[Section 9.4]{Sc15}, we have Fukaya's translation
of $CI_*(Y_1 \# Y_2)$ into a certain combination of $CI_*(Y_1)$ and $CI_*(Y_2)$.
From this viewpoint, we can extend $C^{[s,r]}_0$ and $C^{[s,r]}_1$ to a chain complex
$C_*$ such that we have a ``projection"
$$
CI_*(Y_1 \# Y_2) \twoheadrightarrow C_*.
$$
We guess that an alternative proof of Theorem~\ref{conn}
is obtained from a filtered version $C^{[s,r]}_*$ of $C_*$,
and such a proof would be more natural.
However, establishing $C^{[s,r]}_*$
requires too many extra arguments to prove the theorem, and so we extract
a small part of $C^{[s,r]}_*$.
\end{rem}

\section{Comparison with Daemi's invariants} \label{Comparison with Daemi's invariants}
In this section, we compare our invariants $r_s(Y)$ with Daemi's invariants $\Gamma_Y(k)$.

In \cite{D18}, Daemi constructed a family of real-valued homology cobordism invariants 
$\{\Gamma_Y(k)\}_{k \in \z}$, which has the following properties:
\begin{itemize} 
\item Let $Y_1$ and $Y_2$ be homology $3$-spheres,
and $W$ a negative definite cobordism with $\partial W= Y_1 \amalg -Y_2$
and $b_1(W)=0$. 
Then there exists a constant $\eta(W)\geq 0$ such that  
\[
 \Gamma_{Y_1 } (k) \leq 
\begin{cases}
 \Gamma_{Y_2} (k) - \eta(W) & \text{if $k >0$,} \\
 \max \{\Gamma_{Y_2} (k) - \eta(W), 0 \} & \text{if $k\leq 0$.}
\end{cases}
\]
Moreover, the constant $\eta(W)$ is positive unless 
there exists an $\SU(2)$-representation of $\pi_1(W)$ whose restrictions to
both $\pi_1(Y_1)$ and $\pi_1(Y_2)$ are non-trivial.

\item The inequalities
\[
\dots  \leq  \Gamma_Y(-1) \leq  \Gamma_Y(0) \leq \Gamma_Y(1) \leq \cdots
\]
hold. 
\item $\Gamma_{Y} (k)$ is finite 
if and only if $2h(Y) \geq k$, where $h(Y)$ is the Fr{\o}yshov invariant of $Y$.
\end{itemize} 

In this section, we prove that our invariant $r_{- \infty}(Y)$ coincides with $\Gamma_{-Y}(1)$.
\setcounter{section}{1}
\setcounter{thm}{3}
\begin{thm}
For any $Y$, the equality
  \[ 
  r_{-\infty}  (Y) =   \Gamma_{-Y}(1)
  \]
  holds. 
\end{thm}
\setcounter{section}{4}
\setcounter{thm}{0}
Since $r_s(Y) \leq r_{-\infty}(Y)$ for any $s \in \R_{<0}$, 
several facts and calculations for $r_s$ immediately follow from 
the study of $\Gamma_Y(1)$ in \cite{D18}. We also discuss them in this section.

\subsection{Review of Daemi's $\Gamma_Y(1)$}
Here, we need to compare our notations with 
\cite{D18}. The instanton Floer chain complex depends on the choice of several conventions. For example, our convention of the sign of Chern-Simons functional is different from Daemi's one (see Table~\ref{convention}).
In this section, we consider a fixed homology 3-sphere $Y$. 

\begin{table}[h]
\centering
\begin{tabular}{c|c|c|c}
\hline
 & sign of $\cs$ 
& Cylinder & Gradient   \\ \hline
\rule{0mm}{4mm}
\cite{D18} & $-$  
&  $\R \times Y$ & downward \\ \hline
\rule{0mm}{4mm}
 ours  & $+$  
&   $Y \times \R$  & downward \\ \hline 
\end{tabular}
\vspace{0.5em}
\caption{Conventions.}
\label{convention}
\end{table}

First we introduce the coefficient ring $\Lambda$ given by 
\[
\Lambda := \Set { \sum_{i=1}^{\infty} q_i \lambda^{r_i} | q_i \in \mathbb{Q},\ r_i \in \R,\ \lim_{i\to \infty } r_i =\infty } ,
\]
where $\lambda$ is a formal variable.
We have an evaluating function $\mdeg\colon  \Lambda  \to \R$ defined by 
\[
 \mdeg( \sum_{i=1}^{\infty} q_i \lambda^{r_i} ) := \min_{i \in \z_{>0}} \{r_i \mid q_i \neq 0\}.
 \] 
 Fix a non-degenerate regular perturbation $\pi$ and orientations of $\lambda_{a, X}$.
Then a $\z/8$-graded chain complex $C^{\Lambda}_*(Y)$ over $\Lambda$ is defined by
  \[
C^{\Lambda}_i(Y):= C_i(Y) \otimes_{\mathbb{Z}} \Lambda = \Lambda \{ [a] \in R^* (Y)_\pi \mid \ind (a)=i \},
  \]
with differential
\[
d^{\Lambda} ([a]) 
:= \sum_{\ind (a) -\ind(b) \equiv 1 \bmod 8}  \# (M^Y([a],[b])_\pi/\R) 
\cdot \lambda^{\mathcal{E}(A)} [b],
\]
where  $[A] \in M^Y([a],[b])_\pi$, $\mathcal{E}(A) :=\frac{1}{8\pi^2} \int_{Y \times \R} \Tr ( (F(A) + \pi(A)) \wedge (F(A)+\pi(A)) )$ and $M^Y([a],[b])_\pi$ denotes $M^Y(a, b)_\pi$ for some representatives $a$ and $b$ respectively of $[a]$ and $[b]$ satisfying $\ind (a) -\ind (b)= 1$.
{Note that we use $\wt{R}^\ast(Y)_\pi$ as a generating set of the Floer chain complex.
On the other hand, in Daemi's formulation, the chain group is generated by ${R}^\ast(Y)_\pi$.}
In our notation, $a$, $b$ denote elements of $\wt{R}^\ast(Y)_\pi$ and let $[a]$ and $[b]$ denote the images of $a$ and $b$ in ${R}^\ast(Y)_\pi$.

\begin{rem}In Daemi's formulation, $C_*(Y)$ and $C^{\Lambda}_i(Y)$ are regarded as the $\z/8\z$-graded chain complexes. 
\end{rem}
We extend the function $\mdeg$ to a function on $C^{\Lambda}_*$ by
$$
\mdeg(\sum_{1 \leq k \leq n} \eta_k [a_k]) = \min_{1 \leq k \leq n} \{ \mdeg(\eta_k) \}.
$$
In addition, we define the map $D_1 \colon C^{\Lambda}_1(Y) \to \Lambda$ by
$$
D_1([a]) = (\# M^Y([a],[\theta])_\pi/\R) \cdot \lambda^{\mathcal{E}(A)},
$$
where $M^Y([a],[\theta])_\pi$ denotes $M^Y(a, \theta^i)_\pi$ for some lifts $a$ and $\theta^i$ respectively of $[a]$ and $[\theta]$ satisfying $\ind (a)  -\ind (\theta^i)= 1$, and $A \in M^Y([a],[\theta])_\pi$.
Now, in our conventions, 
$\Gamma_{-Y}(1)$ is described as
$$
\Gamma_{-Y}(1) = \lim_{\|\pi \| \to 0}  \left(
\inf_{\substack{\alpha \in C_1^{\Lambda}(Y),\, d^{\Lambda}(\alpha)=0 \\ D_1 (\alpha) \neq 0 }}
\left\{
\mdeg(D_1 (\alpha))- \mdeg (\alpha)   
\right\}
\right).
$$

\subsection{Translating $\Gamma_Y(1)$ into $\z$-grading}
Following the construction of $C^{\Lambda}_*$, we can define 
a $\z$-graded chain complex
 \[
  CI^\Lambda_i(Y):= CI_i(Y) \otimes_{\mathbb{Q}} \Lambda = \Lambda \Set{ a \in \widetilde{R}^* (Y)_\pi  | \ind (a)=i }
  \]
with differential
\begin{align*}
\partial^\Lambda (a) 
&:= \sum_{\ind (a) -\ind(b)=1, A \in M^Y(a,b)_\pi}  \# (M^Y(a,b)_\pi/\R) 
\cdot \lambda^{\mathcal{E}(A)} b \\
&= \sum_{\ind (a) -\ind(b)=1}  \# (M^Y(a,b)_\pi/\R) \cdot \lambda^{ \cs_{\pi}(a) - \cs_{\pi}(b)  } b .
\end{align*}
We can also define the map $\theta^\Lambda \colon CI^\Lambda_1(Y) \to \Lambda$ by
$$
\theta^\Lambda(a) 
= (\# M^Y(a,\theta)_\pi/\R) \cdot \lambda^{\mathcal{E}(A)}
= \theta^{[-\infty, \infty]}_Y(a) \lambda^{\cs_{\pi}(a)}.
$$
We define a $\z$-graded version of $\Gamma_{-Y}(1)$ by
$$
\widetilde{\Gamma}_{-Y}(1) = \lim_{\|\pi\| \to 0}  \left(
\inf_{\substack{\alpha \in CI_1^{\Lambda}(Y), \partial^{\Lambda}(\alpha)=0 \\ \theta^\Lambda (\alpha) \neq 0 }}
\left\{
\mdeg(\theta^\Lambda (\alpha))- \mdeg (\alpha)   
\right\}
\right).
$$

\begin{lem}
$
\widetilde{\Gamma}_{-Y}(1)= \Gamma_{-Y}(1).
$
\end{lem}
\begin{proof}
The maps $\psi_i \colon CI_i^\Lambda(Y) \to C_{i}^\Lambda(Y)$
$( 0 \leq i \leq 7)$ induced by $\widetilde{R}(Y)_\pi \to R(Y)_\pi$ are 
$\Lambda$-linear isomorphisms such that
\begin{itemize}
\item $d^{\Lambda} \circ \psi_i = \psi_{i-1} \circ \partial^{\Lambda}$ for each $1 \leq i \leq 7$,
\item $D_1 \circ \psi_1 = \theta^{\Lambda}$, and
\item $\mdeg$ is preserved by the $\psi_i$.
\end{itemize}
These imply that the infimum in the definition of $\widetilde{\Gamma}_{-Y}(1)$ coincides
with that of $\Gamma_{-Y}(1)$ for each $\pi$, and hence 
$\widetilde{\Gamma}_{-Y}(1)= \Gamma_{-Y}(1).$
\end{proof}

\subsection{Proof of \Cref{compare}}
In this subsection, we fix orientations of the line bundles $\lambda_{a, X}$ for non-degenerate regular perturbations. 

\begin{lem}
\label{lambda cycle}
Let $\alpha =\sum_{1 \leq k \leq n} a_k \otimes x_k$
be a chain of $CI^{[-\infty,r]}_1(Y) \otimes \q$. 
Then 
$\alpha$ is a cycle if and only if
$\widetilde{\alpha}$
is a cycle of $CI^\Lambda_1(Y)$, where $\widetilde{\alpha}:=\sum_{1 \leq k \leq n} a_k \otimes x_k \lambda^{-\cs(a_k)}$. 
\end{lem}
\begin{proof}
For any generator $b \otimes 1 \in CI^{[-\infty,r]}_0(Y) \otimes \q$,
the coefficient of $b \otimes 1$ in 
$\partial^{[-\infty, r]}(\alpha) \in CI^{[-\infty, r]}_1(Y) \otimes \q$ is
$$
\sum_{1 \leq k \leq n} \# ( M^Y(a_k, b)_{\pi}/\R ) \cdot x_k,
$$
while the coefficient of $b \otimes 1$ in $\partial^\Lambda(\widetilde{\alpha})$ is
$$
\sum_{1 \leq k \leq n} \# ( M^Y(a_k, b)_{\pi}/\R ) \cdot 
\lambda^{\cs_{\pi}(a_k) - \cs_{\pi}(b)} \cdot x_k \lambda^{- \cs_{\pi}(a_k)} 
$$
$$
=\left(\sum_{1 \leq k \leq n} \# ( M^Y(a_k, b)_{\pi}/\R ) \cdot 
 x_k \right) \cdot \lambda^{ - \cs_{\pi}(b)}.
$$
This completes the proof.
\end{proof}

\begin{lem}
\label{lambda theta}
For a chain $\alpha=\sum_{1\leq k \leq n}a_k \otimes x_k$ 
of $CI^{[-\infty,r]}_1(Y) \otimes \q$,
we have
$\theta^{[-\infty,r]}_{Y}(\alpha)\neq0$ if and only if 
$\theta^\Lambda(\widetilde{\alpha}) \neq 0$. 
Moreover, if  
$\theta^{[-\infty,r]}_Y(\alpha) \neq 0$, then 
for a number $k' \in \{1, \ldots, n\}$ with $x_{k'}\neq0$ and $\cs_{\pi}(a_{k'})= 
\max\{\cs_{\pi}(a_k)\mid x_k \neq 0\}$,
we have
$$
\mdeg(\theta^\Lambda(\widetilde{\alpha}))-\mdeg(\widetilde{\alpha})
= \cs_{\pi}(a_{k'})<r.
$$
\end{lem}

\begin{proof}
The proof  of the first assertion follows from the same argument as
Lemma~\ref{lambda cycle}.
Moreover, 
it is easy to see that
$$
\mdeg(\widetilde{\alpha})= 
\mdeg(\sum_{1 \leq k \leq n} a_k \otimes x_k \lambda^{-\cs(a_k)})
= -\cs_{\pi}(a_{k'})
$$
and
$
\mdeg(\theta^\Lambda(\widetilde{\alpha}))= 0.
$
\end{proof}

\begin{lem}
\label{lambda replace}
Let $\hat{\alpha} = \sum_{1 \leq k \leq n} a_k \otimes \eta_k$
be a cycle of $CI^\Lambda_1(Y)$
with $d := \mdeg(\theta^\Lambda(\hat{\alpha})) < \infty$,
and $x_k$ the coefficient of $\lambda^{d-\cs_{\pi}(a_k)}$
in $\eta_k$ $(k=1, \ldots, n)$. Then 
$$
\alpha := \sum_{1 \leq k \leq n} a_k \otimes x_k 
$$
is a cycle of $CI_1(Y) \otimes \q$ with $\theta^{[-\infty, \infty]}_Y(\alpha) \neq 0$.
Moreover, for a number $k' \in \{1, \ldots, n\}$ with $x_{k'}\neq0$ and 
$\cs_{\pi}(a_{k'})= \max\{\cs_{\pi}(a_k)\mid x_k \neq 0\}$,
 the cycle $\widetilde{\alpha}$
satisfies
$$
\mdeg (\theta^\Lambda(\widetilde{\alpha}))-\mdeg(\widetilde{\alpha})
= \cs_{\pi}(a_{k'}) \leq \mdeg (\theta^\Lambda(\hat\alpha))
-\mdeg(\hat\alpha).
$$
\end{lem}
\begin{proof}
The coefficient of
$\lambda^{d}$ in $\theta^\Lambda(\hat\alpha)$ 
is equal to
$$
\sum_{1 \leq k \leq n} 
\# (M^Y(a_k,\theta)_\pi/\R) \cdot x_k,
$$
which coincides 
with $\theta^{[-\infty, \infty]}_Y(\alpha)$.
Moreover, since $\mdeg(\theta^\Lambda(\hat\alpha)) =d$, this value is non-zero.
Hence we have $\theta^{[-\infty, \infty]}_Y(\alpha) \neq 0$. 
In a similar way, we can also verify that for any generator 
$b \otimes 1 \in CI_0(Y) \otimes \q$, the coefficient of $b \otimes 1$ in
$\partial(\alpha)$ is equal to that of $\lambda^{d-\cs_{\pi}(b)}$ in
$\partial^{\Lambda}(\hat\alpha)$, and hence
$\partial(\alpha) = 0$.
Next, it follows from Lemma~\ref{lambda theta}
that $\mdeg (\theta^\Lambda(\widetilde{\alpha}))-\mdeg(\widetilde{\alpha})
= \cs_{\pi}(a_{k'})$.
Moreover, since $x_{k'} \neq 0$ is the coefficient of $\lambda^{d-\cs_{\pi}(a_{k'})}$ in 
$\eta_{k'}$,
we have
\[
\mdeg(\hat\alpha) \leq d- \cs_{\pi}(a_{k'}).
\]
This gives
\[
\mdeg (\theta^\Lambda(\hat\alpha))
-\mdeg(\hat\alpha) =
d-\mdeg(\hat\alpha) \geq \cs_{\pi}(a_{k'}).
\]
\end{proof}

\def\proofname{Proof of \Cref{compare}}
\begin{proof}
Assume that $r_{-\infty}(Y) < \Gamma_{-Y}(1)$,
and then we can take $r \in \R_Y$ with $r_{-\infty}(Y) < r < \Gamma_{-Y}(1)$.
For such $r$ and any sufficiently small perturbation $\pi$, we have $[\theta^{[-\infty,r]}_Y]\neq 0$,
and hence there exists a cycle 
$\alpha =\sum_{1 \leq k \leq n} a_{k} \otimes x_k$
in $CI^{[-\infty,r]}_1(Y) \otimes \q$ 
with $\theta^{[-\infty,r]}_Y(\alpha) \neq 0$.
Therefore, it follows from Lemmas~\ref{lambda cycle}
and \ref{lambda theta}
that there exists a sequence $\{\pi_l\}_{l \in \z_{>0}}$
of perturbations with $\|\pi_l\| \to 0$ ($l \to \infty$)
such that for each $\pi_l$, $CI^\Lambda_1(Y)$ has a cycle $\widetilde{\alpha_l}$
with $\theta^{\Lambda}(\widetilde{\alpha_l}) \neq 0$ and
$\mdeg(\theta^\Lambda(\widetilde{\alpha_l})) -\mdeg(\widetilde{\alpha_l}) <r$.
This gives the inequality
\[
\Gamma_{-Y}(1) = \lim_{\|\pi\| \to 0}  \left(
\inf_{\substack{\alpha \in CI_1^{\Lambda}(Y),\, \partial^{\Lambda}(\alpha)=0, \\ \theta^\Lambda (\alpha) \neq 0 }}
\left\{
\mdeg(\theta^\Lambda (\alpha))- \mdeg (\alpha)   
\right\}
\right) \leq r,
\]
a contradiction. Hence we have $r_{-\infty}(Y) \geq \Gamma_{-Y}(1)$.

Conversely, assume that $\Gamma_{-Y}(1) < r_{-\infty}(Y)$,
and then we can take $r \in \R_Y$ with 
$\Gamma_{-Y}(1) < r < r_{-\infty}(Y)$.
Then, for any sufficiently small perturbation compatible with $r$, 
there exists a cycle $\hat\alpha \in CI^\Lambda_1(Y)$ with $\theta^{\Lambda}(\hat\alpha) \neq 0$ and $\mdeg(\theta^\Lambda (\hat\alpha))- \mdeg (\hat\alpha)<r$.
Then, by Lemma~\ref{lambda replace},
we obtain a cycle $\alpha= \sum_{1 \leq k \leq n} a_k \otimes x_k$
in $CI_1(Y) \otimes \q$ and a number $k' \in \{1, \ldots, n\}$ which satisfies
\begin{enumerate}
\item $x_{k'} \neq 0$ and $\cs_{\pi}(a_{k'}) = \max\{ \cs_{\pi}(a_k) \mid x_k \neq 0\}$,
\item $\theta^{[-\infty, \infty]}_Y(\alpha) \neq 0$, and
\item $\mdeg(\theta^\Lambda(\widetilde{\alpha})) - \mdeg(\widetilde{\alpha}) 
= \cs_{\pi}(a_{k'}) < r$.
\end{enumerate}
Here, (1) and (3) imply that $\alpha \in CI^{[-\infty, r]}_1(Y) \otimes \q$,
and hence (2) implies that $[\theta^{[-\infty,r]}_Y] \neq 0$.
This gives the inequality $r \geq r_{-\infty}(Y)$, a contradiction.
Hence we have $\Gamma_Y(1) \geq r_{-\infty}(Y)$.
\end{proof}
\def\proofname{Proof}

\subsection{Consequences}
Here, we show corollaries of \Cref{compare}.
\setcounter{section}{1}
\setcounter{thm}{2}
\begin{cor}
The inequality $r_{-\infty}(Y) < \infty$ holds if and only if $h(Y)<0$.
In particular, if $h(Y)<0$, then $r_s(Y)$ has a finite value for any 
$s \in [-\infty, 0]$.
\end{cor} 
\begin{proof}
It is shown in \cite{D18} that  the inequality $h(Y)<0$ holds if and only if $
\Gamma_{-Y}(1)  < \infty$. 
This fact and  \Cref{compare} give the conclusion. 
\end{proof}
Recall that $\Sigma(a_1, a_2, \ldots, a_n)$ denotes the Seifert homology 3-sphere corresponding to a tuple of pairwise coprime integers $(a_1, a_2, \ldots, a_n)$
and $R(a_1, a_2 ,\ldots, a_n)$ is an odd integer introduced by Fintushel-Stern \cite{FS85}.
\begin{cor}
If  $R(a_1, a_2 ,\ldots, a_n)>0$, then for any $s \in [-\infty, 0]$,
the equalities
\[
r_s(-\Sigma(a_1,a_2, \ldots, a_n))= \frac{1}{4a_1a_2\cdots a_n}\ 
\text{and}\ 
r_s(\Sigma(a_1,a_2, \ldots, a_n))= \infty
\]
hold.
\end{cor}
\setcounter{section}{4}
\setcounter{thm}{5}

\begin{proof}
By using \Cref{different s} and \Cref{compare}, we obtain 
\begin{align*}
r_s(-\Sigma(a_1, a_2, \ldots, a_n)) 
 &\leq r_{-\infty}(-\Sigma(a_1, a_2, \ldots, a_n)) \\
 &= \Gamma_{\Sigma(a_1, a_2, \ldots, a_n)} (1) = \frac{1}{4a_1 a_2 \cdots a_n}.
\end{align*}
Moreover, it is shown in \cite{Fu90, FS90}
 that 
$$
\min (\Lambda_{-\Sigma(a_1, a_2, \ldots, a_n)} \cap \R_{>0}) =
\frac{1}{4a_1 a_2 \cdots a_n}.$$ 
This gives the first equality in \Cref{general Seifert}. 
The second equality follows from \Cref{obstruct bounding} and the fact that
$\Sigma(a_1, a_2, \ldots, a_n)$ bounds a negative definite 4-manifold.
\end{proof}

\begin{cor}\label{Seifert} For any positive coprime integers $p,q>1$ and 
positive integer $k$,
we have
\[
r_s(-\Sigma(p,q,pqk-1)) = \frac{1}{4pq(pqk-1)} 
\]
and
\[
r_s(\Sigma(p,q,pqk-1)) = \infty.
\]
\end{cor}

\section{Applications}\label{Applications}

In this section, we prove the theorems stated in Section~\ref{section 1.2}.

\subsection{Useful lemmas}
We first give several lemmas which are useful for computing $r_0$.

\begin{lem}
\label{addition}
For any homology $3$-spheres $Y_1$ and $Y_2$,
if $r_0(-Y_1)=r_0(-Y_2)= \infty$, then
$$
r_0(Y_1 \# Y_2) = \min\{r_0(Y_1), r_0(Y_2)\}
$$
and 
$$
r_0(-Y_1 \# -Y_2)= \infty.
$$
\end{lem}
\begin{proof}
The equality
$r_0(-Y_1 \# -Y_2)= \infty$ and
the inequality $r_0(Y_1 \# Y_2) \geq \min\{r_0(Y_1), r_0(Y_2)\}$
immediately follow from Theorem \ref{conn}.
To prove the inequality $r_0(Y_1 \# Y_2) \leq \min\{r_0(Y_1), r_0(Y_2)\}$,
we first consider $r_0(Y_1\# Y_2 \# -Y_2)$.
Then, by \Cref{invariance} and Theorem \ref{conn},
we have
$$
r_0(Y_1)= r_0(Y_1 \# Y_2 \# -Y_2)
\geq \min \{ r_0(Y_1\# Y_2), r_0(-Y_2)\}.
$$
Here, since $r_0(Y_1\# Y_2) \leq r_0(-Y_2) = \infty$,
we obtain 
$
r_0(Y_1)
\geq r_0(Y_1\# Y_2).
$
Similarly, we have
$
r_0(Y_2)
\geq r_0(Y_1\# Y_2).
$
This completes the proof.
\end{proof}

\begin{cor}
\label{order}
Suppose that a homology $3$-sphere $Y$ satisfies $r_0(Y)< \infty$ and $r_0(-Y)= \infty$.
Then, for any $n \in \z_{>0}$, we have
$$
r_0(n[Y])=r_0(Y) < \infty
$$
and 
$$
r_0(-n[Y]) = \infty.
$$
In particular, $Y$ has infinite order in $\Theta^3_\z$.
\end{cor}
\begin{proof}
By induction on $n$, this corollary directly follows from \Cref{addition}.
\end{proof}

\begin{lem}
\label{subtraction}
For any homology $3$-spheres $Y_1$ and $Y_2$,
if the inequality $r_0(Y_1) < \min\{ r_0(Y_2), r_0(-Y_2)\}$ holds, then
$r_0(Y_1 \# -Y_2) = r_0(Y_1)$. 
\end{lem}

\begin{proof}
By applying Theorem~\ref{conn} to $Y_1 \# -Y_2$ and
$Y_1 \# -Y_2 \# Y_2$, we have
$$
r_0(Y_1 \# -Y_2) \geq \min \{ r_0(Y_1), r_0(-Y_2) \} = r_0(Y_1)
$$
and
$$
r_0(Y_1)=r_0(Y_1 \# -Y_2 \# Y_2) \geq \min \{ r_0(Y_1 \# - Y_2), r_0(Y_2)\}.
$$
Here, since $r_0(Y_1) < r_0(Y_2)$, we obtain 
$
r_0(Y_1) \geq r_0(Y_1 \# - Y_2).
$
\end{proof}

\begin{thm}
\label{combination}
Suppose that a linear combination
$
\sum^m_{k=1}n_k[Y_k] \in \Theta^3_{\z}
$
satisfies
\begin{itemize}
\item 
$\displaystyle r_0(Y_m) < \min_{1 \leq k <m} \{ r_0(Y_k), r_0(-Y_k) \}$,
\item
$r_0(-Y_m) = \infty$,
and
\item
$n_m>0$.
\end{itemize}
Then $r_0(\sum^m_{k=1}n_k[Y_k])=r_0(Y_m)< \infty$.
\end{thm}

\begin{proof}
By assumption, it follows from \Cref{order}
that $r_0(n_m[Y_m]) = r_0(Y_m)$.
Moreover, \Cref{conn} implies that
$$
\min \Big\{ r_0\big(\sum^{m-1}_{k=1}(-n_k)[Y_k]\big), 
r_0 \big(-\sum^{m-1}_{k=1}(-n_k)[Y_k] \big) \Big\} 
$$
$$
\geq
\min_{1 \leq k < m} \{ r_0(Y_k), r_0(-Y_k) \}
>  r_0(n_m[Y_m]).
$$
Therefore, by \Cref{subtraction}, we have
$$
r_0(\sum^m_{k=1}n_k[Y_k]) = 
r_0\big(n_m [Y_m] - \sum^{m-1}_{k=1}(-n_k)[Y_k] \big) =r_0(Y_m).
$$
\end{proof}

For a homology 3-sphere $Y$, set
\[
\varepsilon_1(Y) := \inf (\Lambda_Y \cap \R_{>0}) 
\text{ and }
\varepsilon_2(Y) := \min \{ \varepsilon_1(Y), \varepsilon_1(-Y) \}.
\]
\Cref{combination} is regarded as a generalization of the following theorem
due to Furuta.

\begin{cor}[\text{\cite[Theorem 6.1]{Fu90}}]
Let $Y_1, \ldots, Y_m$ be homology $3$-spheres with $\varepsilon_2(Y_i)>0$
\textup{(}$i=1,2, \ldots, m$\textup{)}.
Let $Y_0 = \Sigma(a_1, a_2, \ldots, a_n)$ be a Seifert homology $3$-sphere
such that $R(a_1, a_2, \ldots, a_n)>0$ and $a_1 a_2 \cdots a_n > \varepsilon_2(Y_i)^{-1}$
\textup{(}$i=1,2, \ldots, m$\textup{)}.
Then
\[
\z[Y_0] \cap (\z[Y_1]+\z[Y_2]+ \cdots + \z[Y_m])=0 \ \text{in}\ \Theta^3_{\z}.
\]
\end{cor}

\begin{proof}
Note that since $r_0(Y) \in \Lambda_Y \cap \R_{>0}$ in general, 
we have
$$
\min\{ r_0(Y_i), r_0(-Y_i)\} \geq \varepsilon_2(Y_i) > 
\frac{1}{a_1 a_2 \cdots a_n} = r_0(-Y_0)
$$
for any $i = 1,2, \ldots, m$, where the last equality follows from \Cref{general Seifert}.
Therefore, if we have an equality
$$n_0 [-Y_0] = n_1 [Y_1] + n_2[Y_2] + \cdots + n_m[Y_m]$$
and $n_0 > 0$, then it follows from \Cref{combination} that
$$
\frac{1}{a_1 a_2 \cdots a_n} = r_0(-Y_0) = r_0\big(n_0[-Y_0]-\sum^m_{i=1} n_i [Y_i]\big)
=r_0(S^3) = \infty,
$$
a contradiction.
\end{proof}

\begin{cor}
\label{linear independence}
Let $\{ Y_k\}_{k=1}^{\infty}$ be a sequence of homology $3$-spheres satisfying 
\[
 \infty > r_0(Y_1) > r_0(Y_2) > \cdots 
\]
and
\[
 \infty = r_0(-Y_1) = r_0(-Y_2) = \cdots.
\]
Then, the $Y_k$'s are linearly independent in $\Theta^3_\z$.
\end{cor}

\begin{proof}
Let $\sum^{m}_{k=1}n_k[Y_{k}]$ be a linear combination with $n_m \neq 0$.
Without loss of generality, we can assume $n_m > 0$.
Then
$\sum^{m}_{k=1} n_k[Y_{k}]$ satisfies
the hypothesis of \Cref{combination},
and hence 
$$
r_0(\sum^{m}_{k=1} n_k[Y_{k}]) =r_0(Y_m) < \infty = r_0(S^3).
$$
This implies that 
$
\sum^{m}_{k=1} n_k[Y_{k}] \neq 0.
$
\end{proof}

\subsection{Homology 3-spheres with no definite bounding}
\label{section 5.2}
In this subsection, we prove Theorem~\ref{definite bounding}.

\setcounter{section}{1}
\setcounter{thm}{4}
\begin{thm}
There exist infinitely many homology $3$-spheres $\{Y_k\}_{k=1}^{\infty}$ which cannot bound any definite 4-manifold. Moreover, we can take such $Y_k$ so that the $Y_k$'s are linearly independent in $\Theta^3_\z$.
\end{thm}
\setcounter{section}{5}
\setcounter{thm}{6}

For any $k \in \z_{>0}$, let $K_k$ be the knot depicted in Figure~\ref{K_k}.
Note that $K_k$ is the 2-bridge knot corresponding to the rational number 
$\frac{2}{4k-1}$. In particular, the first two knots $K_1$ and $K_2$
are the left-handed trefoil $3_1$ and the knot $5_2$ in Rolfsen's knot table \cite{Ro90}
respectively.

\begin{figure}[htbp]
\begin{center}
\hspace{5mm}
\includegraphics[scale= 0.5]{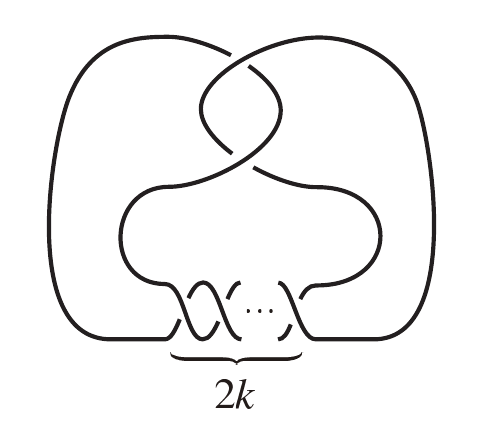}
\vspace{-4mm}
\caption{ The knot $K_k$. \label{K_k}}
\end{center}
\end{figure}

\begin{lem}
\label{twist knots}
For any $k \in \z_{> 0}$, 
we have a diffeomorphism
$$
\Sigma(2,3,6k-1) \cong S^3_{-1}(K_k).
$$
\end{lem}

\begin{proof}
It is well-known that 
$\Sigma(2,3,6k-1) \cong S^3_{-1/k}(3_1)$,
and Figure~\ref{-1 surgery} proves 
$S^3_{-1/k}(3_1)=S^3_{-1/k}(K_1) \cong S^3_{-1}(K_k)$.
\end{proof}

\begin{figure}[htbp]
\begin{center}
\includegraphics[scale= 0.4]{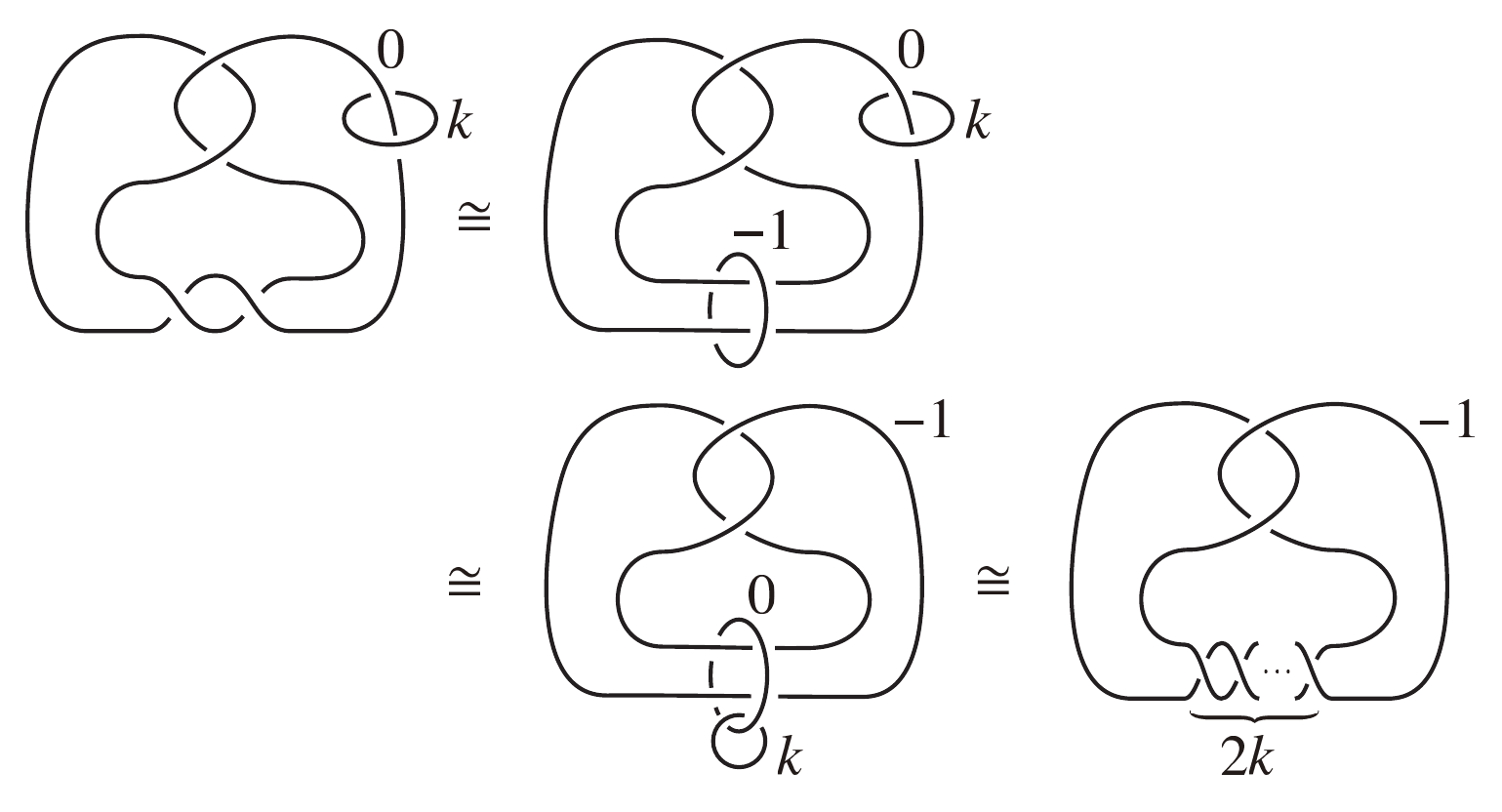}
\vspace{-2mm}
\caption{Kirby calculus for $S^3_{-1/k}(K_1) \cong S^3_{-1}(K_k)$. \label{-1 surgery}}
\end{center}
\end{figure}

While the Fr{\o}yshov invariant $h$ is hard to compute in general, 
we have a nice estimate for $(-1)$-surgeries on genus one knots. 
\begin{lem}[\text{\cite[Lemma 9]{Fr02}}]
\label{genus 1}
For any genus one knot $K$, we have
$$
0 \leq h(S^3_{-1}(K)) \leq 1.
$$ 
\end{lem}

\begin{lem}
\label{Froyshov=1}
For any $k \in \z_{> 0}$, 
we have $h(\Sigma(2,3,6k-1))=1$.
\end{lem}

\begin{proof}
The inequality $h(\Sigma(2,3,6k-1)) = -h(-\Sigma(2,3,6k-1))>0$ follows from Corollaries~\ref{Froyshov} and \ref{Seifert}.
The inequality $h(\Sigma(2,3,6k-1)) \leq 1$ follows from Lemmas~\ref{twist knots} and \ref{genus 1}
and the fact that $K_k$ has genus one for any $k \in \z_{>0}$.
\end{proof}

Now we prove one of the main theorems.

\def\proofname{Proof of \Cref{definite bounding}}
\begin{proof}
We put $Y_k:= 2\Sigma(2,3,5) \# (-\Sigma(2,3,6k+5))$ for any $k \in\z_{>0}$. 
Then it follows from Corollaries~\ref{Seifert} and \ref{order} and \Cref{subtraction}
that $r_0(Y_k)= \frac{1}{24(6k+5)}< \infty$. This fact and \Cref{obstruct bounding}
imply that $Y_k$ cannot bound any negative definite 4-manifold.

Next, since the invariant $h$ is a group homomorphism, 
\Cref{Froyshov=1} gives 
$h(Y_k)=1$.  
This fact and \Cref{Froyshov} imply
that $r_{-\infty}(-Y_k)< \infty$, and hence
it follows from \Cref{obstruct bounding} 
that $Y_k$ cannot bound any positive definite 4-manifold.

The linear independence of $\{Y_k\}^{\infty}_{k=1}$
follows from that of 
$\{\Sigma(2,3,6k-1)\}_{k=1}^{\infty}$.
\end{proof}
\def\proofname{Proof}

\subsection{Linear independence of $1/n$-surgeries}
\label{section 5.3}

In this subsection, we prove the theorems stated in Section~\ref{section 1.2.2}.

\setcounter{section}{1}
\setcounter{thm}{7}
\begin{thm} 
For any knot $K$ in $S^3$, if $h(S^3_1(K))<0$, then $\{S^3_{1/n}(K)\}_{n=1}^{\infty}$
are linearly independent in $\Theta^3_{\z}$.
\end{thm}

\begin{cor}
For any $k \in \z_{>0}$, the homology $3$-spheres $\{S^3_{1/n}(K_k)\}_{n=1}^{\infty}$ are linearly independent in $\Theta^3_{\z}$.
\end{cor}

\begin{cor}
For any knot $K$ in $S^3$ and odd integer $q \geq 3$,
the homology $3$-spheres
$\{S^3_{1/n}(K_{2,q})\}_{n=1}^{\infty}$ are linearly independent in $\Theta^3_{\z}$.
\end{cor}
\setcounter{section}{5}
\setcounter{thm}{9}
Note that \Cref{K_k indep} immediately follows from \Cref{knot surgery}, Lemmas~\ref{twist knots} and \ref{Froyshov=1}.
On the other hand, \Cref{cable indep} follows from
\Cref{knot surgery} and the following two facts.
{Note that the intersection form of any spin 4-manifold whose boundary is a homology 3-sphere is even, and hence it is non-diagonalizable, namely the intersection form is not isomorphic to $\bigoplus (\pm 1)$.}

\begin{thm}[\text{\cite[Theorem 3]{Fr02}}]
If $Y$ bounds a positive definite $4$-manifold with non-diagonalizable intersection form,
then $h(Y)<0$.
\end{thm}

\begin{thm}[\text{\cite[Theorem 1.7]{Sa17}}]
For any knot $K$ and odd $q \geq 3$,
the homology $3$-sphere $S^3_1(K_{2,q})$ bounds a positive definite spin $4$-manifold.
\end{thm}
The proof of \Cref{knot surgery} is obtained by combining \Cref{linear independence}
with the following theorem.

\begin{thm}
\label{1/n r_s}
For any knot $K$ in $S^3$, if $h(S^3_1(K))<0$,  then for any $s$, we have
$$
\infty > r_s(S^3_1(K)) > r_s(S^3_{1/2}(K)) > \cdots
$$
and
$$
\infty = r_s(-S^3_1(K)) = r_s(-S^3_{1/2}(K)) = \cdots.
$$
\end{thm}

\begin{proof}
\def\proofname{Proof}

For any knot $K$ and $n \in \z_{>0}$, since
$S^3_{1/n}(K)$ bounds a positive definite 4-manifold, 
we have
$$
\infty = r_s(-S^3_1(K)) = r_s(-S^3_{1/2}(K)) = \cdots.
$$
Suppose that $K$ satisfies $h(S^3_1(K))<0$.
Then \Cref{Froyshov} gives $r_s(S^3_1(K)) < \infty$.
For any $n \in \z_{>0}$, let $W_n$ be the cobordism given by the relative
Kirby diagram in Figure~\ref{W_n}. It is easy to see that 
$\partial W_n = S^3_{1/(n+1)}(K) \amalg -S^3_{1/n}(K)$.

\begin{figure}[htbp]
\begin{center}
\hspace{5mm}
\includegraphics[scale= 0.5]{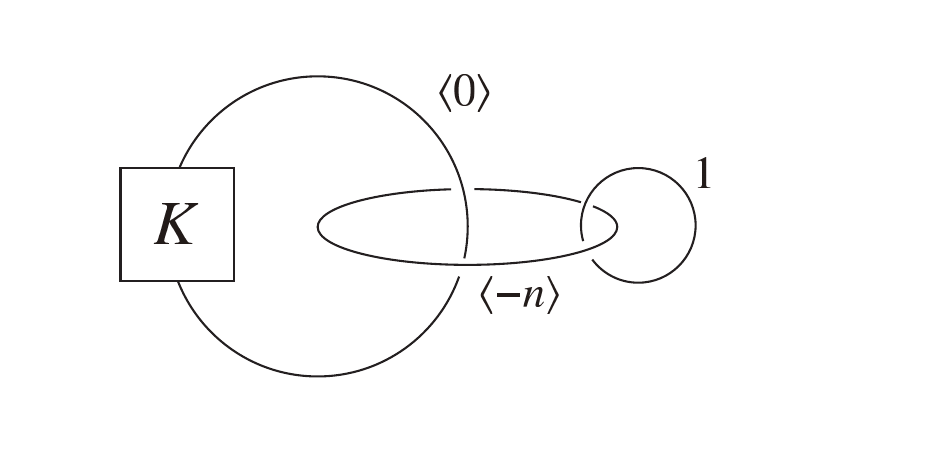}
\vspace{-4mm}
\caption{The cobordism $W_n$. \label{W_n}}
\end{center}
\end{figure}

\begin{claim}
\label{W_n pos}
The cobordism $W_n$ is positive definite.
\end{claim}
\begin{proof}
Let $X_n$ be a 4-manifold given by 
the Kirby diagrams in Figure~\ref{X_n},
and $X'_n$ a 4-dimensional submanifold of $X_n$ obtained by
attaching 2-handles along the 2-component sublink in the left diagram
of Figure~\ref{X_n} whose framing is $(0, -n)$.
Then we have 
the diffeomorphism $X_n \cong X'_n \cup_{S^3_{1/n}(K)} W_n$.
For a 4-manifold $M$, let $b^+_2(M)$ (resp.\ $b^-_2(M)$) denote
the number of positive (resp.\ negative) eigenvalues of the
intersection form of $M$.
Then it is easy to check that $b^+_2(X_n)=2$, $b^-_2(X_n)=1$ and 
$b^+_2(X'_n)=b^-_2(X'_n)=1$.
These imply that $b_2(W_n)=b^+_2(W_n)=1$ and $b^-_2(W_n)=0$.
\end{proof} 

\begin{figure}[htbp]
\begin{center}
\hspace{5mm}
\includegraphics[scale= 0.5]{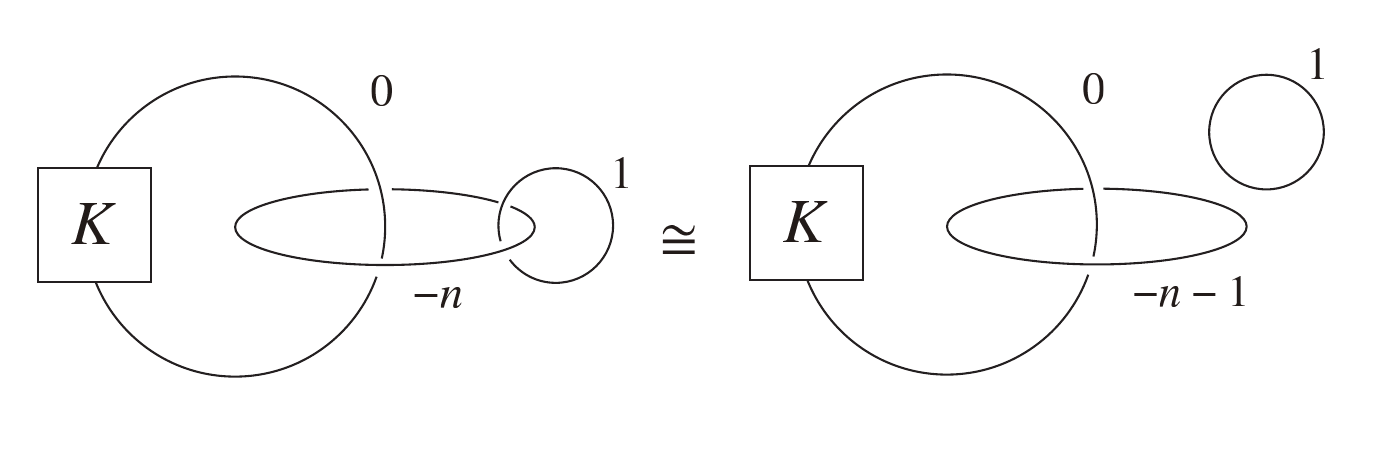}
\vspace{-4mm}
\caption{The 4-manifold $X_n$. \label{X_n}}
\end{center}
\end{figure}

\begin{claim}
\label{W_n pi_1}
The cobordism $W_n$ is simply connected.
\end{claim}

\begin{proof}
Suppose that the number of crossings in the diagram of Figure~\ref{pres1}
is $m+1$.
Then, $\pi_1(S^3_{1/n}(K))$ has the presentation
$$
\left\langle
\begin{array}{l|l}
\  & x_{i+1} = x^{\varepsilon_i}_{k_i} x_i x^{-\varepsilon_i}_{k_i}
\ (i=1, \ldots, m-1),\\
x_1, x_2, \ldots, x_m, y & x_m y = y  x_1,\ y x_1 = x_1 y ,\\
\ & \lambda=1,\ x_1 y^{-n}=1
\end{array}
\right\rangle,
$$
where
\begin{itemize}
\item
the labels $x_1$, $x_m$ and $y$ are associated as shown in
Figure \ref{pres1},
\item
$\varepsilon_i \in \{ \pm 1\}$ and $k_i \in \{1, \ldots, m\}$
($i=1, \ldots, m$), and
\item
$\lambda$ is a word corresponding to a longitude of $K$ with framing 0.
\end{itemize}
Moreover, since the attaching sphere of the (unique) 2-handle of $W_n$ is homotopic to $y$,
$\pi_1(W_n)$ has the presentation
$$
\left\langle
\begin{array}{l|l}
\  & x_{i+1} = x^{\varepsilon_i}_{k_i} x_i x^{-\varepsilon_i}_{k_i}
\ (i=1, \ldots, m-1),\\
x_1, x_2, \ldots, x_m, y & x_m y = y  x_1,\ y x_1 = x_1 y ,\\
\ & \lambda=1,\ x_1 y^{-n}=1,\ y=1
\end{array}
\right\rangle
$$
$$
\cong
\left\langle
\begin{array}{l|l}
x_1, x_2, \ldots, x_m  & 
\begin{array}{l}
x_{i+1} = x^{\varepsilon_i}_{k_i} x_i x^{-\varepsilon_i}_{k_i}
\ (i=1, \ldots, m-1),
\vspace{1mm}\\
x_m = x_1 =1,\ \lambda|_{y=1}=1
\end{array}
\end{array}
\right\rangle,
$$
where $\lambda|_{y=1}$ is the word obtained from $\lambda$ by
substituting 1 for $y$.
Now, by induction, we see that the relations $x_1 = x_2 = \cdots = x_m =1$ hold,
and hence $\pi_1(W_n)=1$.
\end{proof}

\begin{figure}[htbp]
\begin{center}
\hspace{5mm}
\includegraphics[scale= 0.5]{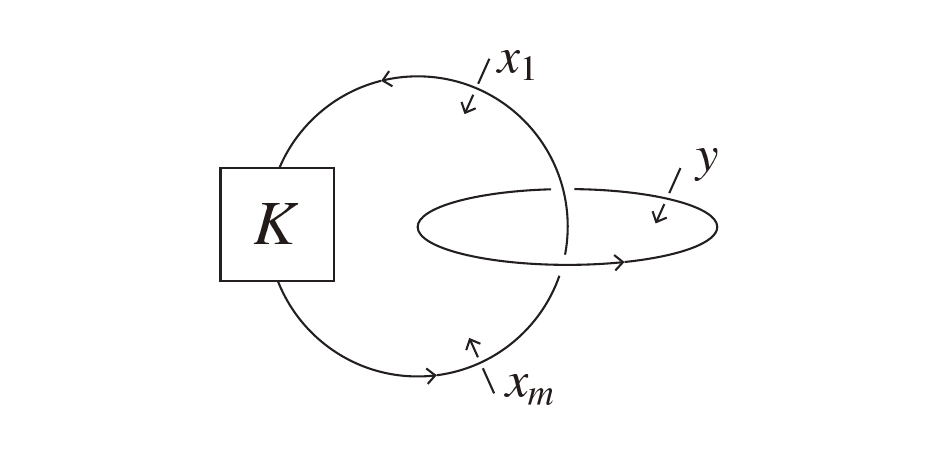}
\vspace{-4mm}
\caption{A surgery link for $S^3_{1/n}(K)$. \label{pres1}}
\end{center}
\end{figure}

By Claims~\ref{W_n pos} and \ref{W_n pi_1}, we can apply \Cref{pi1=1}
to all $-W_n$ ($n=1, 2, \ldots$), and obtain
$$
\infty >r_s(S^3_1(K)) > r_s(S^3_{1/2}(K)) > \cdots.
$$
\end{proof}

\subsection{Linear independence of Whitehead doubles}
In this subsection, we prove \Cref{np+q}.

\setcounter{section}{1}
\setcounter{thm}{11}
\begin{thm}
For any coprime integers $p,q>1$,
the Whitehead doubles $\{ D(T_{p,np+q})\}^{\infty}_{n=0}$ are linearly independent
in $\mathcal{C}$.
\end{thm}
\setcounter{section}{5}
\setcounter{thm}{14}

Let $\Theta^3_{\q}$ denote the rational homology cobordism group of rational
homology 3-spheres.
Then we have a natural group homomorphism
$$
\Theta^3_\z \to \Theta^3_\q, \ [Y] \mapsto [Y]_\q,
$$
where $[Y]_\q$ is the rational homology cobordism class of $Y$.
We say that the rational homology 3-spheres $\{Y_k\}_{k=1}^\infty$ are
\emph{linearly independent in} $\Theta^3_\q$
if $\{[Y_k]_\q \}_{k=1}^\infty$ are linearly independent in $\Theta^3_\q$.
Then the invariance of $r_s$ and \Cref{linear independence}
are naturally generalized in the following sense.

\begin{thm}
\label{linear independence q}
For any homology $3$-sphere $Y$ and $s \in [-\infty, 0]$,
the value $r_s(Y)$ is invariant under rational homology cobordism.
Moreover, if a sequence $\{Y_k\}_{k=1}^\infty$ of homology $3$-spheres
satisfies the assumption of \Cref{linear independence},
then the $Y_k$'s are linearly independent in $\Theta^3_\q$.
\end{thm}

Next, let $K$ be an oriented knot and $\Sigma(K)$ the double branched cover of $S^3$ over $K$. Then it is known that the map
$$
\mathcal{C} \to \Theta^3_\q, \ [K] \mapsto [\Sigma(K)]_\q
$$
is well-defined and a group homomorphism.
Moreover, for Whitehead doubles,
it is also known that
$\Sigma(D(K)) \cong S^3_{1/2}(K \# -K)$,
where $-K$ is orientation-reversed $K$.
In particular, $\Sigma(D(K))$ is a homology 3-sphere and
bounds a positive definite 4-manifold. Hence
$r_0(-\Sigma(D(K)))=\infty$ for any $K$.
These arguments imply the following.
\begin{lem}\label{linear independence D}
For a sequence $\{K_n\}_{n=1}^\infty$ of oriented knots,
if the homology 3-spheres $\{\Sigma(D(K_n))\}_{n=1}^\infty$ satisfy
$$
\infty > r_0(\Sigma(D(K_1))) > r_0(\Sigma(D(K_2))) > \cdots,
$$
then the Whitehead doubles $\{D(K_n)\}_{n=1}^\infty$
are linearly independent in $\mathcal{C}$.
\end{lem}

For any coprime integers $p,q>1$, we abbreviate $D(T_{p,q})$ to $D_{p,q}$.
The proof of \Cref{np+q} is obtained by combining \Cref{linear independence D}
with the following theorem.

\begin{thm}\label{np+q r_s}
For any coprime integers $p,q>1$,
the inequalities
\[
\frac{1}{4pq(2pq-1)} \geq r_s(\Sigma(D_{p,q})) > r_s(\Sigma(D_{p,p+q}))
> r_s(\Sigma(D_{p,2p+q})) > \cdots
\] 
hold.
\end{thm}

As another corollary of \Cref{np+q r_s},
we also have the following family of linearly independent elements.

\begin{cor}
\label{Euclid}
Let $a$ and $b$ be coprime integers with $1<a<b$ and
\[
b = q_0a+r_0,\ 
a = q_1r_0+r_1,\dots,
r_{N-1} = q_{N+1}r_N+1
\]
the sequence derived from the Euclidean algorithm.
Then we have
\begin{align*}
\infty 
&> r_s(\Sigma(D_{r_N, r_N+1})) > r_s(\Sigma(D_{r_N, 2r_N+1})) > \cdots
> r_s(\Sigma(D_{r_N, r_{N-1}})) \\
&> r_s(\Sigma(D_{r_{N-1}, r_{N-1} + r_N})) > \cdots >  r_s(\Sigma(D_{r_{N-1}, r_{N-2}})) \\
&> \cdots \\
&> r_s(\Sigma(D_{a, a + r_0}))  > \cdots >  r_s(\Sigma(D_{a, b})).
\end{align*}
In particular, all of these Whitehead doubles are linearly independent in $\mathcal{C}$.
\end{cor}

Now we start to prove \Cref{np+q r_s}. 
Let $K$ be an oriented knot, $\mathcal{D}$ a diagram of $K$
and $x_1, \ldots, x_m$ the arcs of $\mathcal{D}$.  
Fix a base point in $S^3 \setminus K$,
and associate a loop in $S^3 \setminus K$
to each $x_i$
in the usual way.
(For instance, see \cite[Section 3.D]{Ro90}.)
Then, for any $n \in \z$, 
we have a presentation of $\pi_1(S^3_{1/n}(K))$
in the form
$$
\langle x_1, \ldots, x_m \mid R \cup \{ \lambda^n x_1 =1 \} \rangle,
$$
where $R$ is the set of relations induced from the crossings of $\mathcal{D}$
(in the same way as the Wirtinger presentation),
and $\lambda$ is a word corresponding to a longitude of $K$
with framing 0. (In particular, $\lambda$ is in the commutator subgroup of 
$\pi_1(S^3_{1/n}(K))$.)

Next, we consider the \emph{positive crossing change} at a \emph{positive crossing} $c$, which is a deformation of $\mathcal{D}$ shown in Figure~\ref{pos cc}.
Then we denote the labels of the arcs around $c$ by
$x_{i_c}$, $x_{j_c}$ and $x_{j'_c}$ as shown in Figure~\ref{pos cc}.

\begin{figure}[htbp]
\begin{center}
\includegraphics[scale= 0.5]{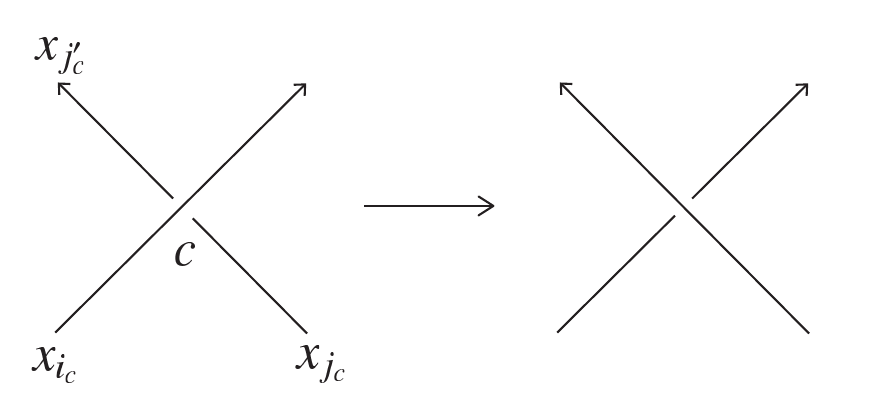}
\vspace{-4mm}
\caption{A positive crossing change. \label{pos cc}}
\end{center}
\end{figure}

\begin{lem}
\label{cc cob}
Suppose that $\mathcal{D}$ is deformed into a diagram $\mathcal{D}'$  of a knot $K'$
by performing positive crossing changes at
crossings $c_1, c_2, \ldots, c_l$ respectively.
Then there exists a negative definite cobordism $W$ such that
$\partial W = - S^3_{1/n}(K) \amalg S^3_{1/n}(K')$ and
\[
\pi_1(W) \cong 
\left\langle
x_1, \ldots, x_m \Bigm| R \cup \{ \lambda^n x_1=1\} \cup \{ x_{i_{c_k}}=x_{j_{c_k}} \}_{k=1}^l
\right\rangle.
\]
\end{lem}

\begin{proof}
We make a relative Kirby diagram $\langle \mathcal{D} \rangle$ from $\mathcal{D}$
in the following way:
\begin{itemize}
\item
Replace a neighborhood of each crossing $c_k$ with the picture
shown in Figure~\ref{cc handle}. (Then, each component
except for the original one has a framing.)
\item Associate the framing $\langle \frac{1}{n} \rangle$
to the original component.
\end{itemize}
Then $\langle \mathcal{D} \rangle$ is a diagram for a cobordism $W$
obtained by attaching $l$ copies of 2-handles to $S^3_{1/n}(K) \times [0,1]$.
Moreover, we can verify that
\begin{itemize}
\item
$\partial W = -S^3_{1/n}(K) \amalg S^3_{1/n}(K')$,
\item
the intersection form of $W$ is 
isomorphic to $\bigoplus^l_{k=1} (-1)$, and
\item
as a loop, the attaching sphere of the 2-handle near $c_k$ is written
by $x^{-1}_{j_{c_k}}x_{i_{c_k}}$.
\end{itemize}
This completes the proof.
\end{proof}

\begin{figure}[htbp]
\begin{center}
\hspace{5mm}
\includegraphics[scale= 0.5]{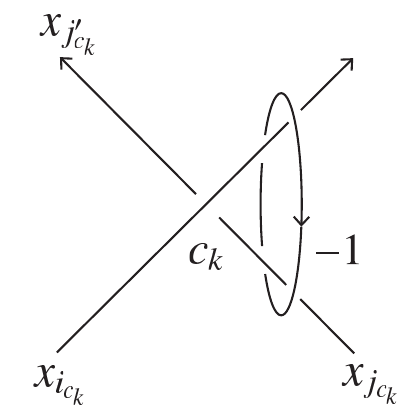}
\vspace{-4mm}
\caption{The 2-handle attached to $S^3_{1/n}(K) \times [0,1]$ 
near $c_k$. \label{cc handle}}
\end{center}
\end{figure}

Next, for any coprime $p,q >1$, we consider 
the homology 3-sphere
$\Sigma(D_{p,q})$.

\begin{lem}
\label{$D_{p,q}$}
$r_s(\Sigma(D_{p,q})) \leq \frac{1}{4pq(2pq-1)}$.
\end{lem}

\begin{proof}
Note that $T_{p,q}$ has a diagram $\overline{\Delta_p^q}$ with only positive crossings.
(Indeed, the closure of the braid $\Delta_p^q=(\sigma_1 \sigma_2 \cdots \sigma_{p-1})^q$ with $p$ strands is such a diagram for $T_{p,q}$.)
For any knot diagram, there exist
finitely many crossings such that after crossing changes at the crossings, 
the resulting diagram describes the unknot.
As a consequence, we have finitely many positive crossings of $\overline{\Delta_p^q}$
such that after positive crossing changes at the crossings, the resulting diagram
$(\overline{\Delta_p^q})^U$ is
for the unknot.
Now, considering the connected sum of two $\overline{\Delta_p^q}$'s, 
we have finitely many positive crossings of $\overline{\Delta_p^q} \# \overline{\Delta_p^q}$
such that after positive crossing changes at the crossings, we have
the diagram $\overline{\Delta_p^q} \# (\overline{\Delta_p^q})^U$.
Hence, by applying \Cref{cc cob}, we have a negative definite cobordism with
boundary $-S^3_{1/2}(T_{p,q} \# T_{p,q}) \amalg S^3_{1/2}(T_{p,q})$.
Therefore, it follows from \Cref{cobneq} and \Cref{Seifert}
that
\begin{align*}
r_s(\Sigma(D_{p,q}))
&= r_s(S^3_{1/2}(T_{p,q} \# T_{p,q})) \\
&\leq r_s(S^3_{1/2}(T_{p,q})) \\
&= r_s(\Sigma(p,q,2pq-1)) = \frac{1}{4pq(2pq-1)}.
\end{align*}
\end{proof}

Now, let us consider a concrete diagram of $T_{p,p+q} \# T_{p,p+q}$, which
is depicted in Figure~\ref{braid} and denoted by $\mathcal{D}$.
Here $\Delta_H$ is the braid
$$
(\sigma_1\sigma_2 \cdots \sigma_{p-1})(\sigma_1 \sigma_2 \cdots \sigma_{p-2})
\cdots (\sigma_1 \sigma_2) \sigma_1,
$$
which is often called the \emph{half-twist}.
In addition, we associate the labels $\{x_k\}_{k=1}^{2p}$ to arcs in $D$ as
shown in Figure~\ref{braid}.

\begin{figure}[htbp]
\begin{center}
\includegraphics[scale= 0.6]{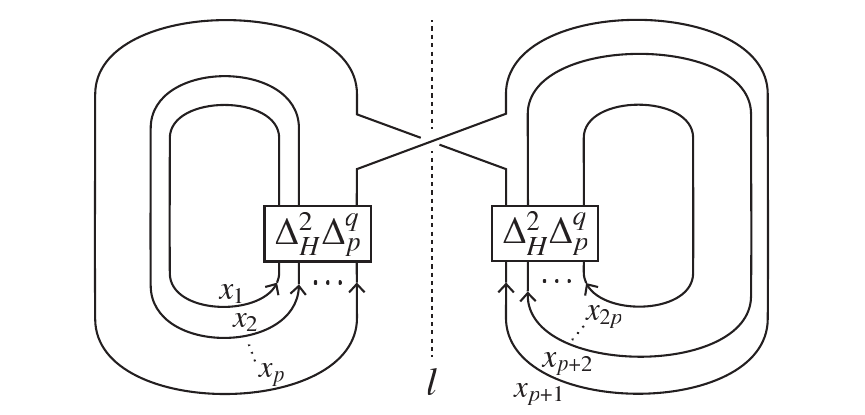}
\vspace{-4mm}
\caption{The diagram $\mathcal{D}$. \label{braid}}
\end{center}
\end{figure}

\begin{lem}
\label{arcs}
As elements of $\pi_1(S^3 \setminus T_{p,p+q} \# T_{p,p+q})$,
all arcs belonging to the left (resp.\ right) side of the dashed line $l$
in the diagram $\mathcal{D}$ are written as a conjugate of
$x_k$ by a word $w$ consisting of $x_1, \ldots, x_p$ (resp.\ $x_{p+1}, \ldots, x_{2p}$)
for some $k \in \{1, \ldots, p\}$ (resp.\ $k \in \{ p+1, \ldots, 2p\}$).
\end{lem}

\begin{proof}
We prove the lemma by induction on the place of arcs.
Here we first consider the left side of $l$.
Let us start from the bottom of the box $\Delta_H^2 \Delta^q_p$.
Then, for any $k \in \{1, \ldots, p\}$, the $k$-th arc from
the left is just $x_k$, and hence these $p$ arcs satisfy the assertion of
this lemma. 

Next, fix a crossing $\sigma_k$ in $\Delta^2_H \Delta^q_p$
and assume that all arcs below this $\sigma_k$ satisfy
the assertion of the lemma.
Then, since the upper right arc of the  $\sigma_k$
is the same as the bottom left arc (denoted $x_{i_{\sigma_k}}$),
it also satisfies the assertion.
Moreover, the upper left arc is equal
to $x^{-1}_{i_{\sigma_k}}x_{j_{\sigma_k}}x_{i_{\sigma_k}}$,
where $x_{j_{\sigma_k}}$ denotes the bottom right arc.
Here, by assumption, there exist some $k' \in \{1, \ldots, p\}$
and a word $w$ consisting of $x_1, \ldots, x_p$
such that $x_{j_{\sigma_k}}=w^{-1}x_{k'} w$.
Therefore, we have
$x^{-1}_{i_{\sigma_k}}x_{j_{\sigma_k}}x_{i_{\sigma_k}}=
(w x_{i_{\sigma_k}})^{-1}x_{k'} (w x_{i_{\sigma_k}})$.
Since $x_{i_{\sigma_k}}$ also consists of 
$x_1, \ldots, x_p$, this completes the proof
for the left side of the dashed line $l$.

Similarly, we can prove the lemma for the right side of $l$.
\end{proof}

\def\proofname{Proof of \Cref{np+q r_s}}

\begin{proof}
It is easy to check that the equality
$$
\Delta_H = \sigma_1 (\sigma_2 \sigma_1) \cdots (\sigma_{p-1} \cdots \sigma_2 \sigma_1) 
$$
holds. In particular, 
we have
$$
(\sigma_1^{-1}\sigma_2^{-1} \cdots \sigma_{p-1}^{-1})
(\sigma_1^{-1} \sigma_2^{-1} \cdots \sigma_{p-2}^{-1})
\cdots (\sigma_1^{-1} \sigma_2^{-1}) \sigma_1^{-1}
 = \Delta^{-1}_H,
$$
and hence
the positive crossing changes at all crossings in $\Delta_H$
gives $\Delta^{-1}_H$.
Now we perform the positive crossing changes at all crossings in the first
$\Delta_H$ of both two $\Delta_H^2 \Delta^q_p$'s
in Figure~\ref{braid}.
Then, by \Cref{cc cob}, we obtain a negative definite cobordism
$W$ with boundary $-S^3_{1/2}(T_{p,p+q} \# T_{p,p+q}) \# S^3_{1/2}(T_{p,q} \# T_{p,q})$
such that $\pi_1(W)$ has a presentation of the form shown in \Cref{cc cob}.
Here, we note that the crossing changes at the first $p$ crossings in 
$\Delta_H^2 \Delta^q_p$ on  the left gives the relations
$$
x_1 = x_2 = \cdots = x_p.
$$   
Similarly, we have $x_{p+1}= \cdots = x_{2p}$.
Therefore, by \Cref{arcs}, all arcs in the left (resp.\ right) side of the dashed line $l$
in $D$ are equal to $x_1$ (resp.\ $x_{p+1}$) as elements of $\pi_1(W)$.
Moreover, $x_{p+1}$ belongs to not only the right side of $l$ but also the left side,
and hence $x_{p+1}=x_1$.

Now, any two generators in our presentation of $\pi_1(W)$ are equal.
Moreover, since $\lambda$ is in the commutator subgroup, 
we have $\lambda=1$, and hence $x_1=1$. 
This gives $\pi_1(W)=1$.
Therefore, by applying \Cref{pi1=1} to $W$,
we have
$$
r_s(\Sigma(D_{p,p+q}))  =
r_s(S^3_{1/2}(T_{p,p+q} \# T_{p,p+q})) 
< r_s(S^3_{1/2}(T_{p,q} \# T_{p,q})) = r_s(\Sigma(D_{p,q})).
$$
Since $q$ is an arbitrary integer with $q>1$ and
$\gcd(p,q)=1$, 
this inequality holds even if we replace $q$ with $kp+q$ for any $k\in \z_{>0}$.
Consequently, we have
$$
r_s(\Sigma(D_{p,q})) > r_s(\Sigma(D_{p,p+q}))
>r_s(\Sigma(D_{p,2p+q}))
> \cdots.
$$
Combining with \Cref{$D_{p,q}$}, this completes the proof.
\end{proof}

\def\proofname{Proof}
\section{Additional structures on $\Theta^3_{\z}$ and $\ker h$}
\label{sec:Additional_structures}
In this section, we prove Theorems~\ref{Theta_zr} and Theorem~\ref{s infty}. 
Recall that for any $r \in [0, \infty]$, the subgroup $\Theta^3_{\z,r} \subset \Theta^3_{\z}$
is defined by
\[
\Theta^3_{\z, r} := \left\{ [Y] \in \Theta^3_\z \mid \min \{ r_0 (Y), r_0 (-Y )\} \geq r \right\}.
\]

\setcounter{section}{1}
\setcounter{thm}{13}
\begin{thm}
For any $r \in (0,\infty]$, 
the quotient group $\Theta^3_\z/\Theta^3_{\z, r}$ contains $\z^\infty$
as a subgroup.
\end{thm}
\setcounter{section}{6}
\setcounter{thm}{0}

\begin{proof}
To prove the theorem, we use the sequence $\{\Sigma(2,3,6k-1)\}_{k=1}^\infty$.
(In fact, we can replace it with any sequence $\{Y_k\}_{k=1}^\infty$ 
such that $\{r_0(Y_k)\}_{k=1}^\infty$ is a decreasing sequence and converges to zero.)
Fix $r \in (0,\infty]$.
Then, since $r_0(-\Sigma(2,3,6k-1))=1/24(6k-1)$ converges to zero, there exists an integer $N$ such that 
$1/24(6N-1) < r$. Let $[Y]_r$ denote the equivalence class of $[Y]$ in 
$\Theta^3_\z/ \Theta^3_{\z,r}$. We prove that
$\{[\Sigma(2,3,6k-1)]_r\}_{k=N}^\infty$ are linearly independent in 
$\Theta^3_\z/\Theta^3_{\z,r}$.

Assume that $\sum_{k=N}^{M} n_k [\Sigma(2,3,6k-1)]_r=0$ and $n_M \neq 0$.
Without loss of generality, we may assume that $n_M >0$.
Then, by the definition of $\Theta^3_{\z,r}$, we have
$$
\min\left\{r_0\big(\sum_{k=N}^{M} n_k [\Sigma(2,3,6k-1)]\big),
r_0\big(-\sum_{k=N}^{M} n_k [\Sigma(2,3,6k-1)]\big)\right\} \geq r.
$$
However,  \Cref{combination} implies
\begin{align*}
r_0\big(-\sum_{k=N}^{M} n_k [\Sigma(2,3,6k-1)]\big)
& = r_0(-\Sigma(2,3,6M-1)) \\
& = \frac{1}{24(6M-1)} \leq \frac{1}{24(6N-1)} < r,
\end{align*}
a contradiction.
\end{proof}

\subsection{A pseudometric on $\ker h$}\label{metric}
 We next  consider a pseudometric on $\ker h$, where $h : \Theta^3_\z \to \z$ is the Fr\o yshov invariant.  To define it, set 
\[
s_\infty(Y):= \sup \{ s \in [-\infty,0] \mid r_s(Y)= \infty \}.
\]
Then, as a corollary of the connected sum formula for $\{r_s\}$, we have the following theorem.
\begin{thm}\label{s infty}
$s_\infty(Y_1\# Y_2) \geq s_\infty(Y_1)+s_\infty(Y_2)$.
\end{thm}
Moreover, \Cref{Froyshov} implies that if $h(Y)=0$, then $\max \{-s_\infty(Y), -s_\infty(-Y)\} < \infty$.
Now we can define a pseudometric on $\ker h$ as
\[
d_\infty([Y_1],[Y_2]):= - s_\infty(Y_1\# (-Y_2)) - s_\infty((-Y_1)\# Y_2).
\]
Moreover, the set of elements with $d_\infty([S^3],[Y])=0$ coincides with $\Theta^3_{\z,\infty}$.
\begin{thm}
\label{d infty}
The map $d_\infty$ gives a metric on $\ker h/ \Theta^3_{\z,\infty}$,
and the action of $\ker h/\Theta^3_{\z, \infty}$ on $(\ker h/\Theta^3_{\z, \infty},d_{\infty})$
is an isometry.
In particular, $\ker h/\Theta^3_{\z, \infty}$ is a topological group 
with respect to the metric topology induced by $d_\infty$.
\end{thm}

Note that $\{\Sigma(2,3,5) \# (-\Sigma(2,3,6k-1)) \}_{k\in \z_{>1}}$ are linearly independent in $\ker h/ \Theta^3_{\z,\infty}$, and hence $\ker h/ \Theta^3_{\z,\infty}$ contains $\z^\infty$ as a subgroup.
Here, we post the following question.

\begin{ques}
\label{ques topology}
What is the isomorphism type of $(\ker h/ \Theta^3_{\z,\infty}, d_\infty)$
as a topological group? In particular, is it a discrete group?
\end{ques}


\begin{lem}
\label{r_0 and s infty}
If $s_\infty(Y)=0$, then $r_0(Y)=\infty$.
\end{lem}
\begin{proof}
Since $s_\infty(Y)=0$, 
for any $s<0$, we have $r_s(Y)=\infty$, and hence
$[\theta^{[s,\infty]}_Y]=0$.
Suppose that $0 \in \R \setminus \Lambda^*_Y$. Then we can take $s<0$ such that
$[s,0] \subset \R \setminus \Lambda^*_Y$.
Hence, Lemmas~\ref{filtration} and \ref{inclusion} give 
an isomorphism from $I^1_{[0,\infty]}(Y)$ to $I^1_{[s,\infty]}(Y)$
which maps $[\theta^{[0,\infty]}_Y]$ to $[\theta^{[s,\infty]}_Y]$.
This implies $[\theta^{[0,\infty]}_Y]=0$, and hence $r_0(Y)=\infty$.

Next, suppose that $0 \in \Lambda^*_Y$. Then, by the definition of $CI^*_{[0,\infty]}(Y)$,
the cohomology group $I^1_{[0,\infty]}(Y)$ and $[\theta^{[0,\infty]}_Y]$ coincide with
$I^1_{[-\frac{1}{2}\lambda_Y,\infty]}(Y)$ and $[\theta^{[-\frac{1}{2}\lambda_Y,\infty]}_Y]$ respectively, where $\lambda_Y:= \min \{ |a-b| \mid a,b \in \Lambda_Y \text{ with }a\neq b  \}>0$. 
This implies $[\theta^{[0,\infty]}_Y]=0$, and hence $r_0(Y)=\infty$.
\end{proof}

Next, we prove \Cref{s infty}.

\begin{proof}[Proof of \Cref{s infty}]
We may assume $s_\infty(Y_1)+s_\infty(Y_2) \neq -\infty $. For any $s  \in (-\infty, s_\infty(Y_1)+s_\infty(Y_2))$, there exist $s_1 \in (-\infty, s_\infty(Y_1))$ and
$s_2 \in (-\infty, s_\infty(Y_2))$ such that $s=s_1+s_2$. For such $s_1$ and $s_2$, we have the following connected sum formula 
\[
r_s(Y_1\# Y_2 ) \geq \min \{ r_{s_1}(Y_1) +s_2,  r_{s_2}(Y_2) +s_1\}.
\]
Since $s_1$ and $s_2$ are in $(-\infty, s_\infty(Y_1))$ and $(-\infty, s_\infty(Y_2))$
respectively, we have $r_s(Y_1\# Y_2 )= \infty$. This completes the proof. 
\end{proof}
Now we prove \Cref{d infty}.
Recall that $d_\infty$ is a function on $\ker h \times \ker h$ defined by
\[
d_\infty([Y_1],[Y_2]):= - s_\infty(Y_1\# -Y_2) - s_\infty(-Y_1\# Y_2).
\]
By \Cref{Froyshov}, the equalities $r_{-\infty}(Y)= r_{-\infty}(-Y)=0$ hold if and only if $h(Y)=0$.
Moreover, we have $r_{-\infty}(\pm Y) = r_s(\pm Y)$ for sufficiently small $s \in (-\infty, 0]$.
These imply that $d_\infty([Y_1],[Y_2])$ is finite  
for any pair $([Y_1],[Y_2]) \in \ker h \times \ker h$.

\begin{proof}[Proof of \Cref{d infty}]
For any $[Y_1], [Y_2] \in \ker h$, the equalities $d_\infty([Y_1], [Y_1])=0$
and $d_\infty([Y_1], [Y_2])= d_\infty([Y_2], [Y_1])$ obviously hold.
Suppose that $[Y_1]$, $[Y_2]$ and $[Y_3]$ are three elements of $\ker h$.
Then, by \Cref{s infty}, we have
\begin{align*} 
 d_\infty ( [Y_1], [Y_3] ) 
& = - s_\infty(Y_1\# -Y_3) - s_\infty(-Y_1\# Y_3)  \\
& =- s_\infty(Y_1\# -Y_3 \# Y_2 \# -Y_2) - s_\infty(-Y_1\# Y_3 \# Y_2 \# -Y_2) \\
& \leq -s_\infty (Y_1 \# -Y_2 ) -s_\infty (Y_2 \# -Y_3 ) - s_\infty (-Y_1\# Y_2)  
- s_\infty (-Y_2 \# Y_3)   \\
& = d_\infty ([Y_1],[Y_2]) + d_\infty ([Y_2],[Y_3]) .
\end{align*}
Therefore $d_\infty$ gives a pseudometric on $\ker h$.  

We next prove that $d_\infty$ is well-defined on $\ker h/ \Theta^3_{\z,\infty}$.
Let $[Y_1], [Y_2] \in \ker h$ and $[M_1], [M_2] \in \Theta^3_{\z,\infty}$. 
Then we have 
$s_\infty(M_i)=s_\infty(-M_i)= 0$ ($i=1,2$), and
hence \Cref{s infty} implies that
\begin{align*} 
d_\infty([Y_1]+[M_1 ],[Y_2]+[M_2 ]) 
& =  - s_\infty(Y_1 \# M_1 \# -Y_2 \# -M_2) - s_\infty(-Y_1\#-M_1\# Y_2\# M_2)\\
& \leq   - s_\infty(Y_1  \# -Y_2) - s_\infty(-Y_1 \# Y_2) = d_\infty([Y_1 ],[Y_2 ])
\end{align*}
and 
\begin{align*} 
&d_\infty([Y_1 ],[Y_2 ])\\
&=   - s_\infty(Y_1  \# -Y_2) - s_\infty(-Y_1 \# Y_2) \\
&=   - s_\infty(Y_1  \# -Y_2\#M_1\#-M_1 \# M_2\#-M_2) - 
s_\infty(-Y_1 \# Y_2 \#M_1\#-M_1 \# M_2\#-M_2) \\
&\leq   - s_\infty(Y_1 \# M_1 \# -Y_2\# -M_2) - s_\infty(-Y_1\# -M_1\# Y_2\# M_2)\\
&=  d_\infty([Y_1]+ [M_1],[Y_2]+[M_2 ]).
\end{align*}
Therefore, $d_\infty$ is well-defined on $\ker h/ \Theta^3_{\z,\infty}$. 

Next, we prove that $d_\infty$ is a metric on $\ker h/ \Theta^3_{\z,\infty}$.
It is easy to check that $d_\infty$ is a pseudometric on $\ker h/ \Theta^3_{\z,\infty}$.
Suppose that elements $[Y_1]$ and $[Y_2]$ of $\ker h/ \Theta^3_{\z,\infty}$ satisfy $d_\infty([Y_1],[Y_2])=0$.
Then we have 
$s_\infty(Y_1\#-Y_2)=s_\infty(-Y_1\#Y_2)=0$, and these equalities 
and \Cref{r_0 and s infty} imply that 
$r_0(Y_1 \# -Y_2)=r_0(-Y_1 \# Y_2)=\infty$.
Therefore, by the definition of $\Theta^3_{\z,\infty}$ we have 
$[Y_1]=[Y_2]$ as elements of $\ker h/ \Theta^3_{\z,\infty}$. 
This proves that $d_\infty$ is a metric on $\ker h/ \Theta^3_{\z,\infty}$.

Finally, we prove that the group operation of $\ker h/ \Theta^3_{\z,\infty}$ is an isometry with respect to $d_\infty$.
Indeed, we see that
\begin{align*} 
d_\infty ( [Y_1]+[M], [Y_2]+ [M] ) 
& = - s_\infty(Y_1\# M \# -Y_2\# -M ) - s_\infty(-Y_1\# -M \# Y_2\# M )  \\
& = d_\infty ( [Y_1], [Y_2 ])
\end{align*}
for any elements $[Y_1]$, $[Y_2]$ and $[M]$ of $\ker h/\Theta^3_{\z, \infty}$.
This completes the proof. 
\end{proof}
As a study of concrete cases,
we give partial estimates of $d_\infty$ for the connected sums of some Seifert homology $3$-spheres. 
\begin{prop} 
\label{partial estimate}
For any $n \in \z_{>0}$, we have 
\[
 d_\infty([S^3] , [\Sigma(2,3,6n-1)\# -\Sigma(2,3,6n+5)]) \geq \frac{1} {4(6n-1)(6n+5)} .
 \]
\end{prop}
\begin{proof}The connected sum formula for $r_s$ gives 
\[
r_s( -\Sigma(2,3,6n+5)) \geq \min \{ r_{s_1}(\Sigma(2,3,6n-1)\# -\Sigma(2,3,6n+5)) +s_2, r_{s_2}(-\Sigma(2,3,6n-1)) +s_1\}
\]
for any $s$, $s_1, s_2 \in (-\infty, 0]$ with $s=s_1+s_2$. In particular, if $s_2=0$, 
then we have
\[
 \frac{1}{24(6n+5)} \geq  \min \{ r_{s_1}(\Sigma(2,3,6n-1)\# -\Sigma(2,3,6n+5)) ,  \frac{1} {24(6n-1)} +s_1\}.
 \]
This implies that if  $\frac{1}{24(6n+5)}< \frac{1} {24(6n-1)} +s_1$, then $r_{s_1}(\Sigma(2,3,6n-1)\# -\Sigma(2,3,6n+5))< \infty$.
Hence we have
\begin{align*}
-s_\infty (\Sigma(2,3,6n-1)\# -\Sigma(2,3,6n+5))
& \geq \frac{1} {24(6n-1)}-\frac{1} {24(6n+5)}\\
&= \frac{1}{4(6n-1)(6n+5)}.
\end{align*}
Moreover, 
since $-\Sigma(2,3,6n-1)=S^3_{1/n}(3_1^*)$, 
we obtain a negative definite 4-manifold with boundary 
$-\Sigma(2,3,6n-1)\# \Sigma(2,3,6n+5)$ from 
the cobordism $W_n$ in Section~\ref{section 5.3} with reversed orientation.
Therefore, we have $r_0(-\Sigma(2,3,6n-1)\# \Sigma(2,3,6n+5))=\infty$ 
and $s_\infty (-\Sigma(2,3,6n-1)\# \Sigma(2,3,6n+5)) =0$. This completes the proof.
\end{proof}
Here we post the following question:
\begin{ques}
\label{ques distance}
Does the equality
\[
 d_\infty([S^3] , [\Sigma(2,3,6n-1)\# -\Sigma(2,3,6n+5)]) = \frac{1} {4(6n-1)(6n+5)} 
 \]
hold?
\end{ques}
If the equality holds, then the sequence
\[
\{a_n\}_{n=1}^\infty := \{[\Sigma(2,3,6n-1)\# -\Sigma(2,3,6n+5)]\}_{n=1}^\infty
\]
converges to $[S^3]$ in $\ker h/ \Theta_{\z, \infty}$.
In particular, we would conclude that the topology on $\ker h/ \Theta_{\z, \infty}$
induced by $d_{\infty}$ is different from the discrete topology on $\ker h/ \Theta_{\z, \infty}$.

\section{Computation for a hyperbolic 3-manifold}\label{Calculations}
In this section, we give approximations of the critical values of the Chern-Simons functional on a certain hyperbolic 3-manifold.
Moreover, using the computer, we obtain an approximated value of $r_s(Y)$ for this hyperbolic 3-manifold.

\subsection{$1/n$-surgery along a knot $K$}
We here review a formula of $\cs$ due to Kirk and Klassen~\cite{KiKl90}, and explain our method of computing an approximate value of $\cs$.
For a compact manifold $M$, we define $\calR(M)$ by $\calR(M) = \Hom(\pi_1(M),\SL(2,\co))$ and call it the \emph{$\SL(2,\co)$-representation space} of $M$.
In this paper, we equip $\calR(M)$ with the compact open topology.

For a knot $K$ in $S^3$, let $E(K)$ denote the exterior of an open tubular neighborhood of $K$, and $\mu$, $\lambda \in \pi_1(T^2)$ a meridian and (preferred) longitude respectively.

\begin{thm}[{\cite[Theorem~4.2]{KiKl90}}]\label{KK}
 Let $\rho_0$, $\rho_1$ be $\SU(2)$-representations of $\pi_1(S^3_{1/n}(K))$ and 
 $\gamma\colon [s_0,s_1] \to \{\rho \in \calR(E(K)) \mid \text{$\rho|_{\pi_1(T^2)}$ is completely reducible}\}$
 a piecewise smooth path with $\gamma(s_i)=\rho_i$ in $\calR(E(K))$.
 Then, 
\begin{align}\label{KKformula}
\cs(\rho_1) - \cs(\rho_0) \equiv 2\int_{s_0}^{s_1} \beta(s)\alpha'(s) ds +n\left(\beta(s_1)^2-\beta(s_0)^2\right) \mod \z,
\end{align}
 where $\alpha, \beta \colon [s_0,s_1] \to \co$ are piecewise smooth functions such that the matrices $\gamma(s)(\mu)$, $\gamma(s)(\lambda)$ are simultaneously diagonalized as 
\[
\begin{pmatrix}
 e^{2\pi i \alpha(s)} & 0 \\
 0 & e^{-2\pi i \alpha(s)}
\end{pmatrix},\ 
\begin{pmatrix}
 e^{2\pi i \beta(s)} & 0 \\
 0 & e^{-2\pi i \beta(s)}
\end{pmatrix},
\]
respectively.
\end{thm}

\begin{rem}
Kirk and Klassen showed \Cref{KK} for a family of $\SU(2)$-connections.
As written in \cite[p.~354]{KiKl90}, the formula can be extended to the case of $\SL(2,\co)$. 
We need to define the smoothness of a path 
\[
\gamma\colon [s_0,s_1] \to \{\rho \in \calR(E(K)) \mid \text{$\rho|_{\pi_1(T^2)}$ is completely reducible}\}
\]
since Stokes' theorem is used in the proof of \Cref{KK}.
If we fix a generating system of $\pi_1(E(K))$, the space $\calR(E(K))$ can be embedded into $SL(2,\co)^N$, where $N$ is the number of generators.
If the composite of $\gamma\colon [s_0,s_1] \to \calR(E(K))$ and $\calR(E(K)) \to SL(2,\co)^N$ is piecewise smooth, we call $\gamma$ a piecewise smooth path.
For such a path $\gamma$ on $[s_0,s_1] = \bigcup_j I_j$, a piecewise smooth family of $\SL(2,\co)$-connections $A_s$ on $E(K)$ is defined by considering the inverse map of the holonomy correspondence.
Then we have a smooth connection on $E(K)\times I_j$ for each $j$, and one can check the formula \eqref{KKformula}.
\end{rem}

It is difficult to find a suitable path and compute the above integral in general.
For a 2-bridge knot $K$, the subspace $\calR^\irr(E(K))$ of the irreducible representations is explicitly described by the Riley polynomial as follows.
We first recall that $\pi_1(E(K))$ admits a presentation of the form $\langle x, y \mid wx=yw \rangle$, where $w$ is a certain word in $x$ and $y$ (see \cite[p.~358]{KiKl90}).
For $t \in \co\setminus\{0\}$, $u \in \co$ and $\varepsilon \in \{\pm 1\}$, let $\rho_{t,u,\varepsilon}$ denote the representation of the free group $\langle x, y \mid -\rangle$ of rank 2 given by 
\[
\rho_{t,u,\varepsilon}(x)=\varepsilon
\begin{pmatrix}
 \sqrt{t} & 1/\sqrt{t} \\
 0 & 1/\sqrt{t}
\end{pmatrix},\ 
\rho_{t,u,\varepsilon}(y)=\varepsilon
\begin{pmatrix}
 \sqrt{t} & 0 \\
 -\sqrt{t}u & 1/\sqrt{t}
\end{pmatrix},
\]
where $\sqrt{re^{i\theta}}=\sqrt{r}e^{i\theta/2}$ for $r \geq 0$ and $-\pi < \theta \leq \pi$.
Here, the Riley polynomial of $K$ (for the above presentation) is defined by $\phi(t,u) = w_{11}+(1-t)w_{12} \in \z[t^{\pm 1/2},u]$, where $w_{ij}$ is the $(i,j)$-entry of $\rho_{t,u,\varepsilon}(w)$.
Then $\rho_{t,u,\varepsilon}$ gives a representation of $\pi_1(E(K))$ if and only if $\phi(t,u)=0$.
Moreover, any irreducible representation of $\pi_1(E(K))$ is conjugate to $\rho_{t,u,\varepsilon}$ for some $t$, $u$ and $\varepsilon$.

Here, $\rho_{t,u,\varepsilon}$ is conjugate to an $\SU(2)$-representation if and only if $|t|=1$, $t \neq 1$ and $u \in (t+t^{-1}-2, 0)$.
Note that an $\SU(2)$-representation $\rho_{t,u,\varepsilon}$ is $\SU(2)$-conjugate to $\rho_{t^{-1},u,\varepsilon}$.

Let us find a path from $\rho_{t_0,u_0,\varepsilon_0}$ to $\rho_{t_1,u_1,\varepsilon_1}$ in
$$\{\rho \in \calR^\irr(E(K)) \mid \text{$\rho|_{\pi_1(T^2)}$ is completely reducible}\}.$$
First note that one needs not care about $\varepsilon_i$ since the right-hand side of \eqref{KKformula} is independent of the choice of $\varepsilon_i$.
Consider the $d$-fold branched cover 
$$\pr_1\colon \{(t,u) \in (\co\setminus\{0,1\})\times\co \mid \phi(t,u)=0\} \to \co\setminus\{0,1\},$$
where $d=\deg_u \phi$.
In order to find a path, we first take a path $\gamma$ from $t_0$ to $t_1$ and its lift $\tilde{\gamma}$ satisfying $\pr_2\circ\tilde{\gamma}(s_j)=u_j$.
Since the lift starting from $(t_0,u_0)$ might end at $(t_1,u_1')$ with $u_1' \neq u_1$, one should choose a path $\gamma$ carefully.
We now have $\alpha(s)=\frac{1}{4\pi i}\log\gamma(s)$ with an analytic continuation along $\gamma$.

Once the function $u(s)$ satisfying $\tilde{\gamma}(s)=(\gamma(s),u(s))$ is given explicitly, one gets $\beta(s)=\frac{1}{2\pi i}\log (P^{-1}\rho_{\gamma(s),u(s),\varepsilon}(\lambda)P)_{11}$, where $P=P(s)$ is a matrix satisfying $(P^{-1}\rho_{\gamma(s),u(s),\varepsilon}(\mu)P)_{11} = e^{2\pi i\alpha(s)}$.
We finally integrate $\beta(s)\alpha'(s)$ on $[s_0,s_1]$.

In fact, one can express $u(s)$ explicitly by solving $\phi(t,u)=0$ when $\deg_u\phi \leq 4$.
Here, we should be careful to connect the solutions.
For instance, let $\phi(t,u)=t-u^2$.
Then we have $u_0(t)=\sqrt{t}$, $u_1(t)=-\sqrt{t}$.
In order to find a path from $(i,e^{\pi i/4})$ to $(-i,e^{3\pi i/4})$, we define $\gamma\colon [1/2,3/2] \to \co$ by $\gamma(s)=e^{s \pi i}$.
The lift of $\gamma$ is obtained by combining $u_0$ and $u_1$:
$$
\tilde{\gamma}(s) = 
\begin{cases}
 (\gamma(s), u_0(\gamma(s))) & \text{if $1/2 \leq s \leq 1$,} \\
 (\gamma(s), u_1(\gamma(s))) & \text{if $1 < s \leq 3/2$.}
\end{cases}
$$

\begin{rem}
 It is difficult to solve $\phi(t,u)=0$ and $\rho_{t,u,\varepsilon}(\mu\lambda^n)=I_2$ simultaneously.
 We actually use the $A$-polynomial $A_K(L,M) \in \z[L,M]$ of $K$.
 Indeed, first solve the one variable equation $A_K(L,L^{-n})=0$, and then $t=M^2=L^{-2n}$.
 We next solve $\phi(L^{-2n},u)=0$ with respect to $u$.
\end{rem}

\subsection{$1/2$-surgery along the knot $5_2^\ast$}
We actually consider the manifold $S^3_{-1/2}(5_2) = -S^3_{1/2}(5_2^\ast)$ and multiply the result of computation of $\cs(\rho)$ by $-1$.
Recall that $\cs_Y(\rho) = -\cs_{-Y}(\rho)$.

We first fix the presentation of the group $\pi_1(E(5_2))$ as $\langle x, y \mid [y,x^{-1}]^2x = y[y,x^{-1}]^2 \rangle$, where a meridian and longitude are expressed as $x$ and $[x,y^{-1}]^2[y,x^{-1}]^2$, respectively.
Then the Riley polynomial and $A$-polynomial of $5_2$ are given by 
\begin{align*}
 \phi(t,u) &= -(t^{-2}+t^2)u+(t^{-1}+t)(2+3u+2u^2)-(3+6u+3u^2 +u^3), \\
 A_{5_2}(L,M) &= -L^3 -M^{14} +L^2 (1-2M^2 -2M^4 +M^8 -M^{10}) \\
 &\quad +LM^4(-1+M^2 -2M^6 -2M^8 +M^{10}).
\end{align*}

\begin{figure}[h]
 \centering
 \includegraphics[width=0.4\columnwidth]{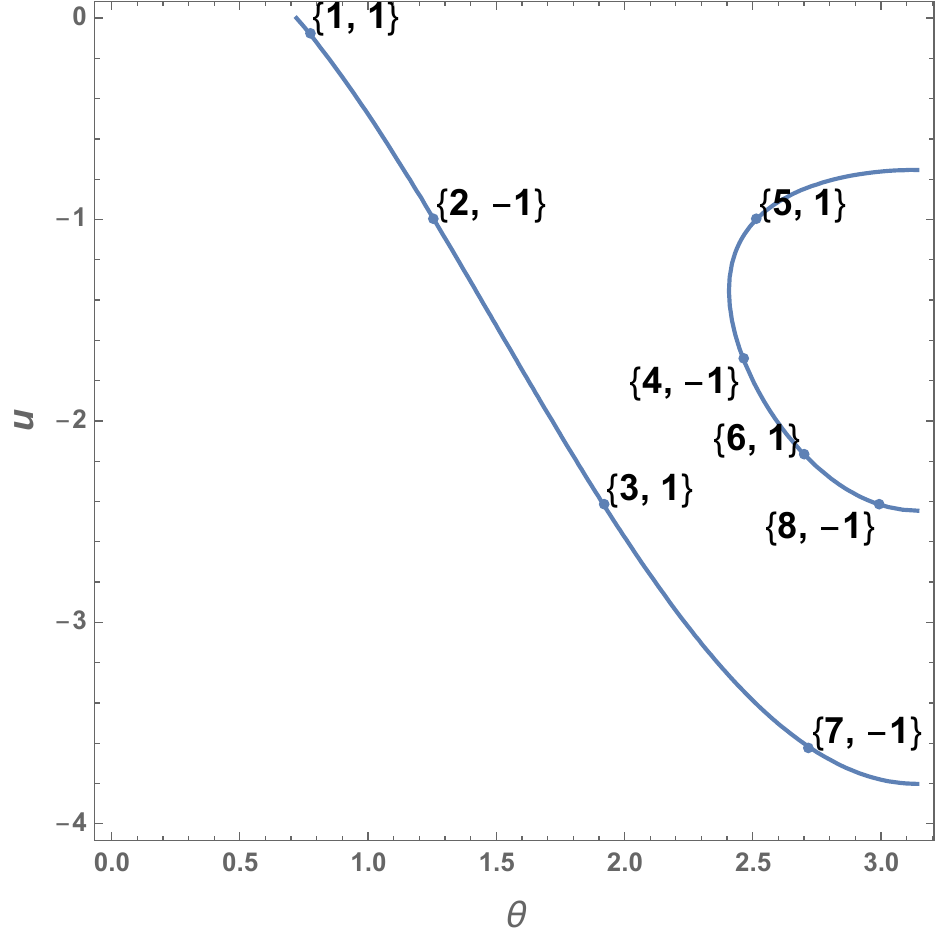}
 \caption{The 8 non-trivial representations of $\pi_1(S^3_{-1/2}(5_2))$ in the quotient space obtained from the non-abelian representations of $\pi_1(E(5_2))$ by identifying $\rho_{t,u,\varepsilon}$ and $\rho_{t,u,-\varepsilon}$, where $t=e^{i\theta}$.
 The second entry of a label indicates $\varepsilon$.}
 \label{fig:RepSp}
\end{figure}

Here, one sees that there are 8 conjugacy classes of non-trivial $\SU(2)$-representations of $\pi_1(S^3_{-1/2}(5_2))$ as drawn in Figure~\ref{fig:RepSp}.
Strictly speaking, these are candidates of representations because of a numerical computation.
Assume that some of them do not give representations.
Then there exists a non-trivial representation $\rho$ close to one of the candidates such that $H^1(S^3_{-1/2}(5_2); \su_{{\Ad}\circ\rho}) \neq 0$ since the Casson invariant of $S^3_{-1/2}(5_2)$ is equal to $-4$ and $|-4|\times 2=8$.
Here, since $\rho$ is non-abelian, we have $\rank_\R \partial_1 =3$ in the chain complex
\[
C_\ast(E(5_2); \su_{{\Ad}\circ\rho}) =
\begin{cases}
 \R^3 & \text{if $\ast=0,2$,} \\
 \R^6 & \text{if $\ast=1$,} \\
 0 & \text{otherwise}
\end{cases}
\]
obtained from the above presentation of $\pi_1(E(5_2))$.
It follows from the Mayer-Vietoris exact sequence and Poincar\'e-Lefschetz duality that $\rank_\R \partial_2 \leq 1$.
On the other hand, we can see that this inequality does not hold for the 8 candidates by the computer.
Therefore, they correspond to true representations.

\begin{figure}[h]
 \centering
 \includegraphics[width=0.5\columnwidth]{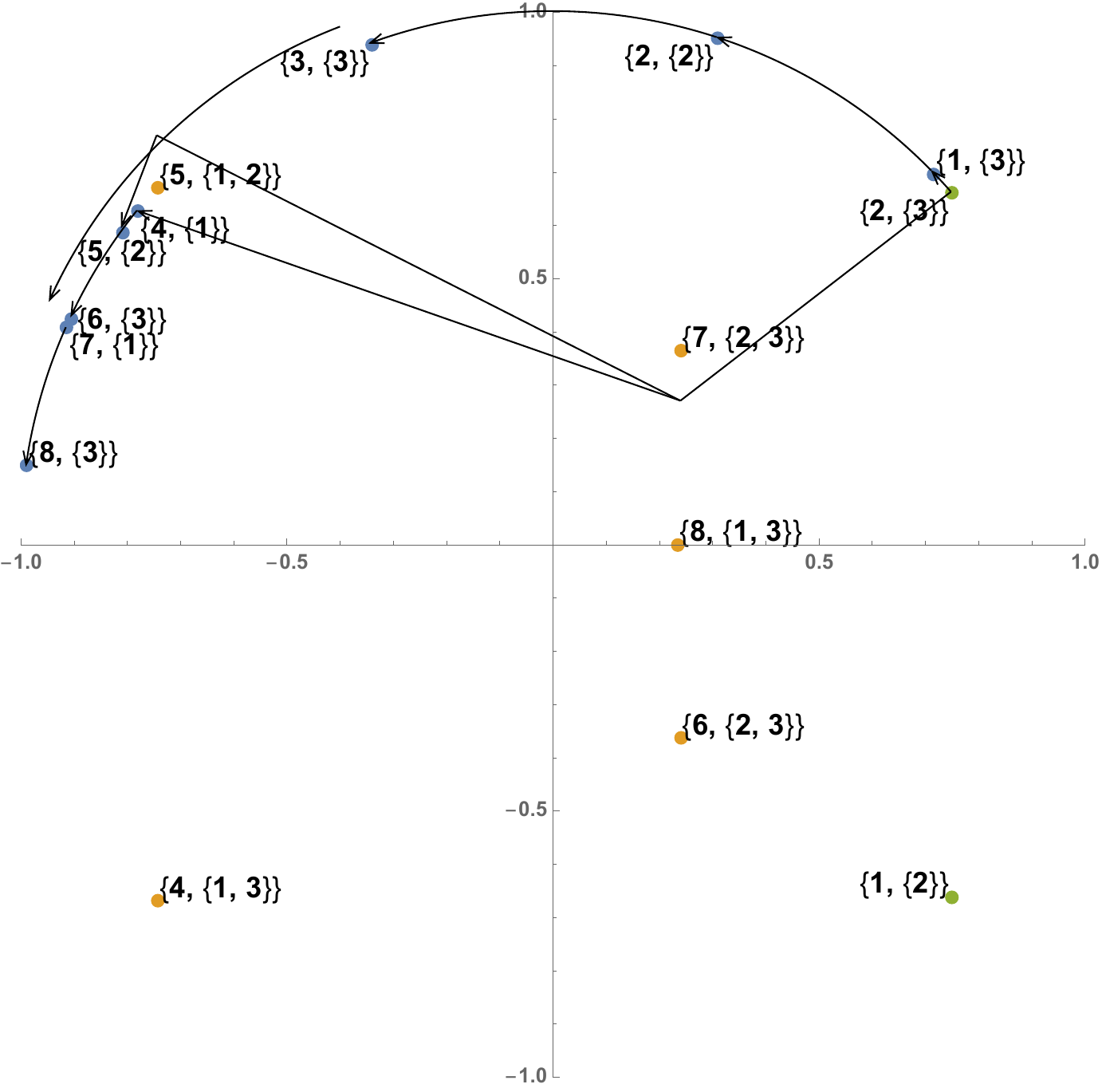}
 \caption{The blue (resp.\ green, orange) dots correspond to representations of $\pi_1(S^3_{-1/2}(5_2))$ (resp.\ the roots of $\Delta_{5_2}(t)$, some branched points).
 Here the label $\{i,\{j\}\}$ (resp.\ $\{i,\{j_1,j_2\}\}$) at $t \in \co$ means $\phi(t, u_j(t))=0$ (resp.\ $\phi(t, u_{j_k}(t))=0$ for $k=1,2$).}
 \label{fig:tPlanePath}
\end{figure}

The following computation is based on Mathematica.
Since $\deg_u \phi = 3$, one gets the explicit solutions $u_1(t), u_2(t), u_3(t)$ of $\phi(t,u) = 0$.
We take 8 paths as illustrated in Figure~\ref{fig:tPlanePath} and apply \Cref{KK} to these paths.
Note that some paths start from a root of the Alexander polynomial $\Delta_{5_2}(t)$ of $5_2$, and for these paths we use \cite[Lemma~5.3]{Fr82} to compute integrals.
The result of the computation is listed in Table~\ref{tab:SU2reps}.

\begin{table}[h]
\centering
$
\begin{array}{r|r|r|r|r}
 & t & u & \varepsilon & -\cs \\\hline
 \rho_1 & 0.716932 + 0.697143 i & -0.0755806 & 1 & 0.00176489 \\\hline
 \rho_2 & 0.309017 + 0.951057 i & -1.00000 & -1 & 0.166667 \\\hline
 \rho_3 & -0.339570 + 0.940581 i & -2.41421 & 1 & 0.604167 \\\hline
 \rho_4 & -0.778407 + 0.627759 i & -1.69110 & -1 & 0.388460 \\\hline
 \rho_5 & -0.809017 + 0.587785 i & -1.00000 & 1 & 0.166667 \\\hline
 \rho_6 & -0.905371 + 0.424621 i & -2.16991 & 1 & 0.865934 \\\hline
 \rho_7 & -0.912712 + 0.408603 i & -3.62043 & -1 & 0.321158 \\\hline
 \rho_8 & -0.988857 + 0.148870 i & -2.41421 & -1 & 0.604167 
\end{array}
$
\vspace*{5pt}
\caption{The values of $-\cs$ for the representations of $\pi_1(S^3_{-1/2}(5_2))$.
Note that $0.16666\cdots67 \approx 1/6$ and $0.60416\cdots67 \approx 29/48$, where both decimals have 46 digits of 6's in the omitted part.}
\label{tab:SU2reps}
\end{table}%

Recall that $r_s(S^3_1(5_2^\ast))=1/4\cdot 2\cdot 3\cdot 11 \approx 0.00379$.
Since 0.00176489 is the only value less than 0.00379 among the 8 values, we conclude that $r_s(S^3_{1/2}(5_2^\ast)) \approx 0.00176489$.
Moreover, we improve the precision, and get
$$
r_s(S^3_{1/2}(5_2^\ast)) \approx 
0.0017648904\ 7864885113\ 0739625897\ 0947779330\ 4925308209
$$
for all $s \in [-\infty, 0]$.

\appendix
\section{Hendricks, Hom, Stoffregen, and Zemke's example} \label{sec:HHSZ20}
In \cite{HHSZ20}, they intensively studied the homology 3-sphere obtained from $S^3$ by Dehn surgery along the framed knot at the top-left in Figure~\ref{fig:Kirby3}.
This appendix is devoted to showing that their homology 3-sphere is a graph manifold.
%
%

\begin{figure}[h]
 \centering
 \includegraphics[width=\textwidth]{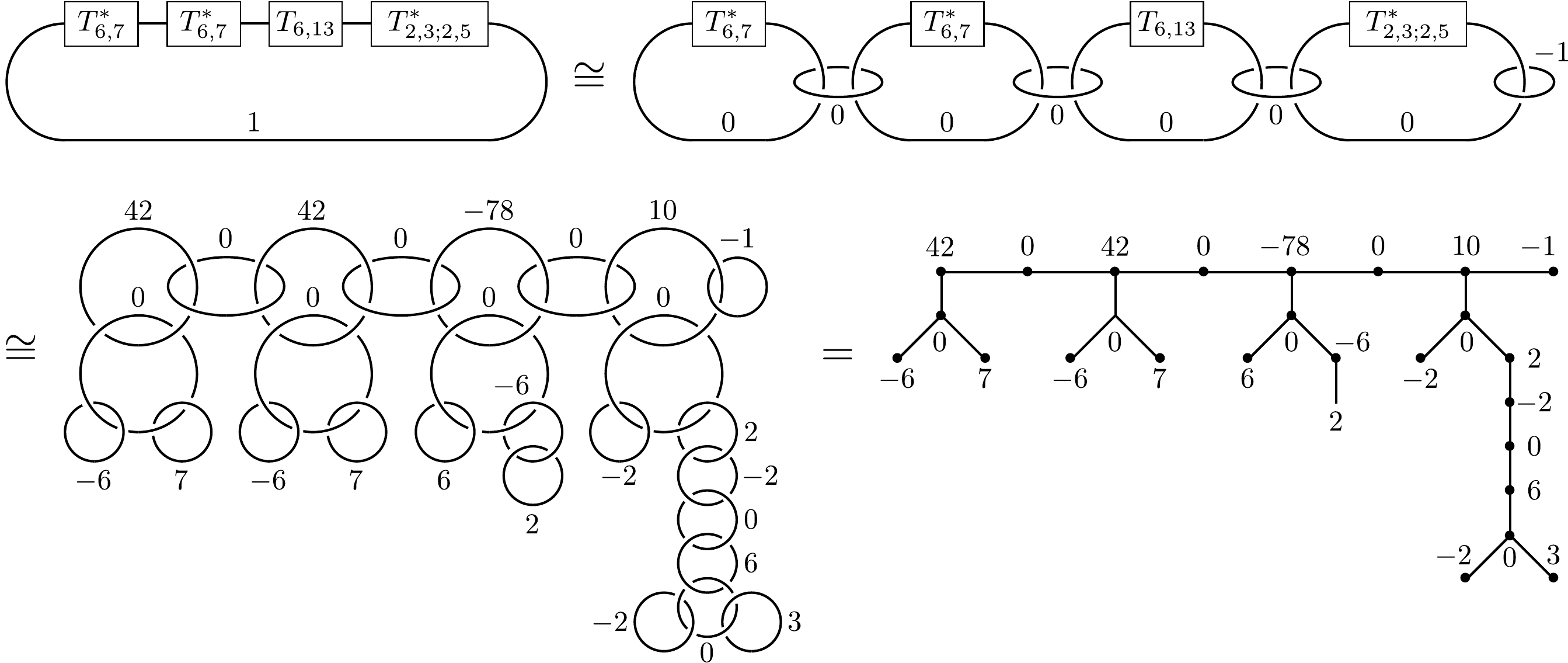}
 \caption{A diffeomorphism between $S^3_1(2T_{6,7}^\ast \# T_{6,13} \# T_{2,3;2,5}^\ast)$ and a graph manifold.}
 \label{fig:Kirby3}
\end{figure}

The first diffeomorphism in Figure~\ref{fig:Kirby3} follows from standard Kirby calculus.
In order to prove the second diffeomorphism, we consider three 3-manifolds obtained by Dehn surgery along $T_{6,7}^\ast$, $T_{6,13}$, and the mirror of the $(2,5)$-cable of $T_{2,3}$, respectively.
Here we put framed knots in these 3-manifolds as drawn in thick lines in Figures~\ref{fig:Kirby1} and \ref{fig:Kirby2}.
Then, regarding Figures~\ref{fig:Kirby1} and \ref{fig:Kirby2} as diffeomorphisms of the exteriors of the framed knots, respectively, one obtains the 3-manifold in Figure~\ref{fig:Kirby3}.

\begin{figure}[h]
 \centering
 \includegraphics[width=\textwidth]{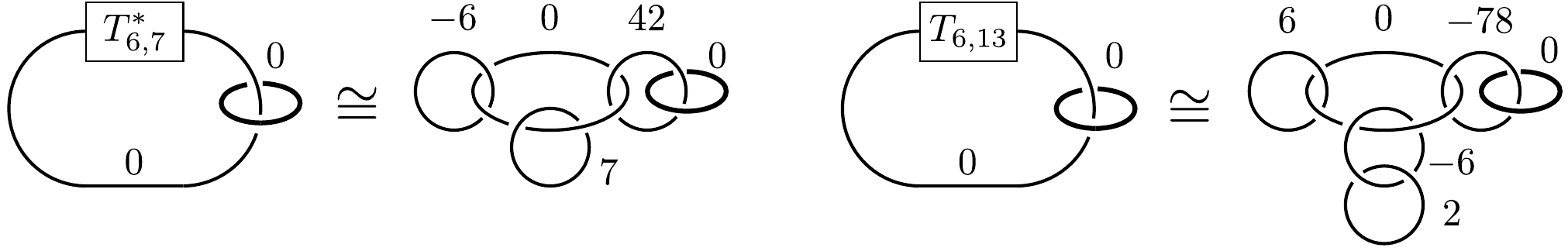}
 \caption{Diffeomorphisms between 3-manifolds with framed knots. Note that the thick components are not used for surgery.}
 \label{fig:Kirby1}
\end{figure}

In Figure~\ref{fig:Kirby1}, the diffeomorphisms between 3-manifolds with framed knots are shown by Kirby calculus including a Rolfsen twist (or the slam-dunk)

\begin{figure}[h]
 \centering
 \includegraphics[width=\textwidth]{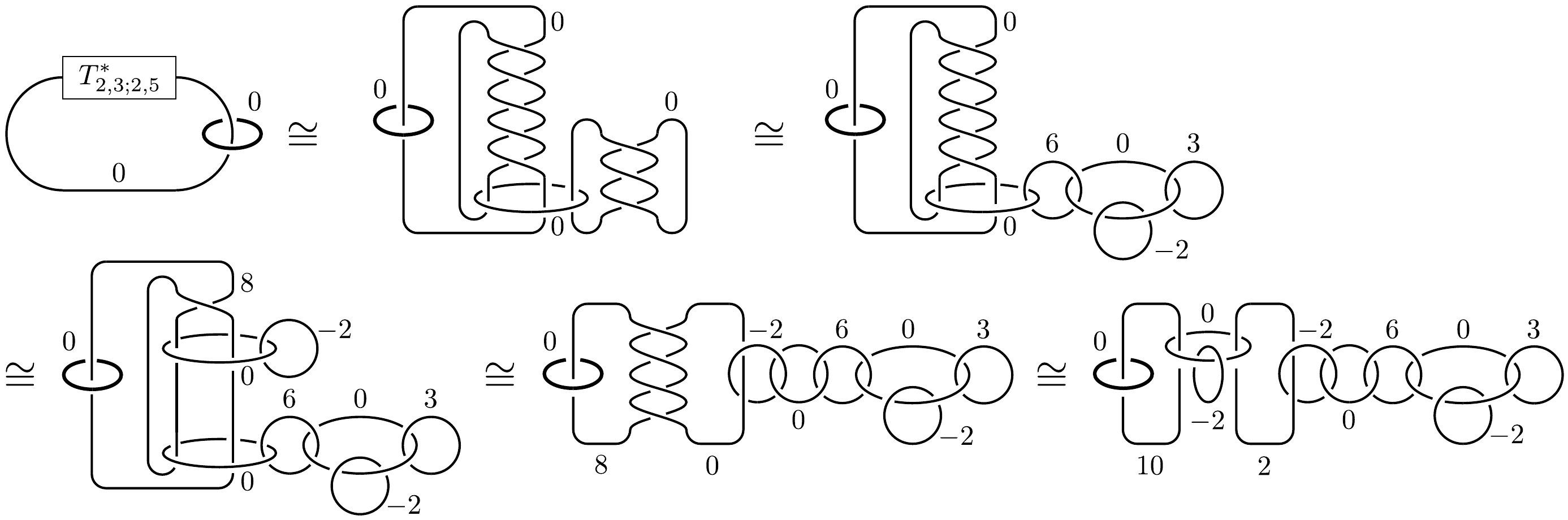}
 \caption{Diffeomorphisms of 3-manifolds with framed knots.}
 \label{fig:Kirby2}
\end{figure}

In Figure~\ref{fig:Kirby2}, the first diffeomorphism follows from the definition of the $(2,5)$-cable of $T_{2,3}$.
The fourth diffeomorphism is obtained by sliding the 0-framed unknot at the bottom to the one at the top.
The rest of the diffeomorphisms are shown by standard Kirby calculus.

\bibliographystyle{hplain}
\bibliography{tex}

\end{document}